\documentclass[11pt]{article}

\usepackage[utf8]{inputenc}
\usepackage[T2A]{fontenc}
\usepackage[russian,english]{babel}
\usepackage{geometry}
\usepackage[normalem]{ulem} 
\usepackage{amsmath}
\usepackage{amssymb}
\usepackage{amsthm}

\usepackage{tabularx} 
\usepackage{graphicx}
\usepackage{tikz}
\usetikzlibrary{patterns}
\usepackage{pgffor}

\usepackage{float}

\usepackage{authblk} 

\usepackage[color, arrow, matrix, curve]{xy}

\newcommand{\dx}[1]{\; \mathrm{d} #1}

\newcommand{\sd}{\: : \:}

\newcommand{\defi}{\: \mathrel{\mathop{\raisebox{1pt}{\scriptsize$:$}}}= \:}

\newcommand{\dist}{\mathrm{dist}}
\newcommand{\tr}{\mathrm{tr}}
\newcommand{\Sym}{\mathrm{Sym}_{0}}
\newcommand{\id}{\mathrm{Id}}
\newcommand{\rhomin}{\rho_\mathrm{min}^\sigma}

\newcommand{\N}{\mathcal{N}}
\newcommand{\E}{\mathcal{E}}
\newcommand{\Eex}{\mathcal{E}_{\eta,\xi}}
\renewcommand{\C}{\mathcal{C}} 
\newcommand{\R}{\mathcal{R}}
\newcommand{\M}{\mathcal{M}}
\newcommand{\A}{\mathcal{A}}

\newcommand{\Qex}{Q_{\eta,\xi}}

\newcommand{\nn}{\mathbf{n}}
\newcommand{\mm}{\mathbf{m}}
\newcommand{\xx}{\mathbf{x}}
\newcommand{\ee}{\mathbf{e}}

\newtheorem{theorem}{Theorem}[section]
\newtheorem{definition}[theorem]{Definition}
\newtheorem{proposition}[theorem]{Proposition}
\newtheorem{lemma}[theorem]{Lemma}
\newtheorem{remark}[theorem]{Remark}
\newtheorem{corollary}[theorem]{Corollary}

\newtheoremstyle{name_and_space}{11pt}{20pt}{\itshape}{}{\bfseries}{}{.5em}{#1 #2 #3}
\theoremstyle{name_and_space}

\newif\ifARMA
\ARMAfalse

\begin{document}

\title{The saturn ring effect in nematic liquid crystals with external field: effective energy and hysteresis}

\date\today
\author[1]{François Alouges}
\author[1]{Antonin Chambolle
}
\author[1]{Dominik Stantejsky }
\affil[1]{Centre de Mathématiques Appliquées, UMR CNRS 7641, École Polytechnique, IP-Paris 91128 Palaiseau Cedex, France}

\parskip 6pt

\maketitle

\begin{abstract}
In this work we consider the Landau-de Gennes model for  liquid crystals with an external electromagnetic field to model the occurrence of the saturn ring effect under the assumption of rotational equivariance. After a rescaling of the energy, a variational limit is derived. Our analysis relies on precise estimates around the singularities and the study of a radial auxiliary problem in regions, where a continuous director field exists. Studying the limit problem, we explain the transition between the dipole and saturn ring configuration and the occurence of a hysteresis phenomenon, giving a rigorous explanation of what was conjectured previously by [H. Stark, Eur. Phys. J. B 10, 311–321 (1999)]. \\
\linebreak
\linebreak
\end{abstract}

\section*{Introduction}
\addcontentsline{toc}{section}{Introduction}

Liquid crystals represent a type of matter with properties intermediate between liquids and crystalline solids. They can be thought of as rod like molecules whose positional and orientational order may vary within space, time and parameters such as temperature. For a general and complete introduction, we refer to \cite{Andrienko2018,Handbook2014}. Depending on the alignment of the molecules and its symmetries, liquid crystals are generally divided into nematic, smectic and cholesteric. Due to their unique properties, liquid crystals exhibit remarkable structures and applications, see for example \cite{Kleman2006,Machon2018,Musevic2017}.

From a mathematical point of view, several models have been introduced to study the phenomena arising from liquid crystals \cite{Ball2017a}. Roughly speaking, the Oseen-Frank model describes liquid crystals by a unit vector field $\nn$, that represents the direction of the molecules. A peculiarity is, that in practice we do not distinguish between $\nn$ and $-\nn$, so that $\nn$ should rather take values in a projective space $\mathbb{R}P^2$ to avoid problems with orientability. 
In order to represent local averages of the directions $\nn$ of the molecules, one gets an additional degree of freedom. Models describing the liquid crystal with such a variable include e.g.\ the Ericksen model.
The Landau-de Gennes model goes one step further by using the idea to describe the arrangement of a liquid crystal by a probability distribution $\rho$ on the sphere of directions, taking into account that opposite points have the same probability. Then the first moment vanishes and the (shifted) second moment $Q$ is a symmetric traceless tensor, which is used to model  $\rho$. This allows to incorporate both the Oseen-Frank and Ericksen model into the one of Landau and de Gennes.
A more detailed introduction to the various models and even for more refined generalizations of the Landau-de Gennes model, e.g.\ the Onsager model or Maier-Saupe model, can be found in \cite{Ball2017,Vollmer2015}. For the challenges and a comparison of the mentioned descriptions, see \cite{Ball2014,Ball2010,Bedford2014,Brezis1986,Riviere1995}.
In general, it is difficult to give precise descriptions of minimizers of the energy functionals associated with one of the models explicitly, except in some very special cases such as in \cite{Yu2020} or for the radial hedgehog solution in \cite{Majumdar2010}.

Mathematically speaking, liquid crystal theory shares several techniques and results with other subjects, for example the Ginzburg-Landau model in micromagnetics, \cite{Bethuel1994,Goldman2017,Jerrard1999}. Also parts of the description, such as function spaces \cite{Badal2016} and liftings \cite{Ignat2017,Majumdar2009}, $Q-$tensors \cite{Braides2013,Mottram2014}, the formation of topological singularities \cite{Tang2017} or similar energy functionals \cite{Contreras2018,Sandier1998} are of interest in a more abstract setting.

One interesting pattern one can observe in liquid crystals is the so called "saturn ring" effect. Under certain circumstances the defect structure forming in order to balance a topological charge on the surface of an immersed object in liquid crystals, takes the form of a ring around the particle, see \cite{Alama2020,Alama2015,Gu2000,Musevic2017}. Also more exotic structures such as knots are possible, we refer to \cite{Musevic2017} for an overview.
In addition, an electromagnetic field can be used to manipulate the occurrence of a saturn ring. While this is known in physics for several years \cite{Amoddeo2011,Fukuda2004,Fukuda2006,Fukuda2004a,Wang2014}, there are only few mathematical results \cite{Alama2017}. 
Starting from the Landau-de Gennes model, we have to find a minimizer of the energy
$$ \Eex(Q) = \int_\Omega \frac{1}{2}|\nabla Q|^2 + \frac{1}{\xi^2} f(Q) + \frac{1}{\eta^2} g(Q) \dx x $$
under suitable anchoring boundary conditions.
Here $\Omega$ is the region filled with the liquid crystal, in our case we consider $\Omega=\mathbb{R}^3\setminus B_1(0)$. The Dirichlet term models elastic forces, while $f$ incorporates bulk forces. The parameter $\xi$ describes the ratio between elastic and bulk forces. We are going to consider the limit of a vanishing elastic constant, i.e.\ $\xi$ will converge to zero. The effect of an external electromagnetic field is described by the function $g$, with the parameter $\eta$ coupling the field to elastic and bulk forces. We are also going to take the limit $\eta\rightarrow 0$. To complete our model, we impose a strong anchoring boundary condition on $\partial\Omega$ that corresponds to a radial director field $\nn=\ee_r$. With $\xi$ and $\eta$ converging to zero, we can consider different regimes regarding the relative speed of convergence of both parameters.
\begin{enumerate}
\item The case of strong fields $\eta|\ln(\xi)|\ll 1$, where we expect to observe a saturn ring was treated in \cite{Alama2017}.
\item The case $\eta|\ln(\xi)|\sim 1$, where the transition between dipole and saturn ring takes place is precisely the purpose of this paper.
\end{enumerate}
Our work is organized as follows. In the first section we define the different parts of the energy carefully, establish fundamental properties and discuss their effects in the minimizing process.

The second section contains the rescaling and states our main theorem, a $\Gamma-$convergence result in a sense that will be precised later. We will prove, that in the limit $\eta,\xi\rightarrow 0$ in our regime and under the assumption of rotational equivariance, the model reduces to a simple energy stated on the surface of the sphere $\mathbb{S}^2=\partial\Omega$. More precisely
$$ \E_0(F) = \sqrt[4]{24}s_*\int_F (1-\cos(\theta))\dx\omega +\sqrt[4]{24}s_*\int_{F^c} (1+\cos(\theta))\dx\omega + \frac{\pi}{2}s_*^2\beta |D\chi_F|(\mathbb{S}^2)\, , $$
where $s_*>0$ is a parameter depending on $f$ and $F\subset\mathbb{S}^2$ is a set of finite perimeter that can be seen as the projection of the region, in which a lifting of $Q$ from $\mathbb{R}P^2$ to $\mathbb{S}^2$ exists and the orientation at infinity agrees with the outward normal of $\partial B_1$. In the same spirit, $F^c$ stands for the region, where the lifting has the opposite orientation. In the above expression $\theta$ stands for the angle between a point $\omega$ on the sphere and $\ee_3$. We see the latter perimeter term as representation of a defect line. It tells us that switching from one orientation to the other comes with a cost, depending on the balance between the forces (modelled by $\beta$), the length of $\nn$ which is related to $s_*$ and the length of the defect line. This is the result we are going to prove in the next two sections. 

Section three is divided into three parts: We first show that the energy bound implies the existence of only a finite number of singularities if we are at some distance from the $\ee_3-$axis. The main idea will be to replace our functions $\Qex$ by the minimizers of  approximate problems and then use the higher regularity to derive a lower bound on the energy cost of a singularity. The energy bound then implies that in fact only finitely many singularities can occur. Next, we provide asymptotically exact lower bounds for the energy near those singularities. Then, the radial auxiliary problem is introduced. Given a ray from the surface $\partial\Omega$ to infinity such that $\Qex$ is close to being uniaxial, we can explicitly calculate the energy necessary to turn along the ray from our boundary conditions to the preferred configuration parallel to the external field in $\pm\ee_3-$direction. Combining the results, we are able to prove the lower bound part of the main theorem. 

The construction for the recovery sequence is made in section four. We use our knowledge about the interplay of the three parts of the energy to define approximate regions for the singularities and the uniaxial part. Here we profit from the exact formula of the optimal profile from the radial auxiliary problem. 

The remaining section deals with the limit energy. We calculate the minimizers (depending on $\beta$) and compare their energy with that of a dipole and a saturn ring at the same $\beta-$value. We find that by varying $\beta$ a hysteresis phenomenon occurs. Our findings rigorously explain physical experiments and known numerical simulations \cite{Lavrentovich2001,Stark1999}.

\section{Definitions and preliminaries}

We start this section by giving precise definitions for the functions and quantities mentioned in the introduction, namely the bulk and magnetic terms that involve the functions $f$ and $g$.

\begin{definition} We denote by $\Sym$ the space of symmetric matrices with vanishing trace
$$ \Sym\defi \{ Q\in\mathbb{R}^{3\times 3}\sd Q^\top = Q\:,\: \tr(Q)=0 \}\, , $$
equipped with the norm $|Q|=\sqrt{\tr(Q^2)}$.
Furthermore, for $a,b,c >0$ we define
\begin{equation} \label{def:f}
f(Q) = C - \frac{a}{2}\tr(Q^2) - \frac{b}{3} \tr(Q^3) + \frac{c}{4}(\tr(Q^2))^2
\end{equation} and 
\begin{equation} \label{def:g}
g(Q) = \begin{cases} \sqrt{\frac{2}{3}} - \frac{Q_{33}}{|Q|} & Q\in\Sym\setminus\{0\} \\
0 & Q=0 \end{cases} \, .
\end{equation} 
\end{definition}

As we stated in the introduction, the definition of $\Sym$ is motivated by the second order moment of a probability distribution $\rho$ on a sphere. The symmetry between $\pm\nn$ reads $\rho(\nn)=\rho(-\nn)$ for all $\nn\in\mathbb{S}^2$, i.e.\ the expectation value of $\nn$ vanishes, $\int_{\mathbb{S}^2} \nn \dx\rho = 0$. The second moment $\int_{\mathbb{S}^2} \nn\otimes\nn \dx\rho$ is symmetric and has trace $1$. From this we subtract the second moment of a uniform distribution on $\mathbb{S}^2$, i.e.\ $\overline{\rho} = \frac{1}{4\pi}$ to get the symmetric and traceless tensor $Q$.

The specific form of the function $f$ comes from the requirement of being invariant under rotations. Indeed, assuming a polynomial function $f$ and demanding frame indifference for the bulk energy (and of course for the elastic energy) we find that $f$ has to satisfy $f(Q)=f(R^\top QR)$ for all $R\in O(3)$. This implies that $f$ is the linear combination of $\tr(Q^2)$, $\tr(Q^3)$, $ (\tr(Q)^2)^2$, $ \tr(Q^2)\tr(Q^3)$, $ \tr(Q^2)^2$, $ \tr(Q^3)^2$, etc (see \cite[Lemma 3]{Ball2017}). It is convenient to consider only the first three terms although one could in principle add more. Another possible generalization is to consider $f$ to be dependent on the temperature. In the simplest case one writes $a(T-T_{*})$ instead of $a$ for a reference temperature $T_{*}$ \cite{Mottram2014}. However, we are not going to include this into our work. We see in the next Proposition that for a certain choice of the constant $C$, $f$ vanishes on so-called uniaxial $Q-$tensors (the set $\N$ in Proposition \ref{prop:prop_f} below). This is the main property of $f$ one should keep in mind during our analysis.

The definition of $g$ is inspired by the classical approach of introducing a quadratic term penalizing $\nn$ not being parallel to the external electric or magnetic field $\mathbf{H}$, e.g. $(\nn\cdot\mathbf{H})^2$. We choose a field in $\ee_3-$direction and separate the field strength from the direction, i.e.\ we write $\mathbf{H}=h\ee_3$. In addition, $g$ should be independent of $|Q|$, since we only want $\nn$ to be parallel to $\ee_3$ without preferring any size $|Q|$. The constant $\sqrt{2/3}$ is chosen such that $g$ is non-negative as we will see in Proposition \ref{prop:prop_g}.

\begin{proposition}[Properties of $f$] \label{prop:prop_f}
There exists a constant $C$ such that $f$ given by (\ref{def:f}) satisfies
\begin{enumerate}
\item $f(Q)\geq 0$ for all $Q\in\Sym$ and $\min_{Q\in\Sym} f(Q)=0$. Let $\N\defi f^{-1}(0)$. We have
$$ \N = \left\{ s_*\left(\nn\otimes\nn - \frac{1}{3}\id\right)\sd \nn\in\mathbb{S}^2 \right\}\, , $$
where $\mathbb{S}^2\subset\mathbb{R}^3$ is the unit sphere and $s_* = \frac{1}{4c}\left(b + \sqrt{b^2+24ac}\right)$. Moreover $\N$ is a smooth, compact, connected manifold without boundary diffeomorphic to $\mathbb{R}P^2$.
\item For all $Q\in \Sym$ with $|Q|>\sqrt{\frac{2}{3}}s_*$, there holds:
$$ f(Q) > f\left(\sqrt{\frac{2}{3}}s_*\frac{Q}{|Q|}\right)\, . $$
\item Furthermore, there exist constants $\delta_0,\gamma_1>0$ such that if $\dist(Q,\N)\leq \delta_0$ for $Q\in \Sym$, then 
$$ f(Q) \geq \gamma_1\: \dist(Q,\N)^2\, .   $$
\item There exist constants $C,C_f>0$ such that for $\dist(Q,\N)>C_f$ it holds
$$ f(Q) \geq C\: |Q|^4 \, . $$
\end{enumerate}
\end{proposition}

\begin{proof}
A proof of the first statement can be found in \cite[Proposition 15]{Majumdar2009}.
For the second result, we refer to \cite[Lemma 3.5]{Canevari2013}.
The third assertion is proved in \cite[Lemma 2.2.4 ($F_2$)]{Canevari2015}.
The last claim follows as in \cite[Lemma 2.4]{Canevari2015a}.
\end{proof}

The last three statements are of technical nature. The second property is only used to establish $L^\infty-$bounds in Remark \ref{rem:Linfty_Qex} and Proposition \ref{prop:bounds_Q_eps}. The estimate in 3. simply states that one can think of $f$ as being quadratic close to its minimum which is attained on $\N$, while 4. tells us that $f$ is of order $4$ far from $\N$.
The first statement gives an interesting connection between $f$ and the space $\Sym$. In fact, $\N$ plays an important role in our analysis as it will allow us to identify $Q$ and $\pm \nn$ and thus give a intuitive meaning to $Q$. This is formalized in the next Proposition.

\begin{proposition}[Structure of $\Sym$] \label{prop:prop_sym0}
\begin{enumerate}
\item For all $Q\in \Sym$ there exist $s\in [0,\infty)$ and $r\in [0,1]$ such that
\begin{equation} \label{prop:prop_sym0_eq}
Q = s\left( \left(\nn\otimes\nn - \frac{1}{3}\id\right) + r\left(\mm\otimes\mm - \frac{1}{3}\id\right)\right)\, ,
\end{equation}
where $\nn,\mm$ are normalized, orthogonal eigenvectors of $Q$. The values $s$ and $r$ are continuous functions of $Q$.
\item Let $\C = \{ Q\in\Sym\sd \lambda_1(Q)=\lambda_2(Q) \}$, where we denoted by $\lambda_1,\lambda_2$ the two leading eigenvalues of $Q$. Then 
$$ \C =\{ Q\in\Sym\setminus\{0\}\sd r(Q)=1 \}\cup\{0\} \quad\text{and}\quad \C\setminus\{0\} \cong \mathbb{R}P^2\times\mathbb{R}\, . $$
\item There exists a continuous function $\R:\Sym\setminus\C\rightarrow\N$ such that $\R(Q)=Q$ for all $Q\in\N$. In particular, $\Sym\setminus\C\simeq \N$. The map $\R$ can be chosen to be the nearest point projection onto $\N$. In this case, for all $Q\in\Sym\setminus\C$ decomposed as in (\ref{prop:prop_sym0_eq}), $\R$ is given by $\R(Q)=s_*(\nn\otimes\nn-\frac{1}{3}\id)$ .
\end{enumerate}
\end{proposition}

\begin{proof}
The first part follows from \cite[Lemma 1.3.1]{Canevari2015} for $s = 2\lambda_1+\lambda_2$ and $r=(\lambda_1+2\lambda_2)/s$, where $\lambda_1\geq\lambda_2$ are the two leading eigenvalues of $Q$.
The second part is a consequence of the definition of $s,r$ in terms of the eigenvalues and \cite[Lemma 1.3.5]{Canevari2015}.
The last part is a reformulation of Lemma 1.3.6 and Lemma 1.3.7 in \cite{Canevari2015}, together with Lemma 2.2.2.
\end{proof}

The decomposition (\ref{prop:prop_sym0_eq}) provides us with a very useful tool to perform calculations, for example in Lemma \ref{lem:continuity_g} or Proposition \ref{prop:phi_close_to_N}.   In the second statement we introduce $\C$, a set that can be thought of as cone over $\mathbb{R}P^2$. It contains exactly the biaxial $Q-$tensors. If a $Q-$tensor is not biaxial, there exists a retraction onto $\N$ which coincides with the nearest point projection and is given by the element of $\N$ corresponding to the dominating eigenvector of $Q$.

We finish this section with two Propositions about $g$. More precisely, we prove that $g$ has the properties that we claimed above and show that $g$ takes an even simpler form on $\N$. Finally we show that $g$ is Lipschitz continuous in a neighbourhood of $\N$. All calculations are straightforward.

\begin{proposition}[Properties of $g$] \label{prop:prop_g}
Let $g$ be given as in (\ref{def:g}). 
\begin{enumerate}
\item $g(Q)\geq 0$ for all $Q\in\Sym$ with equality of and only if $Q=t(\ee_3\otimes\ee_3 - \frac{1}{3}\id)$ for some $t\geq 0$.
\item If $Q\in\N$ is given by $Q=s_*(\nn\otimes\nn-  \frac{1}{3}\id)$ with $\nn\in\mathbb{S}^2$, then
$$ g(Q) = \sqrt{\frac{3}{2}}\left( 1 - \nn_3^2\right)\, . $$
\end{enumerate}
\end{proposition}

\begin{proof}
Minimizing $g$ under the tracelessness constraint, we get the necessary conditions
$$ -\frac{1}{|Q|}+\frac{Q_{33}^2}{|Q|^3}-\lambda = 0\, , \quad \frac{Q_{33}Q_{jj}}{|Q|^3}-\lambda = 0 \text{ for } j=1,2\, , \quad \frac{Q_{33}Q_{ij}}{|Q|^3} = 0 \text{ for }i\neq j  $$
for a Lagrange multiplier $\lambda$. For $Q=0$ the claim is clear by definition. So let $Q\in\Sym\setminus\{0\}$. If $Q_{33}=0$ we get a contradiction. Hence we can assume $Q_{33}\neq 0$. Then the third equation from above implies $Q_{ij}=0$ for $i\neq j$ and the second $Q_{11}=Q_{22}$. By $\tr(Q)=0$, we have $Q_{33}=-2Q_{11}$. Then the first equation reads $0=\frac{3}{2}Q_{33}^2 - |Q|^2$, i.e.\ $Q_{33}=\sqrt{2/3}|Q|$. Inserting this into $g$ we get $\min_{\Sym} g = 0$. Our conditions also imply the claimed representation $Q=t(\ee_3\otimes\ee_3 - \frac{1}{3}\id)$. Reversely, it is obvious that $g=0$ for such $Q$.

For the second claim, it is straightforward to check that for $Q=s_*(\nn\otimes\nn-  \frac{1}{3}\id)\in\N$ we have $|Q|^2=\frac{2}{3}s_*^2$. Thus
\begin{align*}
g(Q) &= \sqrt{\frac{2}{3}} - \frac{s_*(\nn_3^2-\frac{1}{3})}{\sqrt{\frac{2}{3}}s_*} = \sqrt{\frac{2}{3}} + \frac{1}{3}\sqrt{\frac{3}{2}} - \sqrt{\frac{3}{2}}\nn_3^2 = \sqrt{\frac{3}{2}}\left( 1 - \nn_3^2\right)\, .
\end{align*}
\end{proof}

\begin{lemma} \label{lem:continuity_g}
There exist constants $\delta_1,C>0$ such that if $Q\in\Sym$ with $\dist(Q,\N)\leq \delta$ for $0<\delta<\delta_1$, then
\begin{align}
|g(Q) - g(\R\circ Q)| \leq C\: \dist(Q,\N)\, .
\end{align}
\end{lemma}

\begin{proof}
We use Proposition \ref{prop:prop_sym0} to write
$$ Q = s\left( \left( \nn\otimes \nn - \frac{1}{3}\id\right)  + r \left(\mm\otimes \mm - \frac{1}{3}\id \right)\right)\, , $$
with  $s>0$, $0 \leq r <1$ and $\nn,\mm$ orthonormal eigenvectors of $Q$. Thus $\R\circ Q = s_* \left( \nn\otimes \nn - \frac{1}{3}\id\right)$ and from Proposition \ref{prop:prop_g}, we infer that $g(\R\circ Q) = \sqrt{\frac{2}{3}}(1-n_3^2)$. In order to calculate $g(Q)$, we note that
\begin{align*}
|Q|^2 &= s^2 \left| \nn\otimes \nn - \frac{1}{3}\id \right|^2 + (sr)^2 \left| \mm\otimes \mm - \frac{1}{3}\id \right|^2 + 2 s^2r \left(\nn\otimes \nn - \frac{1}{3}\id\right):\left(\mm\otimes \mm - \frac{1}{3}\id\right) \\
&= \frac{2}{3} s^2 \left( r^2 - r +1 \right)\, .
\end{align*}
This implies
\begin{align*}
|g(Q) - g(\R\circ Q)| &= \left| \sqrt{\frac{2}{3}} - \frac{s(n_3^2 - \frac{1}{3}) + sr(m_3^2-\frac{1}{3})}{\sqrt{\frac{2}{3}} s \sqrt{1-r+r^2}} - \sqrt{\frac{2}{3}} + \frac{s_*(n_3^2-\frac{1}{3})}{s_* \sqrt{\frac{2}{3}}} \right| \\
&\leq \frac{n_3^2-\frac{1}{3}}{\sqrt{\frac{2}{3}}} \left(\frac{1}{\sqrt{1-r+r^2}} - 1\right) + \frac{m_3^2-\frac{1}{3}}{\sqrt{\frac{2}{3}}} \frac{r}{\sqrt{1-r+r^2}}\, .
\end{align*}
Note, that the Taylor expansion at $r=0$ is given by $\displaystyle{\frac{1}{\sqrt{1-r+r^2}} - 1} = \frac{r}{2} + \mathcal{O}(r^2)$ and $\displaystyle{\frac{r}{\sqrt{1-r+r^2}} = r + \mathcal{O}(r^2)}$. Hence
\begin{equation}\label{lem:continuity_g_eq}
|g(Q) - g(\R\circ Q)| \leq \frac{3}{2} r + \mathcal{O}(r^2)\, .
\end{equation}
Finally, we can estimate
\begin{align*}
\dist^2(Q,\N) &= |Q - \R(Q)|^2 = \left| (s-s_*)(\nn\otimes\nn - \frac{1}{3}\id) + sr (\mm\otimes\mm - \frac{1}{3}\id) \right|^2 \\
&= \frac{2}{3}|s-s_*|^2 + \frac{2}{3}|sr|^2 - \frac{2}{3}sr(s-s_*) \\
&\geq \frac{1}{3}|s-s_*|^2 + \frac{1}{3}|sr|^2\, .
\end{align*}
This implies $|s-s_*|\leq \sqrt{3}\:\dist(Q,\N)$ and $\displaystyle{|r|\leq \frac{\sqrt{3}\: \dist(Q,\N)}{|s|}}$. We define $\displaystyle{\delta_1 = \frac{1}{2\sqrt{3}}s_*}$ and together with (\ref{lem:continuity_g_eq}) we get
$$ |g(Q) - g(\R\circ Q)| \leq C r \leq \frac{\sqrt{3}\dist(Q,\N)}{|s|} \leq C\frac{2\sqrt{3}}{s_*} \dist(Q,\N)\, . $$
\end{proof}

\section{Scaling and statement of result}

Starting from the one constant approximation of the Landau-de Gennes energy in $\Omega_{r_0} = \mathbb{R}^3 \setminus \overline{B_{r_0}(0)}$ we find the energy 
\begin{align} \label{eq:scaling_full_E}
\E(Q) = \int_{\Omega_{r_0}} \frac{L}{2}|\nabla Q|^2 + f(Q) +  h^2 g(Q) \dx x
\end{align} for parameters $L,h,r_0>0$. As we have seen before, $h$ can be interpreted as a field strength, $r_0$ as the particle radius and $L$ as the elastic constant. In order to be able to work on a fixed domain, we apply the rescaling $\Omega\defi \frac{1}{r_0}\Omega_{r_0}$ and $\tilde{x} =  x/r_0$. We introduce the new function $\widetilde{Q}(\tilde{x}) = Q(r_0 \tilde{x})=Q(x)$ and $\widetilde{\nabla}=\nabla_{\tilde{x}}= \frac{1}{r_0}\nabla_x$. Then 
\begin{align*}
\E(Q) &= \int_\Omega \frac{L r_0^3}{2 r_0^2} |\nabla \widetilde{Q}|^2 + r_0^3 f(\widetilde{Q}) + h^2 r_0^3 g(\widetilde{Q}) \dx\tilde{x} \\
&= \int_\Omega \frac{L r_0}{2} |\widetilde{\nabla} \widetilde{Q}|^2 + r_0^3 f(\widetilde{Q}) + (h r_0^\frac{3}{2})^2\: g(\widetilde{Q}) \dx\tilde{x}\, .
\end{align*} Dividing by $L r_0$, we can define
\begin{equation} \label{def:Eex}
\begin{split}
\Eex(\widetilde{Q}) = \int_\Omega \frac{1}{2}|\widetilde{\nabla}\widetilde{Q}|^2 + \frac{1}{\xi^2} f(\widetilde{Q}) + \frac{1}{\eta^2} g(\widetilde{Q}) \dx\tilde{x}\, ,
\end{split}
\end{equation} where we introduced the new parameters $\xi = \sqrt{L}/r_0$ and $\eta = \sqrt{L}/(r_0 h)$. This is the energy that was announced in the introduction. The natural space for this energy to be well defined is $H^1(\Omega,\Sym)$. Minimizing the first term would lead to a harmonic map, the second term prefers to be uniaxial with a certain norm, while the third term takes its minimum when the director is aligned parallel to $\ee_3$. So the constant uniaxial map $s_*(\ee_3\otimes\ee_3 - \frac{1}{3}\id)$ would be a minimizer of our energy. However, this will violate the boundary conditions we are going to impose, namely we want $\Qex\in H^1(\Omega,\Sym)$ to satisfy
\begin{equation} \label{eq:bc}
\Qex = Q_b \quad \text{on }\:\:\mathbb{S}^2\, ,
\end{equation}
where $Q_b(x) = s_*\left(\xx\otimes\xx - \frac{1}{3}\id\right)$.
So what we expect instead is a map that is close to $s_*(\ee_3\otimes\ee_3 - \frac{1}{3}\id)$ everywhere, except for a transition zone near the boundary. In this boundary layer, which will turn out to be of thickness $\eta$, we will find tubes of cross sectional area $\xi^2$ containing the regions where $\Qex$ is biaxial.

Since the problem is equivariant with respect to rotations around the $\ee_3-$axis, it is natural to consider only rotationally equivariant maps. We say that a map $Q$ is \emph{rotationally equivariant} if $Q$ is equivariant with respect to rotations around the $\ee_3$-axis. In other words, using cylindrical coordinates, one has
\begin{align*}
Q(\rho,\varphi,z) = R_\varphi^\top Q(\rho,0,z) R_\varphi\, , \quad \text{ where }\quad R_\varphi = \begin{pmatrix}
\cos\varphi & -\sin\varphi & 0 \\
\sin\varphi & \cos\varphi & 0 \\
0 & 0 & 1
\end{pmatrix}\, .
\end{align*} 
For uniaxial maps $Q=s_*(\nn\otimes\nn-\frac{1}{3}\id)$ this is equivalent to the usual notion of equivariance for vectors $\nn(R_\varphi\xx) = R_\varphi^\top\nn(\xx)$.
We define the set of admissible functions $\A$ to be the set of rotationally equivariant functions  $\Qex\in H^1(\Omega,\Sym)$ satisfying the boundary condition (\ref{eq:bc}).
This motivates the definition for $Q\in H^1(\Omega,\mathbb{R}^{3\times 3})$
$$ \Eex^\A(Q) = \begin{cases} \Eex(Q) & \text{if }Q\in \A\, , \\ \infty & \text{otherwise.} \end{cases} $$
We strongly believe that minimizers of $\Eex$ are also rotationally equivariant, although this does not follow from our work and remains an open issue. We will remove the hypothesis of rotational equivariance in a work in preparation.

The following theorem is the main result of the paper.

\begin{theorem}\label{thm:main}
Suppose that
\begin{equation} \label{eq:beta}
\eta|\ln(\xi)|\rightarrow \beta\in (0,\infty) \quad \text{ as } \eta\rightarrow 0\, .
\end{equation}
Then $\eta\:\Eex^\A\rightarrow\E_0$ in a variational sense, where the limiting energy $\E_0$ for a set $F\subset \mathbb{S}^2$ is given by
\begin{equation} \label{def:lim_energy} 
\E_0(F) = \sqrt[4]{24}s_*\int_F (1 - \cos(\theta)) \dx\omega + \sqrt[4]{24}s_*\int_{F^c} (1 + \cos(\theta)) \dx\omega + \frac{\pi}{2}s_*^2\beta |D\chi_F|(\mathbb{S}^2)\, .
\end{equation}
 More precisely, we have the following statements: 
\begin{enumerate}
\item Compactness: For any sequence $\Qex\in \A$ such that $\eta\: \Eex(\Qex)\leq C$, there exists a measurable set of finite perimeter $F\subset \mathbb{S}^2$ that is invariant under rotations w.r.t. the $\ee_3-$axis, measurable functions $\nn^\eta:\Omega\rightarrow\mathbb{S}^2$ and a set $\omega_\eta\subset\Omega$ with $\lim_{\eta\rightarrow 0}|\omega_\eta|= 0$, $\Omega\setminus\omega_\eta$ simply connected, such that  for all $\sigma>0$ it holds $\nn^\eta\in C^0(\Omega\setminus(Z_\sigma\cup\omega_\eta),\mathbb{S}^2)$ and  
\begin{equation} \label{thm:main:cptness}
\lim_{\eta\rightarrow 0}\Big\Vert s_*\Big(\nn^\eta\otimes\nn^\eta-\frac{1}{3}\id\Big)-\Qex\Big\Vert_{L^2(\Omega\setminus Z_\sigma)}=0\, , \quad \chi_{F_\eta}\rightarrow \chi_F \text{ pointwise,}
\end{equation} where $Z_\sigma=\{x\in\mathbb{R}^3\sd x_1^2+x_2^2\leq \sigma^2\}$ and $F_\eta=\{ x\in\partial\Omega\sd \nn^\eta(x)\cdot\nu(x)=-1 \}$.
\item $\Gamma-$liminf: For any sequence $\Qex\in \A$ and any measurable set of finite perimeter $F\subset \mathbb{S}^2$, measurable functions $\nn^\eta:\Omega\rightarrow\mathbb{S}^2$ and a measurable set $\omega_\eta\subset\Omega$ that satisfy  $\lim_{\eta\rightarrow 0}|\omega_\eta|= 0$, $\Omega\setminus\omega_\eta$ simply connected with $\nn^\eta\in C^0(\Omega\setminus(Z_\sigma\cup\omega_\eta),\mathbb{S}^2)$ and (\ref{thm:main:cptness}) hold for all $\sigma>0$, we have
\begin{equation} \label{eq:lower_bound}
\liminf_{\eta\rightarrow 0} \eta\:\Eex(\Qex) \geq \E_0(F)\, .
\end{equation}
\item $\Gamma-$limsup: For any measurable set of finite perimeter $F\subset \mathbb{S}^2$ that is invariant under rotations w.r.t. the $\ee_3-$axis there exists a sequence $\Qex\in \A$ with $\Vert \Qex\Vert_{L^\infty}\leq \sqrt{\frac{2}{3}}s_*$ and measurable functions $\nn^\eta:\Omega\rightarrow\mathbb{S}^2$ with $\nn^\eta\in C^0(\Omega\setminus\omega_\eta,\mathbb{S}^2)$, $\lim_{\eta\rightarrow 0}|\omega_\eta|= 0$, $\Omega\setminus\omega_\eta$ simply connected, such that (\ref{thm:main:cptness}) holds for all $\sigma>0$ and
\begin{equation} \label{eq:upper_bound}
\limsup_{\eta\rightarrow 0} \eta\:\Eex(\Qex) \leq \E_0(F)\, .
\end{equation}
\end{enumerate}
\end{theorem}

\begin{remark}\label{rem:Linfty_Qex}
\begin{enumerate}
\item In view of (\ref{eq:beta}) we can replace the bound $\eta\: \Eex(\Qex)\leq C$, by 
\begin{equation} \label{eq:energy_bound}
\Eex(\Qex)\leq C\: \left(1+|\ln(\xi)|\right)\, .
\end{equation}
\item The convergence we show is not a $\Gamma-$convergence in the classical sense since the limit functional is defined on a different functions space.
\end{enumerate}
\end{remark}

\section{Lower bound}

In this section we prove the lower bound of Theorem \ref{thm:main}. Our strategy to obtain the lower bound is the following: First, we approximate the sequence $\Qex$ by a more regular one named $Q_\epsilon$. We use $\epsilon:=\xi$ to  meet the notation in \cite{Alama2017,Canevari2013,Canevari2015} and let out $\eta$ in our notation since $\eta$ and $\xi$ are related via (\ref{eq:beta}), i.e.\ $\eta\sim \frac{\beta}{|\ln(\epsilon)|}$. We also write $\E_\epsilon$ instead of $\Eex$. We find that away from the $\ee_3$-axis the sequence $Q_\epsilon$ has only finitely many singularities  in the neighbourhood of which $Q_\epsilon$ is far from $\N$. Then we can estimate the energy of $Q_\epsilon$ nearby these points from below by balancing $|\nabla Q_\epsilon|^2$ and $f(Q_\epsilon)$. In the region where $Q_\epsilon$ is close to $\N$, we will use the optimal radial profile found in \cite{Alama2017} by balancing $|\nabla Q_\epsilon|^2$ and $g(Q_\epsilon)$.

\subsection{Preliminaries}

The construction of the approximation $Q_\epsilon$ of $\Qex$ follows several steps. First, we are going to show that $\Qex$ can be approximated by another function $\widetilde{\Qex}$ which verifies an additional $L^\infty-$bound.

\begin{proposition} \label{prop:Linfty_bound_Qex}
Let $\Qex\in H^1(\Omega,\Sym)$ such that (\ref{eq:energy_bound}) holds. Then there exists $\widetilde{\Qex}\in H^1(\Omega,\Sym)$ which decreases the energy $\Eex$, verifies
\begin{equation} \label{eq:L_infty_bound}
\Vert \widetilde{\Qex} \Vert_{L^\infty(\Omega)} \leq \sqrt{\frac{2}{3}}s_*
\end{equation} and $\widetilde{\Qex}-\Qex\rightarrow 0$ in $L^2$ as $\eta,\xi\rightarrow 0$.
\end{proposition}

\begin{proof}
We can define $\widetilde{\Qex}$ as
$$ \widetilde{\Qex} \defi \begin{cases} \sqrt{\frac{2}{3}}s_*\frac{\Qex}{|\Qex|} & \text{if }|\Qex|>\sqrt{\frac{2}{3}}s_*\, , \\ \Qex & \text{otherwise.} \end{cases} $$
This function is clearly admissible and has lower Dirichlet energy. From Proposition \ref{prop:prop_f} it follows that $f(\widetilde{\Qex})\leq f(\Qex)$ and by the scaling invariance of $g$ we find that $g(\widetilde{\Qex})=g(\Qex)$. Hence $\Eex(\widetilde{\Qex})\leq \Eex(\Qex)$. The $L^\infty-$ bound is obvious. So it remains to show that $\Vert \widetilde{\Qex}-\Qex\Vert_{L^2(\Omega)}$ converges to zero as $\eta,\xi\rightarrow 0$. 
We remark that $\int |\widetilde{\Qex}-\Qex|^2 \leq \int (\max\{|\Qex|-\sqrt{\frac{2}{3}}s_*, 0\})^2$ and decompose $\Omega$ into three sets:
$$\Omega=\{\dist(\Qex,\N)< \delta_0\}\cup\{\dist(\Qex,\N)>  C_f\}\cup\{\delta_0\leq \dist(\Qex,\N)\leq C_f\}\, ,$$
where $\delta_0$ and $C_f$ are the constants from Proposition \ref{prop:prop_f}.
On the first set, we use 3. in Proposition \ref{prop:prop_f} and the energy bound to get
$$ \int_{\{\dist(\Qex,\N)< \delta_0\}} |\widetilde{\Qex} - \Qex|^2 \dx x \leq \frac{1}{\gamma_1}\int_\Omega f(\Qex)\dx x \leq C (1+|\ln\xi|)\xi^2\, ,  $$
which converges to zero as $\xi\rightarrow 0$. 
Similarly, on the second set we use $|\widetilde{\Qex}-\Qex|^2\leq |\Qex|^2\leq C |\Qex|^4$ and 4. in Proposition \ref{prop:prop_f} to get
$$ \int_{\{\dist(\Qex,\N) > C_f\}} |\widetilde{\Qex} - \Qex|^2 \dx x \leq C \int_\Omega f(\Qex)\dx x \leq C (1+|\ln\xi|)\xi^2\, ,  $$
which again vanishes in the limit $\xi\rightarrow 0$.
For the last set, where $\delta_0\leq \dist(\Qex,\N)\leq C_f$, we can estimate the measure using that $f(\Qex)\geq c>0$ on this set
\begin{align*}
|\{x\in\Omega\sd \delta_0 \leq \dist(\Qex,\N)\leq C_f\}| &\leq \frac{1}{c}\int_{\{\delta_0\leq \dist(\Qex,\N)\leq C_f\}} f(\Qex) \dx x \\
&\leq C(1+|\ln\xi|) \:\xi^2\, .
\end{align*}
Then, using that both $|\widetilde{\Qex}|$ and $|\Qex|$ are bounded by $\sqrt{\frac{2}{3}}s_* + C_f$, the integral $ \int |\widetilde{\Qex}-\Qex|^2\dx x $ over the third set converges to zero for $\xi\rightarrow 0$. 
Combining these three results proves the Proposition.
\end{proof}

From now on, we will use the notation with $\epsilon$ replacing $\eta,\xi$, i.e.\ $\widetilde{Q_\epsilon}\defi\widetilde{\Qex}$. 
The next step will be defining the more regular sequence $Q_\epsilon$ replacing $\widetilde{Q_\epsilon}$. In view of the lower bound for the claimed $\Gamma-$limit we still want $Q_\epsilon$ to be rotationally equivariant and that it converges to the same limit as $\widetilde{Q_\epsilon}$, while decreasing the energy.

We thus define the three dimensional approximate energy for $0<\gamma<2$ and $\omega\subset\Omega$
\begin{align*}
E_\epsilon^{3D}(Q,\omega) = \int_\omega \frac{1}{2}|\nabla Q|^2 + \frac{1}{\epsilon^2} f(Q) + \frac{1}{2\epsilon^\gamma}|Q - \widetilde{Q_\epsilon}|^2 \dx x\, .
\end{align*}
Wee seek $Q_\epsilon$ by minimizing $E_\epsilon^{3D}(Q,\Omega)$ among rotationally equivariant fields $Q$.
Because of the equivariance, the problem can de stated as a two dimensional problem. Indeed, calculating $|\partial_\varphi Q|^2$ for a rotationally equivariant map $Q\in H^1(\Omega,\Sym)$, and using the equivariance, we can write $Q(\rho,\varphi,z) = R_\varphi^\top Q(\rho,0,z) R_\varphi$ and thus
\begin{align*}
|\partial_\varphi Q|^2 &= \left| (\partial_\varphi R_\varphi)^\top Q R_\varphi  +  R_\varphi^\top Q (\partial_\varphi R_\varphi)  \right|^2 = |Q|^2 + 6 (Q_{12}^2 - Q_{11}Q_{22})\, .
\end{align*}
This expression does no longer depend on $\varphi$. In order to shorten notation, we introduce the matrix 
\begin{align*}
Q_{2\times 2} \defi \frac{1}{2}\left( \frac{\partial}{\partial Q_{ij}} |\partial_\varphi Q|^2 \right)_{ij} = \begin{pmatrix}
2(Q_{11} - Q_{22}) & 4 Q_{12} & Q_{13} \\
4Q_{21}  &  2(Q_{22}- Q_{11}) & Q_{23} \\
Q_{31} & Q_{32} & 0
\end{pmatrix}\, . 
\end{align*} Note that, $Q_{2\times 2}:Q = \frac{1}{2}|\partial_\varphi Q|^2$. So the whole energy does not depend on $\varphi$ any more and using cylindrical coordinates, it can be rewritten as
\begin{align*}
E_\epsilon^{3D}(Q_\epsilon,\Omega) &= \int_0^{2\pi} E_\epsilon^{2D}(Q_\epsilon,\Omega') \dx\varphi = 2\pi\: E_\epsilon^{2D}(Q_\epsilon,\Omega')  \, ,
\end{align*} where $E_\epsilon^{2D}$ is the two dimensional energy given by
\begin{align*}
E_\epsilon^{2D}(Q,\omega') = \int_{\omega'} \frac{\rho}{2}|\nabla' Q|^2 + \frac{1}{\rho} Q_{2\times 2}:Q + \frac{\rho}{\epsilon^2} f(Q) + \frac{\rho}{2\epsilon^\gamma}|Q - \widetilde{Q_\epsilon}|^2 \dx \rho \dx z\, ,
\end{align*} where $\nabla'=(\partial_\rho,\partial_z)$ denotes the two dimensional gradient and $\omega'\subset\Omega'=\{ (\rho,z)\in\mathbb{R}^2\sd \rho>0\: ,\: \rho^2+z^2> 1 \}$. In order to shorten notation, we are going to write $\frac{1}{2}|\nabla Q|^2$ instead of $\frac{1}{2}|\nabla' Q|^2+\frac{1}{\rho^2}Q_{2\times 2}:Q$ whenever we make no use of this division of the gradient.
Now we define $Q_\epsilon$ to be 
\begin{equation} \label{def:reg_seq_low_energy}
Q_\epsilon \defi \underset{Q\in \A'}{\mathrm{argmin}}\: E_\epsilon^{2D}(Q,\Omega')\, ,
\end{equation} 
where $\A'=\{Q\in H^1(\Omega',\Sym)\sd \eqref{eq:bc} \text{ holds for } \rho^2+z^2=1 \}$.
We eventually extend $Q_\epsilon$ to a map in $H^1(\Omega,\Sym)$ which we will also call $Q_\epsilon$ by defining $Q_\epsilon(\rho,\varphi,z)\defi R_\varphi^\top Q_\epsilon(\rho,z) R_\varphi$.

\begin{remark} \label{rem:def_Q_eps}
\begin{enumerate}
\item Note that $\widetilde{Q_\epsilon}|_{\Omega'}$ is an admissible function in (\ref{def:reg_seq_low_energy}), so that $Q_\epsilon$ does exist.
\item Neglecting the contribution from $g$ for a moment, then $Q_\epsilon$ has lower energy than $\widetilde{Q_\epsilon}$.
\item Thanks to the energy bound in (\ref{eq:energy_bound}) we know that
$$ \Vert Q_\epsilon - \widetilde{Q_\epsilon}\Vert_{L^2(\Omega)}^2 \leq C (|\ln\epsilon| + 1) \epsilon^\gamma \rightarrow 0 \quad \text{ as } \epsilon\rightarrow 0\, , $$ 
i.e.\ the two sequences have the same limit for vanishing $\epsilon$.
\item The minimizer $Q_\epsilon$ solves the two dimensional Euler-Lagrange equation
\begin{equation} \label{rem:eq:elg_of_Q_eps}
-\rho \Delta Q_\epsilon + \frac{1}{\rho} Q_{\epsilon,2\times 2} - \partial_\rho Q_\epsilon + \frac{\rho}{\epsilon^2} Df(Q) + \frac{\rho}{\epsilon^\gamma} (Q_\epsilon - \widetilde{Q_\epsilon})  = \Lambda\: \id\, .
\end{equation}
Note that the equation contains an additional term (RHS) due to the fact that $\Sym$ is a subspace of the space of real matrices, i.e.\ a Lagrange multiplier $\Lambda$ is needed to ensure the tracelessness constraint.
\item The function $Q_\epsilon$ also solves the three dimensional Euler-Lagrange equation
\begin{equation} \label{rem:eq:elg_of_Q_eps_3D}
-\Delta Q_\epsilon +  \frac{1}{\epsilon^2} Df(Q_\epsilon) + \frac{1}{\epsilon^\gamma} (Q_\epsilon - \widetilde{Q_\epsilon})  = \Lambda_{3D}\: \id\, ,
\end{equation} despite the fact that it does not need to be a minimizer of $E_\epsilon^{3D}$. To see this, write
\begin{align*}
\Lambda_{3D}\: \id &= -\Delta Q_\epsilon + \frac{1}{\epsilon^2} Df(Q_\epsilon) + \frac{1}{\epsilon^\gamma}(Q_\epsilon - \widetilde{Q_\epsilon}) \\
&= -\partial_\rho^2 Q_\epsilon - \frac{1}{\rho} \partial_\rho Q_\epsilon - \frac{1}{\rho^2}\partial_\varphi^2 Q_\epsilon - \partial_z^2 Q_\epsilon + \frac{1}{\epsilon^2} Df(Q) + \frac{1}{\epsilon^\gamma} (Q_\epsilon - \widetilde{Q_\epsilon}) \\
&= R_\varphi^\top \left( -\partial_\rho^2 Q_\epsilon - \frac{1}{\rho} \partial_\rho Q_\epsilon - \partial_z^2 Q_\epsilon  + \frac{1}{\epsilon^\gamma} (Q_\epsilon - \widetilde{Q_\epsilon}) \right) R_\varphi \\
&\:\:\:\:\:- \frac{1}{\rho^2}\partial_\varphi^2 (R_\varphi^\top Q_\epsilon R_\varphi) + \frac{1}{\epsilon^2} Df(R_\varphi^\top Q_\epsilon R_\varphi)\, .
\end{align*} One can explicitly calculate that $\partial_\varphi^2 (R_\varphi^\top Q_\epsilon R_\varphi) = R_\varphi^\top Q_{2\times2,\epsilon}R_\varphi$ and since $Df(P) = -a P - b P^2 + c\,\tr(P^2)P$ for symmetric matrices $P$ we also have $Df(R_\varphi^\top Q_\epsilon R_\varphi) = R_\varphi^\top Df(Q_\epsilon) R_\varphi$. This implies that a rotationally equivariant extended solution of (\ref{rem:eq:elg_of_Q_eps}) is also solution of (\ref{rem:eq:elg_of_Q_eps_3D}).
\end{enumerate}
\end{remark}

The last part of this subsection will be the following Proposition which quantifies the regularity we have gained by replacing $\widetilde{Q_\epsilon}$ with $Q_\epsilon$. This result relies on the three dimensional Euler-Lagrange equation. In fact, this is the only time we use (\ref{rem:eq:elg_of_Q_eps_3D}) and cannot use (\ref{rem:eq:elg_of_Q_eps}) due to its singular behaviour near $\rho=0$.

\begin{proposition} \label{prop:bounds_Q_eps}
Let $\Vert \widetilde{Q_\epsilon}\Vert_{L^\infty}\leq 1$ and let $Q_\epsilon$ be the rotationally equivariant extended minimizer of (\ref{def:reg_seq_low_energy}). Then $Q_\epsilon\in C^1(\Omega,\Sym)$,
$$ \Vert Q_\epsilon\Vert_{L^\infty} \leq \sqrt{\frac{2}{3}}s_* \quad\text{and}\quad \Vert \nabla Q_\epsilon\Vert_{L^\infty} \leq \frac{C}{\epsilon}\, . $$
\end{proposition}

\begin{proof}
From equation (\ref{rem:eq:elg_of_Q_eps_3D}) and by elliptic regularity we deduce that for $\widetilde{Q_\epsilon}\in H^1$ we have $Q_\epsilon\in H^3$, i.e.\ $Q_\epsilon\in C^{1,\frac{1}{2}}$ since we are in dimension $3$. Note that the boundary of $\Omega$ is smooth. To prove the $L^\infty$-bounds we define a comparison map
$$ \overline{Q_\epsilon} \defi \begin{cases} \sqrt{\frac{2}{3}}s_*\frac{Q_\epsilon}{|Q_\epsilon|} & \text{if } |Q_\epsilon|>\sqrt{\frac{2}{3}}s_*\, , \\ Q_\epsilon & \text{otherwise.} \end{cases} $$
Then $|\nabla \overline{Q_\epsilon}|\leq |\nabla Q_\epsilon|$, $f(\overline{Q})\leq f(Q_\epsilon)$ by Proposition \ref{prop:prop_f} and $|\overline{Q_\epsilon}-\widetilde{Q_\epsilon}| \leq |Q_\epsilon-\widetilde{Q_\epsilon}|$. Hence $E_\epsilon^{3D}(\overline{Q_\epsilon},\Omega)\leq E_\epsilon^{3D}(Q_\epsilon,\Omega)$ with strict inequality unless $\overline{Q_\epsilon}=Q_\epsilon$.
The estimate $\Vert \nabla Q_\epsilon\Vert_{L^\infty} \leq \frac{C}{\epsilon}$ follows from \cite[Lemma A.2]{Bethuel1993}, using (\ref{rem:eq:elg_of_Q_eps_3D}), (\ref{eq:L_infty_bound}) and $\gamma< 2$.
\end{proof}

\subsection{Finite number of singularities away from $\rho=0$}

We introduce the notation $\Omega_\sigma\defi \{ x\in\Omega\sd x_1^2+x_2^2\geq \sigma^2 \} =  \Omega\setminus Z_\sigma$ for $\sigma>0$, with $Z_\sigma$ defined as in Theorem \ref{thm:main}.
In the same spirit, we define the two dimensional analogue $\Omega_\sigma' = \{(\rho,z)\in\Omega'\sd \rho> \sigma\}$, i.e.\ $\Omega_\sigma$ can be obtained from $\Omega_\sigma'$ through rotation around the $\ee_3-$axis.

The main theorem we want to prove in this subsection is the following:

\begin{theorem} \label{thm:finite_set_of_sing}
For all $\sigma,\delta>0$ there exists $\lambda_0,\epsilon_0>0$ and a set $X_\epsilon\subset\overline{\Omega'}$ such that for $\epsilon\leq \epsilon_0$
\begin{enumerate}
\item The set $X_\epsilon$ is finite and its cardinality is bounded independently of $\epsilon$.
\item If $x\in\Omega_\sigma'$ and $\dist(x,X_\epsilon)>\lambda_0 \epsilon$, then $\dist(Q_\epsilon(x),\N)\leq \delta$.
\end{enumerate}
\end{theorem}

The general idea behind this subsection is the same as in \cite{Canevari2013,Canevari2015}, where the analysis has been carried out for the case of minimizers of the energy $\int|\nabla Q_\epsilon|^2 + \frac{1}{\epsilon^2} f(Q_\epsilon)$ and uses ideas from \cite{Bethuel1999}. We will show that in our situation with the additional term $\frac{1}{\epsilon^\gamma}\Vert Q_\epsilon-\widetilde{Q_\epsilon}\Vert^2_{L^2}$ the same results hold. 
There are two main ingredients for the proof of Theorem \ref{thm:finite_set_of_sing}: Proposition \ref{prop:bound_grad_Q_eps_eta_log_dist} that tells us that a singularity has an energy cost of order $|\ln\epsilon|$ and Proposition \ref{prop:clearing_out} that allows us to deduce that $Q_\epsilon$ is close to being uniaxial provided $\frac{1}{\epsilon^2}\int f(Q_\epsilon)$ is sufficiently small. While the second ingredient uses only the regularity of $Q_\epsilon$, the first one makes use of equation (\ref{rem:eq:elg_of_Q_eps}) in the form of the following Proposition.

\begin{proposition}[\foreignlanguage{russian}{Похожаев}] \label{prop:pohozaev}
Let $Q_\epsilon$ be the minimizer of (\ref{def:reg_seq_low_energy}) and $\omega'\subset\Omega'$ open with Lipschitz boundary, $\overline{x}\in\omega'$. Then
\begin{align*}
\int_{\partial\omega'} \rho &((x-\overline{x})\cdot\nu) \bigg( \frac{1}{2}|\nabla' Q_\epsilon|^2 + \frac{1}{2\rho^2}|\partial_\varphi Q_\epsilon|^2 + \frac{1}{\epsilon^2}f(Q_\epsilon) + \frac{1}{2\epsilon^\gamma} |Q_\epsilon - \widetilde{Q_\epsilon}|^2 \bigg) \\
&= \frac{1}{2}\int_{\omega'} \rho |\nabla'Q_\epsilon|^2 + \frac{1}{2}\int_{\omega'} \frac{1}{\rho} |\partial_\varphi Q_\epsilon|^2 + \frac{3}{\epsilon^2} \int_{\omega'} \rho f(Q_\epsilon) + \frac{3}{2\epsilon^\gamma} \int_{\omega'}\rho |Q_\epsilon - \widetilde{Q_\epsilon}|^2   \\
&\:\:\:\:\:+ \frac{1}{\epsilon^\gamma} \int_{\omega'}\rho(Q_\epsilon-\widetilde{Q_\epsilon}):((x-\overline{x})\cdot\nabla'\widetilde{Q}) + \int_{\partial\omega'} \rho\:((x-\overline{x})\cdot \nabla' Q_\epsilon):(\nu\cdot\nabla' Q_\epsilon)\, ,
\end{align*}
where $\nu$ denotes the outward unit normal vector on $\partial\omega'$.
\end{proposition}

\begin{proof}
To improve readability, we drop the subscript $\epsilon$ in the proof. Our calculation only requires that $Q$ is solution of equation (\ref{rem:eq:elg_of_Q_eps}). 

Let $\omega'\subset\Omega'$ open with Lipschitz boundary and let $\overline{x}\in\omega'$ be an arbitrary point. By translation and without loss of generality we may assume that $\overline{x}=0$.
Testing the $ij$-component of equation (\ref{rem:eq:elg_of_Q_eps}) with $x_k\partial_{k}Q_{ij}$ and summing over $i,j,k$ we find
\begin{equation}\label{prop:pohozaev:eq_1}
\begin{split}
0 &= \sum_{i,j,k} \int_{\omega'} -\rho\Delta Q_{ij} x_k\partial_k Q_{ij} + \frac{1}{\epsilon^2}\int_{\omega'} \rho \frac{\partial f}{\partial Q_{ij}} x_k\partial_k Q_{ij} + \frac{1}{\epsilon^\gamma}\int_{\omega'} \rho(Q_{ij}-\widetilde{Q}_{ij})x_k\partial_k Q_{ij}  \\
&\hspace*{0.7cm} - \int_{\omega'} \partial_\rho Q_{ij} x_k \partial_k Q_{ij} + \int_{\omega'} \frac{1}{\rho} Q_{2\times 2,ij} x_k \partial_k Q_{ij} \\
&=: I + II + III + IV + V.
\end{split}
\end{equation}
Note, that the RHS of (\ref{rem:eq:elg_of_Q_eps}) vanishes since $Q_{ij}$ is traceless, i.e.\ 
$$ \sum_{i,j,k}\int_{\omega'}\Lambda \delta_{ij} x_k\partial_k Q_{ij} =  \sum_k \int_{\omega'} \Lambda x_k\partial_k\left(\sum_{i,j}\delta_{ij}Q_{ij}\right) = \sum_k \int_{\omega'} \Lambda x_k\partial_k (\tr(Q)) = 0 $$

For the first term $(I)$ we calculate, using integration by parts
\begin{equation}\label{prop:pohozaev:eq_Ia}
\begin{split}
\sum_{i,j,k,l} \int_{\omega'} -\rho\:\partial_l^2 Q_{ij} x_k\partial_k Q_{ij} &= \sum_{i,j,k,l} \int_{\omega'} \rho\:\partial_l Q_{ij} \delta_{lk} \partial_k Q_{ij} + \int_{\omega'} \rho\:\partial_l Q_{ij} x_k \partial_l\partial_k Q_{ij} \\
&\:\:\:\:\:\:- \int_{\partial\omega'} \rho\:\partial_l Q_{ij} x_k \partial_k Q_{ij} \nu_l + \int_{\omega'} \delta_{\rho l} \partial_l Q_{ij} \partial_k Q_{ij} x_k ,
\end{split}
\end{equation} where $\nu$ is the outward-pointing normal vector  on $\partial\omega'$. Note, that the last term reads $\int_{\omega'} (\partial_\rho Q):((x\cdot\nabla')Q)$ and thus is cancelled by (IV).
We apply another integration by parts to the second term on the RHS of (\ref{prop:pohozaev:eq_Ia}). This yields
\begin{align*}
\sum_{i,j,k,l}\int_{\omega'} \rho\: \partial_l Q_{ij} x_k \partial_l\partial_k Q_{ij} &= \sum_{i,j,k,l}\frac{1}{2}\int_{\omega'} \rho\: x_k \partial_k(\partial_l Q_{ij} \partial_l Q_{ij}) \\
&= -\frac{2}{2}\sum_{i,j,l} \int_{\omega'} \rho\: \partial_l Q_{ij} \partial_l Q_{ij}   + \sum_{i,j,k,l}\frac{1}{2}\int_{\partial\omega'} \rho\:\partial_l Q_{ij} \partial_l Q_{ij} x_k \nu_k \\
&\:\:\:\:\: -\frac{1}{2} \int_{\omega'} \delta_{\rho k} x_k \partial_l Q_{ij} \partial_l Q_{ij}.
\end{align*}
Combined with (\ref{prop:pohozaev:eq_Ia}) this gives
\begin{align}\label{prop:pohozaev:eq_Ib}
I+IV = \left(1-\frac{2}{2}-\frac{1}{2}\right)\int_{\omega'} \rho\:|\nabla' Q|^2 + \frac{1}{2}\int_{\partial\omega'} \rho\:|\nabla' Q|^2 (x\cdot\nu) - \int_{\partial\omega'} \rho\:(x\cdot\nabla' Q):(\nu\cdot\nabla' Q).
\end{align}

The second integral $(II)$ simply gives
\begin{align} \label{prop:pohozaev:eq_II}
II = \sum_k \frac{1}{\epsilon^2}\int_{\omega'} \rho\: \partial_k(f(Q)) x_k = -\frac{1}{\epsilon^2} \int_{\omega'} 3\rho\:f(Q) + \frac{1}{\epsilon^2} \int_{\partial\omega'} \rho\:f(Q) (x\cdot\nu).
\end{align}

For (III) we need to add (and subtract) the same integral with derivatives on $\widetilde{Q_{ij}}$. Then
\begin{equation}\label{prop:pohozaev:eq_III}
\begin{split}
III &= \frac{1}{\epsilon^\gamma}\int_{\omega'} \rho\:(Q_{ij}-\widetilde{Q}_{ij})\partial_k Q_{ij} x_k \\
&= \frac{1}{2\epsilon^\gamma} \int_{\omega'} \rho\:\partial_k (Q_{ij}-\widetilde{Q}_{ij})^2 x_k + \frac{1}{\epsilon^\gamma}\int_{\omega'} \rho\:(Q_{ij}-\widetilde{Q}_{ij})\partial_k\widetilde{Q}_{ij} x_k \\
&= - \frac{3}{2\epsilon^\gamma}\int_{\omega'} \rho\:(Q_{ij}-\widetilde{Q}_{ij})^2 + \frac{1}{2\epsilon^\gamma}\int_{\partial\omega'} \rho\:(Q_{ij}-\widetilde{Q}_{ij})^2 x_k\nu_k \\
&\:\:\:\:\: + \frac{1}{\epsilon^\gamma}\int_{\omega'} \rho\:(Q_{ij}-\widetilde{Q}_{ij})\partial_k\widetilde{Q}_{ij} x_k .
\end{split}
\end{equation}
The fifth integral (V) simply gives
\begin{equation} \label{prop:pohozaev:eq_V}
\begin{split}
\int_{\omega'} \frac{1}{\rho} Q_{2\times 2}:((x\cdot\nabla')Q) &= \int_{\omega'} \frac{1}{\rho} \frac{1}{2}(x\cdot\nabla')(Q_{2\times 2}:Q) \\
&= -\frac{1}{2}\int_{\omega'} \left( 0 + \frac{1}{\rho}\right)|\partial_\varphi Q|^2 + \frac{1}{2}\int_{\partial\omega'} (\nu\cdot x) \frac{1}{\rho} |\partial_\varphi Q|^2.
\end{split}
\end{equation}

Combining (\ref{prop:pohozaev:eq_Ib}), (\ref{prop:pohozaev:eq_II}), (\ref{prop:pohozaev:eq_III}) and (\ref{prop:pohozaev:eq_V}), the equality (\ref{prop:pohozaev:eq_1}) reads
\begin{align*}
\int_{\partial\omega'} \rho (x\cdot\nu)&\bigg(\frac{1}{2}|\nabla' Q|^2 + \frac{1}{2\rho^2} |\partial_\varphi Q|^2 + \frac{1}{\epsilon^2}f(Q)+\frac{1}{2\epsilon^\gamma}|Q-\widetilde{Q}|^2\bigg) \\
&= \frac{1}{2}\int_{\omega'} \rho\:|\nabla' Q|^2 + \frac{1}{\rho} |\partial_\varphi Q|^2 + \frac{3}{\epsilon^2}\int_{\omega'} \rho\:f(Q) + \frac{3}{2\epsilon^\gamma}\int_{\omega'} \rho\:|Q-\widetilde{Q}|^2 \\
&\:\:\:\:\:+ \frac{1}{\epsilon^\gamma}\int_{\omega'} \rho\:(Q-\widetilde{Q}):(x\cdot\nabla'\widetilde{Q}) + \int_{\partial\omega'} \rho\:(x\cdot\nabla' Q):(\nu\cdot\nabla' Q),
\end{align*} which gives the result.

\end{proof}

Since almost all term in consideration contain a $\rho$ factor due to the passage from $\Omega$ to $\Omega_\sigma'$, it is natural to introduce 
\begin{equation} \label{def:rho_min}
\rhomin(x_0,l) := \inf\big\{\rho\sd (\rho,z)\in B_{l}(x_0)\cap\Omega_\sigma'\big\}\, ,
\end{equation}
for a point $x_0\in\Omega_\sigma'$ and $l>0$. Note that if we write $x_0=(\rho_0,z_0)$, then $\rhomin(x_0,l)=\max\{\rho_0-l,\sigma\}$. In particular, $\rhomin(x_0,l)\geq \sigma$.

The following Proposition is a key ingredient in the proof of Theorem \ref{thm:finite_set_of_sing}.

\begin{proposition} \label{prop:clearing_out}
For all $\delta>0$ there exist constants $\lambda_0,\mu_0>0$ such that for all $\sigma>0$, $x_0\in\Omega_\sigma'$ and $l\in [\lambda_0\epsilon,1]$ the following implication holds:
$$ \frac{1}{\epsilon^2} \int_{B_{2l}(x_0)\cap \Omega_\sigma'} \rho\: f(Q_\epsilon) \leq \mu_0\: \rhomin(x_0,2l) \quad\Rightarrow\quad \dist(Q_\epsilon,\N)\leq \delta \text{ on } B_l(x_0)\cap\Omega_\sigma'\, . $$
\end{proposition}

\begin{proof}
We claim that $\lambda_0,\mu_0$ can be defined as
$$ \lambda_0\defi \frac{\delta}{2C}\, , \quad \mu_0\defi \frac{\pi}{4}\lambda_0^2 f_\mathrm{min}\, , $$
where $C$ is a constant such that $\epsilon\Vert\nabla Q_\epsilon\Vert_{L^\infty} \leq C$ (see Proposition \ref{prop:bounds_Q_eps}) and $f_\mathrm{min}$ is the minimum of $f$ on the set $\{ Q\in\Sym\sd |Q|\leq \sqrt{\frac{2}{3}}s_*, \dist(Q,\N)\geq \delta/2 \}$. Note that $f_\mathrm{min}>0$ since on this compact set $f$ is strictly positive.

In order to show that the definition indeed gives the desired implication, we argue by contradiction. Therefore we assume that there exists $x_0\in\Omega$ and $l\in [\lambda_0\epsilon,1]$ such that there is an $x\in B_l(x_0)\cap\Omega_\sigma'$ with $\frac{1}{\epsilon^2}\int_{B_{2l}(x_0)\cap \Omega_\sigma'} \rho\: f(Q_\epsilon)\leq\mu_0 \rhomin(x_0,2l)$ and $\dist(Q_\epsilon(x),\N)>\delta$.

This implies that $B_{\lambda_0\epsilon}(x)\subset B_{2l}(x_0)\cap (\mathbb{R}^2\setminus B_1(0))$. Indeed one can show that $\dist(x,\partial\Omega)>\lambda_0\epsilon$. Otherwise one would have $\dist(Q_\epsilon(x),\N)\leq \Vert\nabla Q_\epsilon\Vert_{L^\infty}\dist(x,\partial\Omega)\leq C\lambda_0=\frac{\delta}{2}$ by definition of $\lambda_0$. This clearly contradicts the assumption that $\dist(Q_\epsilon(x),\N)>\delta$. 
Then, for all $y\in B_{\lambda_0\epsilon}(x)\cap \Omega_\sigma'$ by the triangle inequality 
$$ \dist(Q_\epsilon(y),\N) \geq \dist(Q_\epsilon(x),\N) - |Q_\epsilon(x)-Q_\epsilon(y)| > \delta - \lambda_0\epsilon\Vert\nabla Q_\epsilon\Vert_{L^\infty} \geq \frac{\delta}{2}\, . $$
By definition of $f_\mathrm{min}$ this implies $f(Q_\epsilon(y))>f_\mathrm{min}$. Since $B_{\lambda_0\epsilon}(x)\cap \Omega_\sigma'\subset B_{2l}(x_0)\cap \Omega_\sigma'$ and $|B_{\lambda_0\epsilon}(x)\cap \Omega_\sigma'|\geq \frac{1}{2}\pi(\lambda_0\epsilon)^2$ we know that 
\begin{align*}
\frac{1}{\epsilon^2}\int_{B_{2l}(x_0)\cap\Omega_\sigma'} \rho\: f(Q_\epsilon) &\geq \frac{1}{\epsilon^2}\rhomin(x_0,2l)\int_{B_{\lambda_0\epsilon}(x)\cap\Omega_\sigma'} f(Q_\epsilon) \\
&\geq \frac{1}{\epsilon^2}\rhomin(x_0,2l) \frac{\pi}{2}(\lambda_0 \epsilon)^2 f_\mathrm{min} = 2\mu_0 \rhomin(x_0,2l)\, ,
\end{align*}
which contradicts our assumption.
\end{proof}

The next Lemma basically tells us that for $\alpha\in (0,1)$ there has to be some radius $r\leq \epsilon^{\alpha/2}$ so that we can control the energy on $\partial B_r$ in terms of the energy on $B_{\epsilon^{\alpha/2}}$. It will become important later on when we will use it to bound the energy contributions of the boundary terms from Pokhozhaev identity.

\begin{lemma} \label{lem:averaging}
For all $x_0\in\Omega'$ there exists $r\in (\epsilon^\alpha,\epsilon^\frac{\alpha}{2})$ (depending on $x_0$ and $\epsilon$) such that
$$ \int_{\partial B_r(x_0)\cap\Omega'} \rho\left(\frac{1}{2}|\nabla Q_\epsilon|^2  + \frac{1}{\epsilon^2} f(Q_\epsilon) + \frac{1}{2\epsilon^\gamma}|Q_\epsilon-\widetilde{Q_\epsilon}|^2 \right) \dx x \leq \frac{4 E_\epsilon^{2D}(Q_\epsilon,B_{\epsilon^{\alpha/2}}(x_0)\cap\Omega')}{\alpha r |\ln\epsilon|}\, . $$
\end{lemma}

\begin{proof}
The proof consists of an averaging argument. Assume that no such $r$ exists. With the notation $B'=B_{\epsilon^{\alpha/2}}(x_0)\cap\Omega'$, this would imply
\begin{align*}
E_\epsilon^{2D}(Q_\epsilon,B') &= \int_0^{\epsilon^{\alpha/2}} \int_{\partial B_r(x_0)\cap\Omega'} \rho\left(\frac{1}{2}|\nabla Q_\epsilon|^2  + \frac{1}{\epsilon^2} f(Q_\epsilon) + \frac{1}{2\epsilon^\gamma}|Q_\epsilon-\widetilde{Q_\epsilon}|^2 \right) \dx x \dx r \\
&\geq \int_{\epsilon^\alpha}^{\epsilon^{\alpha/2}} \int_{\partial B_r(x_0)\cap\Omega'} \rho\left(\frac{1}{2}|\nabla Q_\epsilon|^2  + \frac{1}{\epsilon^2} f(Q_\epsilon) + \frac{1}{2\epsilon^\gamma}|Q_\epsilon-\widetilde{Q_\epsilon}|^2 \right) \dx x \dx r \\
&\geq \frac{4 E_\epsilon^{2D}(Q_\epsilon,B')}{\alpha |\ln\epsilon|} \int_{\epsilon^{\alpha}}^{\epsilon^{\alpha/2}}\frac{1}{r} \dx r \\
&= \frac{4 E_\epsilon^{2D}(Q_\epsilon,B')}{\alpha |\ln\epsilon|} \frac{\alpha}{2} |\ln(\epsilon)| \\
&= 2 E_\epsilon^{2D}(Q_\epsilon,B') \, .
\end{align*}
This gives that $E_\epsilon^{2D}(Q_\epsilon,B')=0$ and thus $Q_\epsilon$ is constant on $B'$ and $Q_\epsilon = \widetilde{Q_\epsilon}\equiv q\in\N$. But since the constant map $q$ satisfies the Lemma, we get a contradiction.
\end{proof}

The following two results (Lemma \ref{lem:bound_mean_f} and Proposition \ref{prop:bound_grad_Q_eps_eta_log_dist})  are similar to \cite{Bethuel1999}, see also \cite[Lemma 1.4.8, Proposition 1.4.9]{Canevari2015}. Lemma \ref{lem:bound_mean_f} states that we can derive a better bound (independent of $\epsilon$) than (\ref{eq:energy_bound}) on balls $B_{\epsilon^\alpha}$ for the energy contribution of $f$. Then Proposition \ref{prop:bound_grad_Q_eps_eta_log_dist} tells us the cost in terms of energy for such a ball if $Q_\epsilon$ is not close to $\N$. Both results rely on Pokhozhaev identity (Proposition \ref{prop:pohozaev}) and Lemma \ref{lem:averaging}.

\begin{lemma} \label{lem:bound_mean_f}
Let $x_0\in\Omega'$. Then there exists a constant $C_\alpha>0$ which depends only on $\alpha,\gamma,\Omega$, the energy bound in (\ref{eq:energy_bound}) and the boundary data in \eqref{eq:bc} such that if $\epsilon$ is small enough
$$ \frac{1}{\epsilon^2} \int_{B_{\epsilon^\alpha}(x_0)\cap\Omega'} \rho\: f(Q_\epsilon) \dx x \leq C_\alpha\, . $$
\end{lemma}

\begin{proof} 
By Lemma \ref{lem:averaging} there exists $r\in (\epsilon^\alpha,\epsilon^\frac{\alpha}{2})$ and a constant $\overline{C}>0$ such that for $\epsilon$ small enough
\begin{equation}\label{lem:bound_mean_f:eq_e_eps}
\begin{split}
\int_{\partial B_r(x_0)\cap\Omega'} \rho\bigg(\frac{1}{2}|\nabla Q_\epsilon|^2 + &\frac{1}{\epsilon^2} f(Q_\epsilon) + \frac{1}{2\epsilon^\gamma}|Q_\epsilon-\widetilde{Q_\epsilon}|^2 \bigg) \leq \frac{\overline{C}}{\alpha r}\, ,
\end{split}
\end{equation} where we also used the energy bound (\ref{eq:energy_bound}).

Now assume in a first step that $B_r(x_0)\subset\Omega'$. Using the Pokhozhaev identity from Proposition \ref{prop:pohozaev} with $\omega'=B_r(x_0)$ and $\overline{x}=x_0$, we find
\begin{equation} \label{lem:bound_mean_f:eq_Br}
\begin{split}
\frac{3}{\epsilon^2} \int_{B_r(x_0)}\rho\: f(Q_\epsilon) &\leq \int_{\partial B_r(x_0)} \rho\:((x-x_0)\cdot\nu)\left( \frac{1}{2}|\nabla Q_\epsilon|^2 + \frac{1}{\epsilon^2} f(Q_\epsilon) + \frac{1}{2\epsilon^\gamma}|Q_\epsilon-\widetilde{Q_\epsilon}|^2 \right)  \\
&+ \frac{1}{\epsilon^\gamma}\int_{B_r(x_0)} \rho\:|Q_\epsilon - \widetilde{Q_\epsilon}| |(x-x_0)\cdot\nabla' \widetilde{Q_\epsilon}|   \\
&- \int_{\partial B_r(x_0)} \rho\:((x-x_0)\cdot\nabla' Q_\epsilon):(\nu\cdot\nabla' Q_\epsilon)\, .
\end{split}
\end{equation} 
Notice that since $x\in\partial B_r(x_0)$ we have $(x-x_0)\cdot\nabla' Q_\epsilon = r\nu\cdot\nabla'Q_\epsilon$, i.e.\
$$ ((x-x_0)\cdot\nabla' Q_\epsilon):(\nu\cdot\nabla' Q_\epsilon) = r \left|\nu\cdot\nabla' Q_\epsilon  \right|^2 \geq 0\, , $$
and $(x-x_0)\cdot\nu=r|\nu|^2 = r$.
Substituting this into (\ref{lem:bound_mean_f:eq_Br}), one gets 
\begin{align*}
\frac{3}{\epsilon^2}  \int_{B_r(x_0)} \rho\: f(Q_\epsilon)  &\leq r \int_{\partial B_r(x_0)}  \rho\:\left(\frac{1}{2}|\nabla Q_\epsilon|^2 + \frac{1}{\epsilon^2} f(Q_\epsilon) + \frac{1}{2\epsilon^\gamma}|Q_\epsilon-\widetilde{Q_\epsilon}|^2 \right)  \\
&\:\:\:\:\:+ \frac{1}{\epsilon^\gamma}\int_{B_r(x_0)} \rho\:|Q_\epsilon - \widetilde{Q_\epsilon}| |(x-x_0)\cdot\nabla' \widetilde{Q_\epsilon}|\, .
\end{align*}
By (\ref{lem:bound_mean_f:eq_e_eps}) and Cauchy-Schwarz inequality this entails
\begin{align*}
\frac{3}{\epsilon^2} \int_{B_r(x_0)} \rho\: f(Q_\epsilon) \dx x &\leq r \frac{\overline{C}}{\alpha r} + \frac{r}{\epsilon^\gamma}\left(\int_{B_r(x_0)} \rho\:|Q_\epsilon - \widetilde{Q_\epsilon}|^2\right)^\frac{1}{2} \left(\int_{B_r(x_0)}\rho\:|\nabla' \widetilde{Q_\epsilon}|^2\right)^\frac{1}{2} \\
&\leq \frac{\overline{C}}{\alpha} + C\frac{\epsilon^\frac{\alpha}{2}}{\epsilon^\gamma}\left( (1+|\ln\epsilon|)^2\epsilon^\gamma \right)^\frac{1}{2} \leq \frac{\overline{C}}{\alpha} + C\epsilon^{(\alpha-\gamma)/4}\, ,
\end{align*}
provided $\alpha>\gamma$ and $\epsilon$ small enough. This proves the claim in the case where $B_r(x_0)\subset \Omega'$. 

In a second step we show that the result also holds if $B_r(x_0)\nsubseteq \Omega'$. We define $\Gamma = B_r(x_0)\cap\partial\Omega'$ which is now non-empty. This enables us to write $\partial (B_r(x_0)\cap\Omega') = \Gamma\cup (\partial B_r(x_0)\cap\Omega')$.
Again we apply Proposition \ref{prop:pohozaev} with $\omega'=B_r(x_0)\cap\Omega'$ but this time we set $\overline{x}=z$, where $z\in\Omega'\cap B_r(x_0)$ is given by Proposition \ref{prop:Ome_bdry_Lip} for $y=x_0$. By Proposition  \ref{prop:pohozaev} we get
\begin{align*}
\frac{3}{\epsilon^2} \int_{B_r(x_0)\cap\Omega'} &\rho\:f(Q_\epsilon) \dx x \leq \int_{\partial B_r(x_0)\cap\Omega'} \rho\:((x-\overline{x})\cdot\nu) \left(\frac{1}{2}|\nabla Q_\epsilon|^2 + \frac{1}{\epsilon^2} f(Q_\epsilon) + \frac{1}{2\epsilon^\gamma}|Q_\epsilon-\widetilde{Q_\epsilon}|^2\right) \\
&+ \int_{\Gamma} \rho\:((x-\overline{x})\cdot\nu) \left(\frac{1}{2}|\nabla Q_\epsilon|^2 + \frac{1}{\epsilon^2} f(Q_\epsilon) + \frac{1}{2\epsilon^\gamma}|Q_\epsilon-\widetilde{Q_\epsilon}|^2 \right) \\
&- \frac{3}{2\epsilon^\gamma}\int_{B_r(x_0)\cap\Omega'} \rho\:|Q_\epsilon-\widetilde{Q_\epsilon}|^2 - \frac{1}{\epsilon^\gamma} \int_{B_r(x_0)\cap\Omega'} \rho\:(Q_\epsilon-\widetilde{Q_\epsilon}):((x-\overline{x})\cdot\nabla'\widetilde{Q}) \\
&- \int_{\Gamma} \rho\:((x-\overline{x})\cdot\nabla' Q_\epsilon):(\nu\cdot\nabla' Q_\epsilon) - \int_{\partial B_r(x_0)\cap\Omega'} \rho\:((x-\overline{x})\cdot\nabla' Q_\epsilon):(\nu\cdot\nabla' Q_\epsilon)\, ,
\end{align*}
where we denoted  $\nu$ the unit outward normal. For the integrals on $\partial B_r(x_0)\cap\Omega'$ and $B_r(x_0)\cap\Omega'$ we proceed as before using $|(x-\overline{x})\cdot\nu|\leq 2r$. Note, that this time $(x-\overline{x})\cdot \tau$ does not necessarily vanish. Nevertheless, the integral involving this term can be estimated from above by $\int_{\partial B_r\cap\Omega'} 2r \rho\:|\nabla'Q_\epsilon|^2$ and then be estimated using (\ref{lem:bound_mean_f:eq_e_eps}).
Now we estimate the integrals involving $\Gamma$. First note that $Q_\epsilon=\widetilde{Q_\epsilon}=Q_b$ on $\Gamma\cap \partial\Omega$ with $f(Q_b)=0$, i.e.\ $\int_{\Gamma\cap\partial\Omega}\rho\: f(Q_\epsilon)=0$ and $\int_{\Gamma\cap\partial\Omega} \rho\:|Q_\epsilon-\widetilde{Q_\epsilon}|^2=0$. 
On $\Gamma\setminus \partial\Omega\subset \{\rho=0\}$ we find that all integrals vanish because of the bounds in $Q_\epsilon$ established in Proposition \ref{prop:bounds_Q_eps}. We are left with the two integrals on $\Gamma\cap\partial\Omega$ with gradients. The idea is now to split the gradient into a tangential and a normal part. The tangential part depends only on the boundary data $Q_b$, the normal part needs to be estimated. So let $\tau$ be the unit tangent vector on $\Gamma$.  Decomposing $\nabla' Q_\epsilon = (\nu\cdot\nabla' Q_\epsilon)\nu + (\tau\cdot\nabla' Q_\epsilon)\tau$ and substituting this into $\int_{\Gamma\cap\partial\Omega}\rho(x-\overline{x})\cdot\nu\frac{1}{2}|\nabla' Q_\epsilon|^2$ yields
\begin{align*}
\frac{3}{\epsilon^2} \int_{B_r(x_0)\cap\Omega'} \rho\: f(Q_\epsilon) \dx x &\leq 4\frac{\overline{C}}{\alpha} + C\epsilon^{(\alpha-\gamma)/4} - \int_{\Gamma\cap\partial\Omega} \rho\:((x-\overline{x})\cdot\nabla' Q_\epsilon):(\nu\cdot\nabla' Q_\epsilon) \\
&\:\:\,+ \frac{1}{2}\int_{\Gamma\cap\partial\Omega} \rho\:((x-\overline{x})\cdot\nu)|\nu\cdot\nabla' Q_\epsilon|^2 + \frac{1}{2}\int_{\Gamma\cap\partial\Omega} \rho\:((x-\overline{x})\cdot\nu)|\tau\cdot\nabla' Q_\epsilon|^2 \\
&\leq 4\frac{\overline{C}}{\alpha} + C\epsilon^{(\alpha-\gamma)/4} + C_{Q_b}\epsilon^{\alpha/2} -\frac{1}{2}\int_{\Gamma\cap\partial\Omega} \rho\:((x-\overline{x})\cdot\nu)|\nu\cdot\nabla' Q_\epsilon|^2 \\
&\:\:\:\:\:-\int_{\Gamma\cap\partial\Omega} \rho\:((x-\overline{x})\cdot\tau)(\tau\cdot\nabla' Q_b):(\nu\cdot\nabla' Q_\epsilon)\, ,
\end{align*} where we used that $(x-\overline{x})=((x-\overline{x})\cdot\nu)\nu + ((x-\overline{x})\cdot\tau)\cdot\tau$ and that $\tau\cdot\nabla' Q_\epsilon = \tau\cdot\nabla' Q_b$ only depends on the given boundary values. 
We apply the inequality $ab\leq a^2/(2C^2)+ C^2 b^2/2$ with $C=\sqrt{C_\Omega/2}$ from Proposition \ref{prop:Ome_bdry_Lip} to get
\begin{align*}
\frac{3}{\epsilon^2} \int_{B_r(Q_\epsilon)\cap\Omega'} \rho\: f(Q_\epsilon) \dx x &\leq 4\frac{\overline{C}}{\alpha} + C\epsilon^{(\alpha-\gamma)/4} + C_{Q_b}\epsilon^{\alpha/2} -\frac{1}{2}\int_{\Gamma\cap\partial\Omega} \rho\:((x-\overline{x})\cdot\nu)|\nu\cdot\nabla' Q_\epsilon|^2 \\
&\hspace*{-0.5cm}+ \frac{1}{C_\Omega}\int_{\Gamma\cap\partial\Omega} \rho\:|(x-\overline{x})\cdot\tau||\tau\cdot\nabla' Q_b|^2 + \frac{C_\Omega}{4}\int_{\Gamma\cap\partial\Omega} \rho\:|(x-\overline{x})\cdot\tau||\nu\cdot\nabla' Q_\epsilon|^2 \, .
\end{align*}
Then we apply Proposition \ref{prop:Ome_bdry_Lip} to get
\begin{align*}
\frac{1}{\epsilon^2} \int_{B_r(Q_\epsilon)\cap\Omega'} \rho\: f(Q_\epsilon) \dx x &\leq 4\frac{\overline{C}}{\alpha} + C\epsilon^{(\alpha-\gamma)/4} + C_{Q_b}\epsilon^{\alpha/2} -\frac{1}{2}\int_{\Gamma\cap\partial\Omega} C_\Omega r \rho\:|\nu\cdot\nabla' Q_\epsilon|^2 \\
&\:\:\:\:\:+ \frac{C_\Omega}{4}\int_{\Gamma\cap\partial\Omega}  2 r \rho\:|\nu\cdot\nabla' Q_\epsilon|^2 \\
&= 4\frac{\overline{C}}{\alpha} + C\epsilon^{(\alpha-\gamma)/4} + C_{Q_b}\epsilon^{\alpha/2} \, .
\end{align*}
\end{proof}

\begin{proposition} \label{prop:Ome_bdry_Lip}
There exist constants $C_\Omega,\epsilon_1>0$ such that for all $0<\epsilon\leq\epsilon_1$, $r\in (\epsilon^\alpha,\epsilon^\frac{\alpha}{2})$ and $y\in\Omega'$ there exists $z\in B_r(y)\cap\Omega'$ such that
$$ \nu(x)\cdot (x-z) \geq C_\Omega r \quad \forall x\in \partial\Omega'\cap B_r(y)\, , $$
where $\nu$ is the outward unit normal on $\partial\Omega'$.
\end{proposition}

\begin{proof}
Let us start by considering the domain $R=\{(x_1,x_2)\in\mathbb{R}^2\sd x_1,x_2>0\}$. Let $y\in R$ and $r>0$ such that $B_r(y)\cap\partial R\neq\emptyset$ (otherwise the result is trivial). Let $L_1=|\{x_2=0\}\cap B_r(y)|$ and $L_2=|\{x_1=0\}\cap B_r(y)|$. Then we define $z=y+\frac{r}{2}\left(R_1/L(0,1)^\top+L_1/L(1,0)^\top\right)$, where $L^2=L_1^2+L_2^2$. We will show that this definition of $z$ indeed satisfies our claim. Without loss of generality we may assume that $y_1\geq y_2$. We consider the following cases:
\begin{enumerate}
\item $(0,0)\in B_r(y)$. In this case, $L_1=y_1+\sqrt{r^2-y_2^2}$ and $L_2=y_2+\sqrt{r^2-y_1^2}$. Let $x=(x_1,0)$. Then $\nu(x)=(0,-1)^\top$ and
$$ \nu(x)\cdot(x-z) = (y_2-x_2) + \frac{r}{2}\frac{L_1}{L} \geq \frac{r}{2} \frac{L_1}{L}\, . $$
Analogously, for $x=(0,x_2)$ we find $\nu\cdot(x-z)\geq\frac{r}{2}\frac{L_2}{L}$.
Since $y_1\geq y_2$ we have also the inequality $L_1\geq L_2$. Minimizing $L_2/L$ subject to the constraint $y_1\geq y_2$ we get $y_1=y_2$ and thus $L_1=L_2$, i.e.\ $\nu(x)\cdot(x-z)\geq \frac{r}{2\sqrt{2}}$.
\item $L_2\neq 0$ and $(0,0)\notin B_r(y)$. Then $L_1=2\sqrt{r^2-y_2^2}$ and $L_2=2\sqrt{r^2-y_1^2}$. A similar calculation as in the first case shows that $\nu(x)\cdot(x-z)\geq \frac{r}{2\sqrt{2}}$.
\item $L_2=0$. The lengths $L_1,L_2$ are given as in the second case, but since $L_2=0$ we get directly $\nu(x)\cdot(x-z)\geq \frac{r}{2} \frac{L_1}{L}=\frac{r}{2}$.
\end{enumerate}
Now we consider the domain $\Omega'$. For a radius $0<r<\frac{1}{2}$ the angular difference between the normal vectors of $\Omega'$ and $R$ is smaller than $\arccos(1-r)$. Thus, for $\epsilon_1$ small enough, $0<\epsilon\leq\epsilon_1$, $r\in (\epsilon^\alpha,\epsilon^\frac{\alpha}{2})$, we can find $C_\Omega>0$ such that
$$ \nu(x)\cdot(x-z) \geq \frac{r}{2}\cos\left(\frac{\pi}{4} + \arccos(1-r)\right) \geq \frac{r}{2}\cos\left(\frac{\pi}{4} + \arccos(1-\epsilon_1^{\alpha/2})\right) \geq C_\Omega\: r > 0\, . $$
\end{proof}

We have now all the necessary tools to prove the second important ingredient for the proof of Theorem \ref{thm:finite_set_of_sing}.

\begin{proposition} \label{prop:bound_grad_Q_eps_eta_log_dist}
For all $\delta,\sigma>0$ there exist $\epsilon_2,\zeta_\alpha>0$ such that for $0<\epsilon\leq \epsilon_2$ and $x_0\in\Omega_\sigma'$ the following implication holds:
$$\dist(Q_\epsilon(x_0),\N)>\delta \quad\Rightarrow\quad E_\epsilon^{2D}(Q_\epsilon,B_{\epsilon^\alpha}(x_0)\cap\Omega') \geq \zeta_\alpha(|\ln\epsilon| + 1)\rhomin(x_0,\epsilon^\alpha)\, , $$
with $\rhomin\geq\sigma$ defined as in \eqref{def:rho_min}. The constant $\zeta_\alpha$ can be chosen to be dependent only on $\alpha$ and $\delta$, while $\epsilon_2$ depends on $\delta,\sigma,\alpha,\gamma$.
\end{proposition}

\begin{proof}
Let's assume that the conclusion does not hold at $x_0\in\Omega_\sigma'$, i.e.\ $E_\epsilon^{2D}(Q_\epsilon,B_{\epsilon^\alpha}(x_0)\cap\Omega') \leq \zeta_\alpha(|\ln\epsilon| + 1)\rhomin(x_0,\epsilon^\alpha) $. Then there exists a radius $r\in (\epsilon^{2\alpha},\epsilon^\alpha)$ such that 
\begin{equation} \label{prop:bound_grad_Q_eps_eta_log_dist_eq1}
\int_{\partial B_r(x_0)\cap\Omega'} \rho\:\left(\frac{1}{2}|\nabla Q_\epsilon|^2 + \frac{1}{\epsilon^2} f(Q_\epsilon) + \frac{1}{2\epsilon^\gamma}|Q_\epsilon-\widetilde{Q_\epsilon}|^2\right) \dx x \leq  \frac{2\zeta_\alpha\rhomin(x_0,\epsilon^\alpha)}{\alpha r}\, .
\end{equation}
Indeed, otherwise
$$ E_\epsilon^{2D}(Q_\epsilon,B_{\epsilon^\alpha}(x_0)\cap\Omega') \geq \int_{\epsilon^{2\alpha}}^{\epsilon^\alpha} \frac{2\zeta_\alpha\rhomin(x_0,\epsilon^\alpha)}{\alpha r} \dx r = 2\zeta_\alpha\rhomin(x_0,\epsilon^\alpha)|\ln(\epsilon)|\, , $$
which clearly contradicts our assumption for $\epsilon<\frac{1}{e}$.

Replacing (\ref{lem:bound_mean_f:eq_e_eps}) by (\ref{prop:bound_grad_Q_eps_eta_log_dist_eq1}) in the proof of Lemma \ref{lem:bound_mean_f}, i.e.\ $\overline{C}=2\zeta_\alpha\rhomin(x_0,\epsilon^\alpha)$, we find
$$ \frac{1}{\epsilon^2}\int_{B_r(x_0)\cap\Omega'} \rho\: f(Q_\epsilon) \leq \frac{8\zeta_\alpha\rhomin(x_0,\epsilon^\alpha)}{\alpha} + C\epsilon_2^{(\alpha-\gamma)/4}\, , $$
where the constant $C$ can be chosen to be independent of $\alpha$ and $\epsilon$. 
We choose $\epsilon_2$ small enough such that it satisfies the  estimate $\lambda_0 \epsilon_2 < \frac{1}{2}\epsilon_2^{\alpha}$.
Now choose $\zeta_\alpha\leq \frac{\alpha\:\mu_0}{16}$ and $\epsilon_2\leq (\frac{\mu_0\sigma}{2C})^\frac{4}{\alpha-\gamma}$, where $\mu_0$ is the constant from Proposition \ref{prop:clearing_out}. These bounds imply that $\mu_0\rhomin(x_0,\epsilon^\alpha)\geq\frac{8\zeta_\alpha\rhomin(x_0,\epsilon^\alpha)}{\alpha} + C\epsilon_2^{(\alpha-\gamma)/4}$, i.e.\ we can apply Proposition \ref{prop:clearing_out} with $l=\frac{1}{2}\epsilon^\alpha$. This implies $\dist(Q_\epsilon(x_0),\N)\leq \delta$, which proves the claim.
\end{proof}

Now we can finally prove Theorem \ref{thm:finite_set_of_sing} and define the set of singularities $X_\epsilon$. To do this, one can proceed as follows: In a first step we cover $\Omega$ with balls of size $\epsilon^\alpha$ and look for balls where the energy is large. The number of such balls has to be finite because of the energy bound. In view of Proposition \ref{prop:bound_grad_Q_eps_eta_log_dist}, $Q_\epsilon$ will be almost uniaxial outside of these balls. In the second step we improve our estimates to the scale $\epsilon$. We cover the balls with high energy from step one with balls of size $\epsilon$ and determine balls where $f$ is large. By Lemma \ref{lem:bound_mean_f} this number will be finite too and Proposition \ref{prop:clearing_out} implies that $Q_\epsilon$ is indeed close to $\N$ on all other balls. We can then take $X_\epsilon$ to be the set of all centers of balls with large energy.

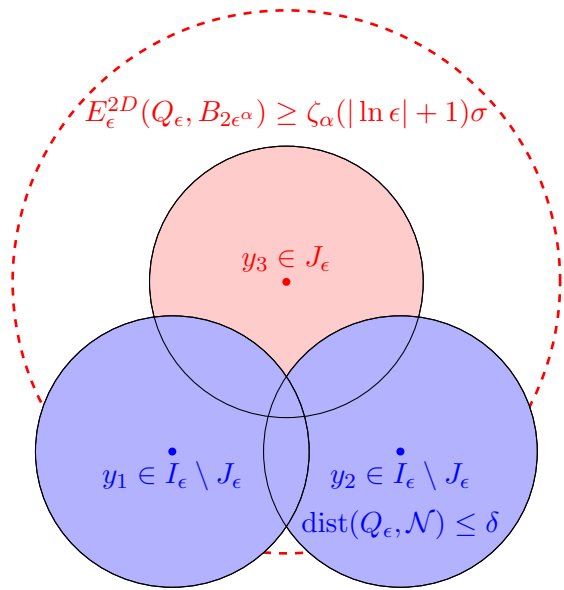
\begin{figure}[H]
\begin{center}
\begin{tikzpicture}[scale=0.75]
\draw[fill=red!20!white] (2,3) circle (2.4cm);
\draw[red, dashed, line width=1] (2,3) circle (4.8cm);
\draw[fill=blue!30!white] (4,0) circle (2.4cm);
\draw[fill=blue!30!white] (0,0) circle (2.4cm);
\draw[black] (4,0) circle (2.4cm);
\draw[black] (0,0) circle (2.4cm);
\draw[black] (2,3) circle (2.4cm);

\fill[red] (2,3)  circle[radius=2pt] node[above] {$y_3\in J_\epsilon$};
\fill[blue] (0,0)  circle[radius=2pt] node[below] {$y_1\in I_\epsilon\setminus J_\epsilon$};
\fill[blue] (4,0)  circle[radius=2pt] node[below] {$y_2\in I_\epsilon\setminus J_\epsilon$};

\node[blue] at (4,-1.3) {$\dist(Q_\epsilon,\N)\leq \delta$};
\node[red] at (2,6) {$E_\epsilon^{2D}(Q_\epsilon, B_{2\epsilon^{\alpha}}) \geq \zeta_\alpha (|\ln\epsilon|+1)\sigma $};
\end{tikzpicture}
\end{center}
\caption{First covering argument: Find balls $B_{\epsilon^\alpha}$, where the energy is large}
\label{fig:finite_sing_proof_1}
\end{figure}

\begin{proof}[Proof of Theorem \ref{thm:finite_set_of_sing}]
Let $\delta,\sigma>0$ be given and choose $\alpha\in (0,1)$. Let $\{ B_{\epsilon^{\alpha}}(y)\sd y\in \Omega' \}$ be a covering of $\Omega'$. By Vitali Covering Lemma there exists a countable family of points $\{y_i\}_{i\in I_\epsilon}$ such that
$$ \Omega'\subset \bigcup_{i\in I_\epsilon} B_{\epsilon^{\alpha}}(y_i)\, , \quad B_{\frac{1}{5}\epsilon^{\alpha}}(y_i)\cap B_{\frac{1}{5}\epsilon^{\alpha}}(y_j) = \emptyset \text{ if } i\neq j\, . $$

Let $\zeta_\alpha>0$ be given as in Proposition \ref{prop:bound_grad_Q_eps_eta_log_dist}. We define 
$$ J_\epsilon \defi \left\{ i\in I_\epsilon \sd E_\epsilon^{2D}(Q_\epsilon,B_{2\epsilon^{\alpha}}(y_i)\cap\Omega') > \zeta_\alpha(1+|\ln\epsilon|)\sigma \right\}\, . $$
 Then by the energy bound (\ref{eq:energy_bound}),
\begin{equation} \label{thm:finite_set_of_sing:eq_bound_J_eps}
\zeta_\alpha(1+|\ln\epsilon|)\sigma \# J_\epsilon \leq \sum_{i\in J_\epsilon} E_\epsilon^{2D}(Q_\epsilon,B_{2\epsilon^{\alpha}}(y_i)\cap\Omega') \leq C E_\epsilon^{2D}(Q_\epsilon,\Omega')\leq  C(1+|\ln\epsilon|)\, .
\end{equation} Indeed, note that there is a constant $C$ depending only on the space dimension such that each point in $\Omega'$ is covered by at most $C$ balls. This implies the second inequality in (\ref{thm:finite_set_of_sing:eq_bound_J_eps}). From (\ref{thm:finite_set_of_sing:eq_bound_J_eps}) we directly infer that the cardinality of $J_\epsilon$ is bounded by a constant dependent on $\delta,\sigma,\alpha$ as well as the space dimension and the energy bound, but independent of $\epsilon$. Let $i\in I_\epsilon\setminus J_\epsilon$ and $x_0\in B_{\epsilon^\alpha}(y_i)\cap \Omega_\sigma'$. If $\dist(Q_\epsilon(x_0),\N)>\delta$ we deduce by  Proposition \ref{prop:bound_grad_Q_eps_eta_log_dist} that $E_\epsilon^{2D}(Q_\epsilon, B_{2\epsilon^\alpha}(y_i)\cap\Omega') \geq E_\epsilon^{2D}(Q_\epsilon, B_{\epsilon^\alpha}(x_0)\cap\Omega') > \zeta_\alpha(|\ln(\epsilon)|+1)\sigma$, a contradiction to $i\in I_\epsilon\setminus J_\epsilon$. Hence
$$ \dist(Q_\epsilon(x),\N)\leq \delta \quad \forall x\in B_{\epsilon^\alpha}(y_i)\cap \Omega_\sigma', i\in I_\epsilon\setminus J_\epsilon\, . $$
See also Figure \ref{fig:finite_sing_proof_1}. Note, that this estimate is not good enough since we announced the radius around points in $X_\epsilon$ to be of order $\epsilon$ instead of $\epsilon^\alpha$. 

Now fix $i\in J_\epsilon$. Again by Vitali covering Lemma we can consider a covering of $B_{\epsilon^{\alpha}}(y_i)\cap\Omega_\sigma'$ of the form
$$ B_{\epsilon^{\alpha}}(y_i)\cap\Omega_\sigma' \subset \bigcup_{j\in I_{\epsilon,i}} B_{\lambda_0\epsilon}(z_j)\, , \quad B_{\frac{1}{5}\lambda_0\epsilon}(z_j)\cap B_{\frac{1}{5}\lambda_0\epsilon}(z_k) = \emptyset \text{ if } j\neq k\, , $$
with all $z_j\in B_{\epsilon^\alpha}(y_i)$ and where $\lambda_0$ is given by Proposition \ref{prop:clearing_out}. Furthermore, we define
$$ J_{\epsilon,i} \defi \left\{ j\in I_{\epsilon,i} \sd \frac{1}{\epsilon^2}\int_{B_{2\lambda_0\epsilon}(z_j)\cap\Omega_\sigma'} \rho\: f(Q_\epsilon) \geq \mu_0\:\sigma \right\}\, , $$
with $\mu_0$ again from Proposition \ref{prop:clearing_out}. By Lemma \ref{lem:bound_mean_f}, recalling that $2\lambda_0\epsilon < \epsilon^\alpha$
\begin{equation} \label{thm:finite_set_of_sing:eq_bound_J_eps_i}
\mu_0\:\sigma\: \# J_{\epsilon,i} \leq \sum_{j\in J_{\epsilon,i}} \frac{1}{\epsilon^2}\int_{B_{2\lambda_0\epsilon}(z_j)\cap\Omega_\sigma'} \rho\: f(Q_\epsilon) \leq \frac{C}{\epsilon^2} \int_{B_{\epsilon^\alpha}(y_i)\cap\Omega'} \rho\:f(Q_\epsilon)  \leq C_\alpha\, ,
\end{equation}
so that $\# J_{\epsilon,i}$ is also bounded independently of $\epsilon$. Applying Proposition \ref{prop:clearing_out} to the sets $B_{2\lambda_0\epsilon}(z_j)$ for $j\in I_{\epsilon,i}\setminus J_{\epsilon,i}$ we get that $\dist(Q_\epsilon(x),\N)\leq \delta$ for all $x\in B_{\lambda_0\epsilon}(z_j)\cap\Omega_\sigma'$, see Figure \ref{fig:finite_sing_proof_2}. 
Thus, setting $ X_\epsilon\defi \bigcup_{i\in J_\epsilon} J_{\epsilon,i} $ yields the result.
\end{proof}

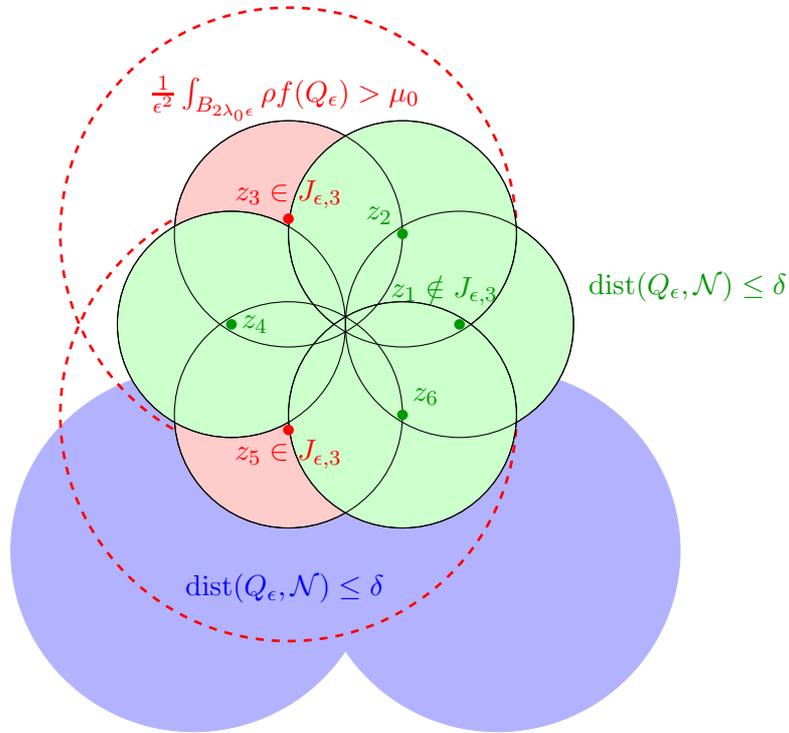
\begin{figure}[H]
\begin{center}
\begin{tikzpicture}
\draw[blue!30!white,fill=blue!30!white] (4,0) circle (2.4cm);
\draw[blue!30!white,fill=blue!30!white] (0,0) circle (2.4cm);

\draw[red, dashed, line width=1] (2-0.75,3-1.2) circle (3cm);
\draw[red, dashed, line width=1] (2-0.75,3+1.2) circle (3cm);
\draw[fill=red!20!white] (2-0.75,3+1.2) circle (1.5cm);
\draw[fill=red!20!white] (2-0.75,3-1.2) circle (1.5cm);

\draw[fill=green!20!white] (2+1.5,3) circle (1.5cm);
\draw[fill=green!20!white] (2-1.5,3) circle (1.5cm);
\draw[fill=green!20!white] (2+0.75,3+1.2) circle (1.5cm);
\draw[fill=green!20!white] (2+0.75,3-1.2) circle (1.5cm);

\draw[black] (2+1.5,3) circle (1.5cm);
\draw[black] (2-1.5,3) circle (1.5cm);
\draw[black] (2+0.75,3+1.2) circle (1.5cm);
\draw[black] (2+0.75,3-1.2) circle (1.5cm);
\draw[black] (2-0.75,3+1.2) circle (1.5cm);
\draw[black] (2-0.75,3-1.2) circle (1.5cm);

\fill[green!60!black] (2+1.5,3)  circle[radius=2pt]; \node[green!60!black] at (2+1.3,3+0.45) {$z_1\notin J_{\epsilon,3}$};
\fill[green!60!black] (2+0.75,3+1.2)  circle[radius=2pt] node[above left] {$z_2$};
\fill[red] (2-0.75,3+1.4)  circle[radius=2pt] node[above] {$z_3\in J_{\epsilon,3}$};
\fill[green!60!black] (2-1.5,3)  circle[radius=2pt] node[right] {$z_4$};
\fill[red] (2-0.75,3-1.4)  circle[radius=2pt] node[below] {$z_5\in J_{\epsilon,3}$};
\fill[green!60!black] (2+0.75,3-1.2)  circle[radius=2pt] node[above right] {$z_6$};

\node[blue] at (1.2,-0.5) {$\dist(Q_\epsilon,\N)\leq \delta$};
\node[green!60!black] at (6.5,3.5) {$\dist(Q_\epsilon,\N)\leq \delta$}; 
\node[red] at (1.2,6) {$\frac{1}{\epsilon^2}\int_{B_{2\lambda_0\epsilon}}\rho f(Q_\epsilon) >\mu_0$};
\end{tikzpicture}
\end{center}
\caption{Second covering argument: Find balls, where $\frac{1}{\epsilon^2}\int \rho f(Q_\epsilon)$ is large}
\label{fig:finite_sing_proof_2}
\end{figure}

\subsection{Lower bound near singularities}

The goal of this subsection is to precisely determine the cost of a singularity. The plan is to use estimates as in \cite[Chapter 6]{Chiron2004} which generalize the idea of \cite{Jerrard1999,Sandier1998}. The general idea is to decompose the gradient of a function into a derivative of its norm and of its phase as for example 
$$ |\nabla u|^2 = |\nabla|u||^2 + |u|^2\Big|\nabla \frac{u}{|u|}\Big|^2 $$
for any vectorial function $u$ that does not vanish.
Following \cite{Canevari2015}, we replace the phase $u/|u|$ by the projection of $Q_\epsilon$ onto $\N$. As a substitute for the norm, we introduce the auxiliary function $\phi$.

\begin{definition} \label{def:phi}
We define the function $\phi:\Sym\rightarrow\mathbb{R}$ by
$$ \phi(Q) = \begin{cases} \frac{1}{s_*} s(Q)\left(1-r(Q)\right) & Q\in\Sym\setminus\{0\}\, , \\ 0 & Q=0\, , \end{cases} $$
where $s_*$ is given as in Proposition \ref{prop:prop_f} and $s,r$ are the parameters from the decomposition of $Q$ in Proposition \ref{prop:prop_sym0}.
\end{definition}

\begin{proposition} \label{prop:grad_Q_geq_grad_R}
The function $\phi$ is Lipschitz continuous on $\Sym$ and $C^1$ on $\Sym\setminus\C$ with $\phi(Q)=1$ for all $Q\in \N$. Furthermore, for a domain $\omega\subset\Omega$ and $Q\in C^1(\omega,\Sym)$, the function $\R\circ Q$ is $C^1$ on the open set $Q^{-1}(\Sym\setminus \C)$ and the following estimate holds:
$$ |\nabla Q|^2 \geq \frac{s_*^2}{3}|\nabla (\phi\circ Q)|^2 + (\phi\circ Q)^2 |\nabla (\R\circ Q)|^2 \quad\text{ in } \omega\, , $$
where we use the convention that $(\phi\circ Q)^2 |\nabla (\R\circ Q)|^2\defi 0$ if $Q(x)\in\C$.
\end{proposition}

\begin{proof}
The Proposition follows directly from Lemma 2.2.3 and Lemma 2.2.7 in \cite{Canevari2015}.
\end{proof}

The next theorem gives the desired lower bound close to a singularity on a two dimensional unit disk. A proof of this can be found in \cite[Proposition 2.5]{Canevari2015a}.

\begin{theorem} \label{thm:near_singular}
There exist constants $\kappa_*,C>0$ such that for $Q\in H^1(B_1,\Sym)$ satisfying $Q(x)\notin \C$ for all $x\in B_1\setminus B_\frac{1}{2}$ and $(\R\circ Q)|_{\partial B_1}$ is non-trivial, seen as element of $\pi_1(\N)$ the following inequality holds
\begin{align} \label{thm:eq:near_singular_eq}
\int_{B_1} \frac{1}{2}|\nabla' Q|^2  + \frac{1}{\epsilon^2} f(Q) \dx x \geq \kappa_* \phi_0^2(Q,B_1\setminus B_\frac{1}{2})|\ln\epsilon| - C\, ,
\end{align} for a number $\phi_0(Q,B_1\setminus B_\frac{1}{2}) = \mathrm{essinf}_{B_1\setminus B_\frac{1}{2}}\phi(Q) > 0$. Furthermore, $\kappa_* = s_*^2\frac{\pi}{2}$.
\end{theorem}

The constant $\kappa_*$ can be calculated as in \cite[Lemma 2.9]{Canevari2015a} or \cite[Lemma 1.3.4]{Canevari2015} and is specific for $\N\cong\mathbb{R}P^2$. For other manifolds, there are analogous results with different constants, see \cite{Chiron2004}.
For our purposes, we will use the following version of Theorem \ref{thm:near_singular}.

\begin{corollary} \label{cor:near_singular}
Let $x_0\in \Omega'$ such that $B_\eta(x_0)\subset\Omega'$. Let $Q\in H^1(B_\eta(x_0),\Sym)$ satisfying $Q(x)\notin \C$ for all $x\in B_\eta\setminus B_{\frac{1}{2}\eta}$ and $(\R\circ Q)|_{\partial B_\eta}$ is non-trivial, seen as element of $\pi_1(\N)$. Then, with the same constant $C>0$ as in Theorem \ref{thm:near_singular} 
\begin{align} \label{thm:eq:near_singular}
\int_{B_\eta(x_0)} \frac{1}{2}|\nabla' Q|^2  + \frac{1}{\epsilon^2} f(Q) \dx x \geq \kappa_* \phi_0^2(Q,B_\eta\setminus B_{\frac{1}{2}\eta})\big(|\ln\epsilon| - |\ln\eta|\big)-C\, ,
\end{align} where $\kappa_*=s_*^2\frac{\pi}{2}$.
\end{corollary}

\begin{proof}
By translating $\Omega'$ we can assume that $x_0=0$. In order to apply Theorem \ref{thm:near_singular},  we define $\overline{x}=\frac{1}{\eta}x$ and $\overline{Q}(\overline{x})=Q(\eta \overline{x})=Q(x)$. Therefore $\overline{Q}\in H^1(B_1(0),\Sym)$ and verifies the hypothesis of Theorem \ref{thm:near_singular} with $\widetilde{\epsilon}=\epsilon\eta$, i.e.\
\begin{align*}
\int_{B_\eta(x_0)} \frac{1}{2}|\nabla' Q|^2  + \frac{1}{\epsilon^2} f(Q) \dx x &= \int_{B_1(x_0)} \frac{1}{2}|\nabla' \overline{Q}|^2  + \frac{1}{\eta^2\epsilon^2} f(\overline{Q}) \dx \overline{x} \\
&\geq \kappa_* \phi_0^2(\overline{Q},B_1\setminus B_\frac{1}{2})|\ln\widetilde{\epsilon}| - C \\
&\geq \kappa_* \phi_0^2(Q,B_\eta\setminus B_{\frac{1}{2}\eta})\big(|\ln\epsilon| - |\ln\eta|\big) - C \, .
\end{align*}
\end{proof}

\subsection{Lower bound away from singularities}

In the previous subsection we introduced the functions $\phi$ and $\phi_0$. The following Proposition shows that we can uniformly bound these functions if $Q$ is close to $\N$.

\begin{proposition} \label{prop:phi_close_to_N}
Let $\dist(Q,\N)\leq \delta$ on $\omega\subset\Omega$. Then 
$$ 1-\frac{2\sqrt{3}}{s_*}\delta \leq (\phi\circ Q)(x) \leq 1+\frac{2\sqrt{3}}{s_*}\delta\, . $$
\end{proposition}

\begin{proof}
Let $Q\in \Sym$ with $\dist(Q,\N)\leq\delta$. In other words, $|Q-\R(Q)|\leq \delta$, since $\R$ is the nearest-point projection onto $\N$.
We use Proposition \ref{prop:prop_sym0} to write
$$ Q = s\left( \left(\nn\otimes\nn - \frac{1}{3}\id\right) + r\left(\mm\otimes\mm - \frac{1}{3}\id\right)\right) \quad\text{and}\quad \R(Q) = s_*\left(\nn\otimes\nn - \frac{1}{3}\id\right)\, ,  $$
for $\nn,\mm$ orthonormal eigenvectors of $Q$, $s>0$ and $r\in [0,1)$. We can estimate
\begin{align*}
\delta^2 &\geq |Q - \R(Q)|^2 = \left| (s-s_*)(\nn\otimes\nn - \frac{1}{3}\id) + sr (\mm\otimes\mm - \frac{1}{3}\id) \right|^2 \\
&= \frac{2}{3}|s-s_*|^2 + \frac{2}{3}|sr|^2 - \frac{2}{3}sr(s-s_*) \\
&\geq \frac{1}{3}|s-s_*| + \frac{1}{3}|sr|^2\, .
\end{align*}
This implies $|sr|\leq \sqrt{3}\delta$ and $|s-s_*|\leq \sqrt{3}\delta$. Therefore, using $\phi(Q)=s_*^{-1}(s-sr)$, we have 
\begin{align*}
|s_*(\phi(Q) - 1)| &= |s-s_*-sr| \leq 2\sqrt{3}\delta\, .
\end{align*}
\end{proof}

Away from singularities the main contribution to the energy comes from the Dirichlet term and the external field since $Q_\epsilon$ is close to $\N$. More precisely, we only need the energy in radial direction, i.e.\ $|\nabla Q_\epsilon|^2$ can be replaced by $|\partial_r Q_\epsilon|^2$ and the problem becomes essentially one dimensional. We formalize this thoughts by introducing the following auxiliary problem as in \cite{Alama2017}
\begin{equation} \label{def:one_dim_aux_problem}
\inf_{\underset{n_3(r_1)=a,\: n_3(r_2)=b}{n_3\in H^1([r_1,r_2],[0,1])}} \int_{r_1}^{r_2} \frac{s_*^2 |n_3'|^2}{1-n_3^2} + \sqrt{\frac{3}{2}}(1-n_3^2) \dx r
\end{equation}
for $0\leq r_1\leq r_2\leq \infty$, $a,b\in [-1,1]$ and name the infimum $I(r_1,r_2,a,b)$. Note, that this is equivalent to minimizing $\int \left(\frac{1}{2}|\partial_r Q|^2 + g(Q)\right) \dx r$ for uniaxial $Q$ subject to suitable boundary conditions. For the infimum we have the following result.

\begin{lemma} \label{lem:radial_turning}
Let $0\leq r_1\leq r_2\leq r_3\leq \infty$ and $a,b,c\in [-1,1]$. Then
\begin{enumerate}
\item $I(r_1,r_2,a,b) + I(r_2,r_3,b,c)\geq I(r_1,r_3,a,c)$.
\item $I(r_1,r_2,-1,1)\geq 2\sqrt[4]{24}s_*$.
\item Let $\theta\in [0,\pi]$. Then 
$$ I(0,\infty,\cos(\theta),\pm 1) = \sqrt[4]{24}s_*(1\mp \cos(\theta))\, . $$
Furthermore, the minimizer $\nn(r,\theta)$ of $I(0,\infty,\cos(\theta),1)$ is $C^1$ and $|\partial_\theta \nn|^2,|\partial_r \nn|^2,|\nn-\ee_3|$ decay exponentially as $r\rightarrow\infty$. The minimizer can be explicitly expressed as
\begin{equation*}
\nn(r,\theta) = \begin{pmatrix} \sqrt{1-\nn_3^2} \\ 0 \\ \nn_3 \end{pmatrix}\, , \quad \nn_3(r,\theta) = \frac{A(\theta)-\exp(-\sqrt[4]{24}/s_* r)}{A(\theta)+\exp(-\sqrt[4]{24}/s_* r)}\, , \quad A(\theta)=\frac{1+\cos(\theta)}{1-\cos(\theta)}\, .
\end{equation*}
\end{enumerate}
\end{lemma}

\begin{proof}
The first part follows directly from definition, since any function that is admissible for  $I(r_1,r_2,a,b)$ combined with one for $I(r_2,r_3,b,c)$ is admissible for $I(r_1,r_3,a,c)$. 
For the second claim, we use the inequality $X^2+Y^2\geq 2XY$ with $X=s_*|\nn_3'|/\sqrt{1-\nn_3^2}$ and $Y=\frac{1}{2}\sqrt[4]{24}\sqrt{1-\nn_3^2}$ to get
$$ I(r_1,r_2,-1,1) \geq \sqrt[4]{24}s_*\int_{r_1}^{r_2} |\nn_3'| \dx r \geq \sqrt[4]{24}s_*\left|\int_{r_1}^{r_2} \nn_3' \dx r\right| = \sqrt[4]{24}s_*|\nn_3(r_2)-\nn_3(r_1)| = 2\sqrt[4]{24}s_* \, . $$
The third part follows from Lemma 3.4 and Remark 3.5 in \cite{Alama2017}.
\end{proof}

In the definition of the approximate energy $E_\epsilon^{2d}$ we did not include the term modelling the external field since this would complicate the proof of Theorem \ref{thm:finite_set_of_sing} and in particular the proof of the regularity for $Q_\epsilon$. Nevertheless, it is desirable to completely replace $\widetilde{Q_\epsilon}$ in $\Eex$ by $Q_\epsilon$ in order to work only with the regularized sequence. The following Lemma shows that on bounded sets that exclude singularities the replacement of $g(\widetilde{Q_\epsilon})$ by $g(\R\circ {Q_\epsilon})$ can be justified.
Although for our needs only this substitution is necessary, one could also take $g(Q_\epsilon)$ instead of $g(\R\circ Q_\epsilon)$ (see Remark \ref{rem:replace_Qtilde_Q_in_g}).

\begin{lemma}\label{lem:g_tildeQ_g_Q} 
There exists a constant $K>0$ such that for all $\sigma,\delta>0$ it exists $\epsilon_3>0$ such that on $\omega'\subset\Omega_\sigma'\cap B_R(0)$ for $1< R<\infty$ and $\dist(Q_\epsilon,\N)<\delta$ on $\omega'$, $\epsilon\leq \epsilon_3$ it holds
$$ \int_{\omega'} \rho \frac{1}{\eta^2}g(\widetilde{Q_\epsilon}) \dx x \geq \int_{\omega'} \rho \frac{1}{\eta^2}g(\R\circ Q_\epsilon) \dx x - K R^\frac{3}{2} \epsilon^{\gamma/4}\, . $$
\end{lemma}

\begin{proof}
By triangle inequality we can estimate
\begin{align*}
\int_{\omega'} \rho\frac{1}{\eta^2}g(\widetilde{Q_\epsilon}) \dx x 
{}\geq& \int_{\omega'}  \rho\frac{1}{\eta^2}g({\R\circ Q_\epsilon}) \dx x \\
&- \frac{1}{\eta^2}\int_{\omega'} \rho|g(\widetilde{Q_\epsilon}) - g(\R\circ\widetilde{Q_\epsilon})| \dx x - \frac{1}{\eta^2}\int_{\omega'} \rho |g(\R\circ\widetilde{Q_\epsilon}) - g(\R\circ Q_\epsilon)| \dx x \, .
\end{align*}
Compared to the announced estimate, it remains to show that the last two terms are bounded by $KR^\frac{3}{2}\epsilon^{\gamma/2}$. For the last term we use Proposition \ref{prop:prop_g} and Cauchy-Schwarz to get
\begin{align*}
\frac{1}{\eta^2}\int_{\omega'} \rho|g(\R\circ\widetilde{Q_\epsilon}) - g(\R\circ Q_\epsilon)| \dx x &\leq C\frac{1}{\eta^2}\int_{\omega} |\widetilde{Q_\epsilon} - Q_\epsilon| \dx x \leq \frac{C}{\eta^2} \Vert \widetilde{Q_\epsilon} - Q_\epsilon\Vert_{L^2({\omega})} |\omega|^\frac{1}{2} \, ,
\end{align*}
where $\omega\subset\mathbb{R}^3$ is the set defined through rotating $\omega'$ around the $\ee_3-$axis. The last expression is seen to be bounded by $\frac{C}{\eta^2}\epsilon^{\gamma/2} |\ln\epsilon|^\frac{1}{2} R^\frac{3}{2}$ in view of the energy bound.

For the other term, we want to use Proposition \ref{prop:prop_f} to derive a similar bound, but this requires $\dist(\widetilde{Q_\epsilon},\N)<2\delta$. So on the set $U'=\{x\in \omega'\sd \dist(\widetilde{Q_\epsilon}(x),\N)<2\delta \}$ we get
\begin{align*}
\frac{1}{\eta^2}\int_{U'} \rho|g(\widetilde{Q_\epsilon}) - g(\R\circ \widetilde{Q_\epsilon})| \dx x &\leq C\frac{1}{\eta^2}\int_{U'} \rho \:\dist(\widetilde{Q_\epsilon},\N) \dx x\\
&\leq \frac{C}{\eta^2} \int_{U'} \rho \sqrt{f(\widetilde{Q_\epsilon})}\dx x \leq \frac{C}{\eta^2} \left(\int_\omega f(\widetilde{Q_\epsilon})\dx x\right)^\frac{1}{2} |\omega|^\frac{1}{2} \, ,
\end{align*}
resulting in an upper bound of $\frac{C}{\eta^2}\epsilon |\ln\epsilon|^\frac{1}{2} R^\frac{3}{2}$.
We claim that $|\omega'\setminus U'|\leq \frac{C}{\eta}\epsilon^\gamma$.  Since on $\omega'\setminus U'$ one has $|Q_\epsilon-\widetilde{Q_\epsilon}| \geq \dist(\widetilde{Q_\epsilon},\N) - \dist(Q_\epsilon,\N)\geq \delta$, we infer from the energy bound that
$$ C\geq \int_{\omega'\setminus U'} \rho \frac{\eta}{2\epsilon^\gamma}|Q_\epsilon - \widetilde{Q_\epsilon}|^2\dx r \geq \frac{\eta \delta^2 \sigma}{2\epsilon^\gamma} |\omega'\setminus U'|\, , $$
which proves the claim. Since $g(\widetilde{Q_\epsilon}),g(\R\circ\widetilde{Q_\epsilon})\leq \sqrt{\frac{3}{2}}$ are bounded, we can deduce that $\frac{1}{\eta^2}\int_{\omega'\setminus U'} \rho |g(\widetilde{Q_\epsilon}-g(\R\circ\widetilde{Q_\epsilon})|\dx r\leq CR\epsilon^{\gamma/4} $ for $\eta,\epsilon$ small enough. 
\end{proof}

\begin{remark} \label{rem:replace_Qtilde_Q_in_g}
With the same assumptions as in Lemma \ref{lem:g_tildeQ_g_Q}, there exists another constant $\tilde{K}>0$ such that
$$ \int_{\omega'}  \rho\frac{1}{\eta^2}g(\widetilde{Q_\epsilon}) \dx x \geq \int_{\omega'} \rho\frac{1}{\eta^2}g(Q_\epsilon) \dx x - \tilde{K} R^\frac{3}{2} \epsilon^{\gamma/4}\, . $$
This follows as in the proof of Lemma \ref{lem:g_tildeQ_g_Q} if we introduce the additional term $-\frac{1}{\eta^2}\int_{\omega'} \rho |g(\R\circ Q_\epsilon)-g(Q_\epsilon)|\dx x$ and use $\dist(Q_\epsilon,\N)<\delta$.
\end{remark}

Now we can combine all our previous results to prove the lower bound of Theorem \ref{thm:main}. The idea consists in replacing $\widetilde{Q_\epsilon}$ by its approximation $Q_\epsilon$ (except for the term with $g$) and use the equivariance to write the energy as a two dimensional integral. By Theorem \ref{thm:finite_set_of_sing} we can exclude regions in $\Omega_\sigma'$ where $Q_\epsilon$ is far from $\N$. Extending the sets if necessary, we can assure that the union has vanishing measure in the limit $\eta,\epsilon\rightarrow 0$ and that the complement $\Omega_0$ is simply connected. The scaling of $\eta$ and $\epsilon$ allows to apply Corollary \ref{cor:near_singular} to each of these extended sets where the boundary datum is nontrivial. The expression we calculate here can later be identified as the perimeter term in $\E_0$.
In the simply connected complement $\Omega_0$ there exists a lifting $\nn^\epsilon$ of $Q_\epsilon$ which fulfils the compactness (\ref{thm:main:cptness}). In order to apply Lemma \ref{lem:radial_turning} to the rays in $\Omega_0$, we want to replace $g(\widetilde{Q_\epsilon})$ by $g(\R\circ Q_\epsilon)$. This can be accomplished using Lemma \ref{lem:g_tildeQ_g_Q}   on bounded sets. In order to get the lower bound, consider the rays with high energy (that we can estimate easily) and those with low energy where we need to be more precise about their behaviour far from the boundary $\partial\Omega$.
Using a diagonal sequence, we can pass to the limit $\sigma\rightarrow 0$.

\begin{proof}[Proof of the lower bound (\ref{eq:lower_bound}) of Theorem \ref{thm:main}]

Let $\delta,\sigma>0$ be arbitrary. We define $Q_\epsilon$ as in (\ref{def:reg_seq_low_energy}) and extend it rotationally equivariant. 
From Theorem \ref{thm:finite_set_of_sing} for $\epsilon\leq \epsilon_0$ we know that there exists a finite set $X_\epsilon$ of singular points $x_1^\epsilon,...,x_{N_\epsilon}^\epsilon$ in $\Omega_\sigma'$. 
In a first step, we suppose that all these points are included in the set $\Omega_{R}'=\Omega_\sigma'\cap B_{R}(0)$.

Since $\Omega_{R}'$ is bounded, there exists another finite set $X$, such that each sequence $x_j^\epsilon$ converges (up to a subsequence) to a point in $X$ as $\epsilon,\eta\rightarrow 0$. Note that there may be more than one sequence converging to the same point in $X$ and we a priori only know that $X\subset \overline{\Omega'\cap B_R}$. 

We first assume that the set $X$ is contained in $\Omega_\sigma'\setminus\partial\Omega$. Since $\eta|\ln\epsilon|\rightarrow \beta\in (0,\infty)$ we know that $\epsilon\leq C \exp(-\frac{1}{\eta})$. Assume that $\eta$ is small enough such that $2\lambda_0\epsilon\leq \frac{1}{2}\eta$.

For $x_i\in X$ we define $\widetilde{\Omega_i^\epsilon}{'} = \mathrm{conv}\{ B_{\eta}(x_i)\cup \{0\} \}\cap\Omega'$. If $x_i$ is the only point of the set $X$ that lies on the ray from $0$ through $x_i$ we define $\Omega_i^\epsilon{'} \defi\widetilde{\Omega_i^\epsilon}{'}$. If $x_j$ for $j\in J\subset I$ define the same ray, i.e.\ lie on a common line through $0$, then we set $\Omega_j^\epsilon{'}\defi\bigcup_{k\in J} \widetilde{\Omega_k^\epsilon}{'}$. After relabelling, we end up with a finite number $N$ of sets $\Omega_k^\epsilon{'}$, $k=1,...,N$. We define $\Omega_0^\epsilon{'}\defi \Omega_\sigma'\setminus \bigcup_{k=1}^N \Omega_k^\epsilon{'}$ (see  Figure \ref{fig:away_from_sing}).
Since all points in $X_\epsilon$ converge to some point  in $X$, we may assume that $\epsilon$ is small enough such that 
\begin{equation}\label{proof_main_eq_lambEps_Ome_k_}
\bigcup_{x\in X_\epsilon}B_{\lambda_0\epsilon}(x)\subset \bigcup_{x\in X}B_{2\lambda_0\epsilon}(x) \subset \bigcup_{k=1}^N \Omega_k^\epsilon{'}\subset \Omega_\sigma'\, .
\end{equation}
We drop the $\epsilon$ in the notation of $\Omega_k^\epsilon{'}$ for simplicity and call $\Omega_k$ the three dimensional set defined by rotating $\Omega_k'$ around the $\ee_3-$axis.

Using (\ref{def:reg_seq_low_energy}) and Remark \ref{rem:def_Q_eps} we can write
\begin{equation} \label{proof_main_eq_low_0_k}
\begin{split}
\eta\: \E_\epsilon(\widetilde{Q_\epsilon}) &\geq \eta\int_\Omega \frac{1}{2}|\nabla Q_\epsilon|^2 + \frac{1}{\epsilon^2} f(Q_\epsilon) + \frac{1}{\eta^2} g(\widetilde{Q_\epsilon}) + \frac{1}{2\epsilon^\gamma}|Q_\epsilon-\widetilde{Q_\epsilon}|^2 \dx x \\
&= \eta\int_0^{2\pi}\int_{\Omega'} \rho\:\left( \frac{1}{2}|\nabla Q_\epsilon|^2 + \frac{1}{\epsilon^2} f(Q_\epsilon) + \frac{1}{\eta^2} g(\widetilde{Q_\epsilon}) + \frac{1}{2\epsilon^\gamma}|Q_\epsilon-\widetilde{Q_\epsilon}|^2 \right) \dx \rho\dx z \dx\varphi \\
&\geq \eta\int_{\Omega_0} \frac{1}{2}|\nabla Q_\epsilon|^2 + \frac{1}{\eta^2} g(\widetilde{Q_\epsilon}) + \frac{1}{2\epsilon^\gamma}|Q_\epsilon-\widetilde{Q_\epsilon}|^2 \dx x  \\
&\:\:\:\:+ \eta\sum_{k=1}^N \int_0^{2\pi}\int_{\Omega_k'} \rho\:\left( \frac{1}{2}|\nabla Q_\epsilon|^2 + \frac{1}{\epsilon^2} f(Q_\epsilon) \right) \dx \rho\dx z \dx\varphi\, .
\end{split}
\end{equation}

For $x\in \Omega_0$ we know by Theorem \ref{thm:finite_set_of_sing} that $\dist(Q_\epsilon(x),\N)\leq \delta$. Since $\Omega_0'$ and thus $\Omega_0$ is simply connected there exist liftings $\pm\nn^\epsilon:\Omega_0\rightarrow\mathbb{S}^2$ such that 
$$ s_*\left(\nn^\epsilon\otimes\nn^\epsilon-\frac{1}{3}\id\right) = \R\circ Q_\epsilon  \quad \text{and}\quad \left\Vert s_*\left(\nn^\epsilon\otimes\nn^\epsilon-\frac{1}{3}\id\right) - Q_\epsilon\right\Vert_{\infty}\leq\delta\quad\text{ on }\Omega_0\, . $$
In particular, $Q_\epsilon(x)\in\Sym\setminus\C$ for all $x\in\partial\Omega_k'$ for all $k=1,...,N$.
Let $\M\subset \{1,...,N\}$ be the set of elements $k\in\{1,...,N\}$ such that $(\R\circ Q_\epsilon)|_{\partial \Omega_k'}$ is non-trivial as an element of $\pi_1(\N)$. We then want to apply Corollary \ref{cor:near_singular}. By Proposition \ref{prop:phi_close_to_N} we can estimate $\phi_0$ from below and get
\begin{equation} \label{proof_main_eq_low_near_sing}
\begin{split}
\eta\sum_{k=1}^N \int_{\Omega_k'} \rho\left(\frac{1}{2}|\nabla Q_\epsilon|^2 + \frac{1}{\epsilon^2} f(Q_\epsilon)\right) \dx \rho\dx z  \geq{} & \eta \sum_{k=1}^N \inf_{\Omega_k'}\rho \int_{\Omega_k'} \left(\frac{1}{2}|\nabla Q_\epsilon|^2 + \frac{1}{\epsilon^2} f(Q_\epsilon)\right) \dx \rho\dx z \\
\geq{}& \eta\sum_{k\in \M} \kappa_* \phi_0^2(Q_\epsilon,B_{\eta}(x_k)\setminus B_{\frac{1}{2}\eta}(x_k))\frac{\rho_k^\epsilon-\eta}{|x_k^\epsilon|} |\ln\epsilon|\eta \\
&- C\phi_0^2(Q_\epsilon,B_{\eta}(x_k)\setminus B_{\frac{1}{2}\eta}(x_k))\:\eta|\ln\eta| - C\: \eta \\
\geq{} & \left(1-\frac{2\sqrt{3}}{s_*}\delta\right)^2 \sum_{k\in \M} \frac{\rho_k^\epsilon-\eta}{|x_k^\epsilon|} \frac{\pi}{2}s_*^2\eta|\ln(\epsilon)| \\
&- C\left(1+\frac{2\sqrt{3}}{s_*}\delta\right)^2 \eta|\ln\eta| - C \eta\, .
\end{split}
\end{equation}

Before estimating the energy coming from $\Omega_0$, we need an additional information, namely we want to show that $\nn^\epsilon(r\omega)$ approaches $+\ee_3$ and $-\nn^\epsilon(r\omega)$ approximates $-\ee_3$ (or vice versa) as $r\rightarrow \infty$ for a.e. $\omega\in\mathbb{S}^2$. However, it will be enough for our analysis to just show that $\nn^\epsilon$ is close to either $+\ee_3$ or $-\ee_3$ up to some factor times $\delta$.
To start with, we show that the vector $\nn^\epsilon(r\omega)$ for $r\rightarrow \infty$ is close to $+\ee_3$ or $-\ee_3$ almost everywhere. 
By (\ref{def:reg_seq_low_energy}) and the energy bound we know, that for a.e. $\omega\in\mathbb{S}^2$ the integral 
\begin{equation}\label{proof_main_lifting_cv_e3_energy}
\int_R^\infty \frac{\eta}{2\epsilon^\gamma}|\widetilde{Q_\epsilon}(r\omega) - Q_\epsilon(r\omega)|^2 + \frac{1}{\eta} g(\widetilde{Q_\epsilon}(r\omega)) \dx r < \infty\, .
\end{equation}
We argue by contradiction, i.e.\ assume that there exists some $\omega\in\mathbb{S}^2$ satisfying (\ref{proof_main_lifting_cv_e3_energy}) such that $\limsup_{r\rightarrow\infty}||\nn_3^\epsilon(r\omega)|-1|>4 \mathfrak{C}\delta$ for a $\mathfrak{C}>0$ to be specified later. This implies that there exists a sequence $r_k$ such that $r_k\rightarrow\infty$ as $k\rightarrow\infty$ and $|\nn_3^\epsilon(r_k\omega)|< 1-8\mathfrak{C}\delta$ for all $k\in\mathbb{N}$ or in other words $|Q_\epsilon-s_*(\ee_3\otimes\ee_3-\frac{1}{3}\id)|> 4\delta$ for a suitably chosen $\mathfrak{C}$ (A calculation shows that $\mathfrak{C}\geq \frac{5}{4\sqrt{2} s_*}$ is sufficient). By Lipschitz continuity of $Q_\epsilon$
this implies $|Q_\epsilon-s_*(\ee_3\otimes\ee_3-\frac{1}{3}\id)|> 2\delta$ for all $r\in I_k:=(r_k-\frac{2\epsilon \mathfrak{C}\delta}{C},r_k+\frac{2\epsilon\mathfrak{C}\delta}{C})$. Now suppose that for some point in $I_k$ it holds that $|\widetilde{Q_\epsilon}-Q_\epsilon|<\frac{\delta}{4}$. Then $\dist(\widetilde{Q_\epsilon},\N)\leq |\widetilde{Q_\epsilon}-Q_\epsilon| + \dist(Q_\epsilon,\N)\leq \frac{5}{4}\delta$ and 
$$ \left|\widetilde{Q_\epsilon} - s_*\bigg(\ee_3\otimes\ee_3-\frac{1}{3}\id\bigg)\right| \geq \left|{Q_\epsilon} - s_*\bigg(\ee_3\otimes\ee_3-\frac{1}{3}\id\bigg)\right|  - |\widetilde{Q_\epsilon}-Q_\epsilon| > 2\delta - \frac{\delta}{4} \geq \frac{7}{4}\delta \, . $$ 
This implies that $g(\widetilde{Q_\epsilon})\geq g_\mathrm{min}>0$  for such points in $I_k$, where we used $g_\mathrm{min}=\min\big\{ g(Q) \sd Q\in\Sym\, , \dist(Q,\N)\leq\frac{5}{4}\delta \, ,|Q - s_*(\ee_3\otimes\ee_3-\frac{1}{3}\id)|\geq\frac{7}{4}\delta \big\}>0$. With this estimate in mind it becomes clear that we have the lower bound
$$ \frac{\eta}{2\epsilon^\gamma}|\widetilde{Q_\epsilon} - Q_\epsilon|^2 + \frac{1}{\eta}g(\widetilde{Q_\epsilon}) \geq \min\bigg\{ \frac{1}{\eta}g_\mathrm{min},\, \frac{\eta}{2\epsilon^\gamma}\Big(\frac{\delta}{4}\Big)^2 \bigg\}> 0 \quad \text{on } I_k\, . $$
Integrating over $I_k$ and summing over disjoint intervals yields a contradiction to (\ref{proof_main_lifting_cv_e3_energy}).
This implies that either $\limsup_{r\rightarrow\infty}\nn_3^\epsilon(r\omega)\geq 1-4\mathfrak{C}\delta$ or $\liminf_{r\rightarrow\infty}\nn_3^\epsilon(r\omega)\leq -1+4\mathfrak{C}\delta$. Indeed, $\nn_3^\epsilon(r\omega)$ cannot alternate between $\pm 1$ since by continuity this yields a contradiction for $\delta$ small enough such that $4\mathfrak{C}\delta \ll \frac{1}{2}$.
Next, consider the lifting $\nn^\epsilon$ and suppose that there exist directions $\omega_+,\omega_-\in\mathbb{S}^2$ such that $\nn^\epsilon(r\omega_+)$ is close to $+\ee_3$ (resp. $\nn^\epsilon(r\omega_-)$ close to $-\ee_3$) as $r\rightarrow\infty$. Since our previous analysis holds a.e., we can assume that the angle between $\omega_+$ and $\omega_-$ is smaller than $\pi$ and that $\omega_\pm$ are not parallel to $\ee_3$. Let $v=\omega_+ - \omega_-$ and $w=\omega_+ + \omega_-$. We estimate the energy in new coordinates $(r,s)$ in the segment between the rays defined through $\omega_+$ and $\omega_-$ and apply Lemma \ref{lem:g_tildeQ_g_Q} to get 
\begin{align*}
C &\geq \int_{R+1}^{\tilde{R}} \int_{-r|v|/2}^{r|v|/2} \rho\left(\frac{\eta}{2} \Big|\nabla'Q_\epsilon\Big(r\frac{v}{|v|}+s\frac{w}{|w|}\Big)\Big|^2 + \frac{1}{\eta} g\Big(\widetilde{Q_\epsilon}\Big(r\frac{v}{|v|}+s\frac{w}{|w|}\Big)\Big) \right) \dx s \dx r \\
&\geq C(1-C\delta)^2 \int_{(R+1)}^{\tilde{R}} \int_{r|v|/2}^{r|v|/2} \rho \left(\eta s_*^2\Big|\frac{v}{|v|}\cdot\nabla' \nn^\epsilon\Big|^2 + \frac{1}{\eta} \sqrt{\frac{3}{2}}(1-\nn_3^\epsilon) - C\delta \right) \dx s \dx r - K\epsilon^{\gamma/4}\tilde{R}^\frac{3}{2}\, .
\end{align*}
Lemma \ref{lem:radial_turning} gives the lower bound  $\int_{-r|v|/2}^{r|v|/2}\left(\eta s_*^2 |\frac{v}{|v|}\cdot\nabla' \nn^\epsilon|^2 + \frac{1}{\eta} \sqrt{\frac{3}{2}}(1-\nn_3^\epsilon) \right) \dx s \geq 2\sqrt[4]{24}s_*-C\delta$. Using $\rho\geq r\min\{\sin(\theta_+),\sin(\theta_-)\}$ for $\theta_\pm$ being the angular coordinate of $\omega_\pm$, we end up with
\begin{align*}
C &\geq C(1-C\delta)^2 \int_{R+1}^{\tilde{R}} r (2\sqrt[4]{24}s_* - C\delta) \dx r - K\epsilon^{\gamma/2}\tilde{R}^\frac{3}{2} \geq C_R(1-\delta-\epsilon^{\gamma/2})\tilde{R}^\frac{3}{2} > 0\, ,
\end{align*}
provided $\epsilon,\delta>0$ small enough.
Sending $\tilde{R}$ to infinity, we get a contradiction. Hence, $\nn^\epsilon$ has to approach either $+\ee_3$ or $-\ee_3$ a.e. and thus we can distinguish the two liftings by their asymptotics far from $\partial\Omega$.

We now introduce sets $F_{\sigma,\epsilon},\widetilde{F_{\sigma,\epsilon}}$ which we use later to prove the compactness result. First choose one of the two possible liftings $\nn^\epsilon\in C^0(\Omega_0,\mathbb{S}^2)$. Without loss of generality we choose the lifting such that $\nn^\epsilon(r\omega)$ is close to $+\ee_3$ as $r\rightarrow\infty$. 
The boundary conditions (\ref{eq:bc}) imply that $\nn^\epsilon(\omega)=\pm \nu(\omega)$, where $\nu$ is the outward normal on $\mathbb{S}^2$ for all $\omega\in\partial\Omega_0\cap\mathbb{S}^2$.
We define $F_{\sigma,\epsilon}\defi \{ \omega\in\mathbb{S}^2\cap\partial\Omega_0\sd \nn^\epsilon(\omega)\cdot \nu(\omega) = 1 \}$. Conversely, $\widetilde{F_{\sigma,\epsilon}}$ is then given by $\widetilde{F_{\sigma,\epsilon}} = \{ \omega\in\mathbb{S}^2\cap\partial\Omega_0\sd \nn^\epsilon(\omega)\cdot \nu(\omega) = -1 \}$. The remaining part of $\mathbb{S}^2\cap\Omega_\sigma$ is denoted $S_{\sigma,\epsilon} = (\mathbb{S}^2\cap\Omega_\sigma)\setminus(F_{\sigma,\epsilon}\cup \widetilde{F_{\sigma,\epsilon}})=\bigcup_{k\geq 1} (\mathbb{S}^2\cap\partial\Omega_k)$. Note that the sets $F_{\sigma,\epsilon}$, $\widetilde{F_{\sigma,\epsilon}}$ and $S_{\sigma,\epsilon}$ are rotationally symmetric with respect to the $\varphi$ coordinate. Since the $\theta-$angular size of all $\Omega_k$ converges to zero (i.e.\ $|S_{\sigma,\epsilon}|\rightarrow 0$ as $\epsilon\rightarrow 0$) and $\mathbb{S}^2\cap\Omega_\sigma$ is compact, we get that (up to extracting a subsequence) $\chi_{F_{\sigma,\epsilon}}$ (resp. $\chi_{\widetilde{F_{\sigma,\epsilon}}}$) converges pointwise to a characteristic function $\chi_{F_\sigma}$ (resp. $\chi_{\widetilde{F_\sigma}}$). Note that also $\Vert s_*(\nn^\epsilon\otimes\nn^\epsilon-\frac{1}{3}\id)-\widetilde{Q_\epsilon}\Vert_{L^2(\Omega_0)}$ converges to zero by Remark \ref{rem:def_Q_eps} and the definition of $\nn^\epsilon$.

As a last step, it remains the energy estimate on $\Omega_0$. We split the integral over $\Omega_0$ in \eqref{proof_main_eq_low_0_k} in several parts: For $\omega\in F_{\sigma,\epsilon}$ such that the energy on the ray in direction $\omega$ is large, i.e.\ $\int_1^\infty \frac{\eta}{2}|\nabla Q_\epsilon|^2 + \frac{\eta}{\epsilon^2} f(Q_\epsilon) + \frac{1}{\eta} g(\widetilde{Q_\epsilon}) + \frac{\eta}{2\epsilon^\gamma} |Q_\epsilon-\widetilde{Q_\epsilon}|^2 \dx r \geq 2\sqrt[4]{24} s_*$, we can use Lemma \ref{lem:radial_turning} that implies
\begin{equation}\label{proof_main_eq_highE_rays}
\int_1^\infty \frac{\eta}{2}|\nabla Q_\epsilon|^2 +\frac{\eta}{\epsilon^2} f(Q_\epsilon) + \frac{1}{\eta} g(\widetilde{Q_\epsilon}) + \frac{\eta}{2\epsilon^\gamma} |Q_\epsilon-\widetilde{Q_\epsilon}|^2 \dx r \geq 2\sqrt[4]{24} s_* \geq I(1,\infty,\nu_3(\omega),+1)\, .
\end{equation}
Analogously, for points $\omega\in \widetilde{F_{\sigma,\epsilon}}$ with energy greater than $2\sqrt[4]{24}s_*$ we use $I(1,\infty,\nu_3(\omega),-1)$ as a lower bound.
Let's consider the set of points $\omega\in\mathbb{S}^2\cap\partial\Omega_0$ such that the energy on the ray through $\omega$ is smaller than $2\sqrt[4]{24}s_*$. 
We claim that there exists a constant $\overline{C}>0$ independent of $\omega$ and a radius $R_{\eta,\omega}\in (R-\overline{C}\eta,R]$ such that $||\nn^\epsilon_3(R_{\eta,\omega}\omega)|-1|\leq 8\mathfrak{C}\delta\ll 1$. Indeed, the bound implies that $|\{ r\in (1,R)\sd |\widetilde{Q}-Q|>\delta \}|\leq 4\sqrt[4]{24}\delta^{-2}\epsilon^{\gamma/4}$ and 
if $||\nn^\epsilon_3(R_{\eta,\omega}\omega)|-1|> 8\mathfrak{C}\delta$ on $(R-\overline{C}\eta,R]\setminus \{ r\in (1,R)\sd |\widetilde{Q}-Q|>\delta \}$ then on this set $|\widetilde{Q_\epsilon}-s_*(\ee_3\otimes\ee_3-\frac{1}{3}\id)|\geq \delta$. Hence for $\overline{C}$ large enough this contradicts $2\sqrt[4]{24}\geq \int \frac{1}{\eta}g(\widetilde{Q_\epsilon}) \dx r \geq (R-(R-\overline{C}\eta))\frac{C\delta}{\eta}. $
In order to conclude that the energy from $1$ to $R_{\eta,\omega}$ is (up to some small contributions of size $\delta$) close to $I(1,\infty,\nu_3(\omega),\pm 1)$ we need to show that for $\omega\in F_{\sigma,\epsilon}$ the vector $\nn^\epsilon(R_{\eta,\omega})$ is close to $+\ee_3$ and not $-\ee_3$ (and vice versa for $\omega\in\tilde{F_{\sigma,\epsilon}}$). 
Again we argue by contradiction, i.e.\ we assume that $|\nn^\epsilon(R_{\eta,\omega}) + \ee_3|\leq 8\mathfrak{C}\delta$. We subdivide the ray in direction $\omega$ from $R$ to infinity into segments of length $1$, identified with the intervals $J_k=[k,k+1]$ for the radial variable,  for integers $k\geq R$. On every segment, the energy bound on the ray implies the existence of two points $a_k,b_k\in J_k$ with $|a_k-k|\leq\overline{C}\eta$, $|b_k-(k+1)|\leq\overline{C}\eta$ such that $||\nn^\epsilon_3(a_k)|-1|\leq 8\mathfrak{C}\delta$, $||\nn^\epsilon_3(b_k)|-1|\leq 8\mathfrak{C}\delta$. Since we assumed $\nn^\epsilon(R_{\omega,\eta})$ close to $-\ee_3$ and $\nn^\epsilon$ approaches $+\ee_3$ for $r\rightarrow \infty$, there exists some integer $k\geq R$ such that $|\nn^\epsilon_3(a_k)+1|\leq 8\mathfrak{C}\delta$, $|\nn^\epsilon_3(b_k)-1|\leq 8\mathfrak{C}\delta$. As before we see that the set where $|\widetilde{Q_\epsilon}-Q_\epsilon|>\delta$ is of size $C \epsilon^\gamma$ and that changes of $Q_\epsilon$ can only be of size $\epsilon^\gamma$ due to the energy bound on this ray. Together with Lemma \ref{lem:continuity_g} this implies
$$ \int_{J_k} \frac{\eta}{2}|\nabla Q_\epsilon|^2 + \frac{1}{\eta} g(\widetilde{Q_\epsilon}) \dx r \geq I(k,k+1,\nn_3^\epsilon(a_k),\nn_3^\epsilon(b_k)) - C(\mathfrak{C}+1+\overline{K}\epsilon^{\gamma/4})\delta \geq 2\sqrt[4]{24}s_* - C\delta\, , $$
where $\overline{K}>0$ is the constant coming from replacing $g(\widetilde{Q_\epsilon})$ by $g(Q_\epsilon)$ in the spirit of Lemma \ref{lem:g_tildeQ_g_Q} on the ray through $\omega$.
In order to show that for $\delta$ and $\epsilon$ small enough this contradicts the assumption of the ray having energy smaller than $2\sqrt[4]{24}s_*$, we prove that the energy coming from the segment $[0,R]$ has to be positive with a uniform lower bound. Since $\omega\in F_{\sigma,\epsilon}\subset\partial\Omega_\sigma$ one can show as in 2. in Lemma \ref{lem:radial_turning} that on such a  ray $ \int_1^R \frac{\eta}{2}|\nabla Q_\epsilon|^2 + \frac{1}{\eta}g(\widetilde{Q_\epsilon})\dx r  \geq \sqrt[4]{24}s_*(\frac{1}{2}\sigma^2 - 8\mathfrak{C}\delta) - \overline{K}\sqrt{R}\epsilon^{\gamma/4}$. So combining this result and the estimate for $J_k$ we get
$$ 2\sqrt[4]{24}s_* \geq 2\sqrt[4]{24}s_* - C\delta + \sqrt[4]{24}s_*\Big(\frac{1}{2}\sigma^2 - 8\mathfrak{C}\delta\Big) - \overline{K}\sqrt{R}\epsilon^{\gamma/4}\, , $$
which yields a contradiction for $\delta,\epsilon$ small enough.
For $\omega\in F_{\sigma,\epsilon}$ we then use the change of variables $r = 1+\eta\tilde{r}$, Proposition \ref{prop:grad_Q_geq_grad_R} and Proposition \ref{prop:phi_close_to_N} to get
\begin{equation}\label{proof_main_eq_lowE_rays}
\begin{split}
\int_1^R \frac{\eta}{2}|\nabla Q_\epsilon|^2 + \frac{1}{\eta} g(\widetilde{Q_\epsilon}) \dx r &= \int_0^{(R-1)/\eta} \frac{1}{2}|\nabla Q_\epsilon|^2 + g(\R\circ{Q_\epsilon}) \dx \tilde{r} - \overline{K}\sqrt{R}\epsilon^{\gamma/4} \\
&\geq (1-C\delta)^2 \int_0^{(R-1)/\eta} \frac{1}{2}|\nabla (\R\circ Q_\epsilon)|^2 + g(R\circ Q_\epsilon) \dx \tilde{r} - C\delta \\
&\geq I(0,(R_{\eta,\omega}-1)/\eta,\nu_3(\omega),\nn_3^\epsilon((R_{\eta,\omega}-1)/\eta))-  C\delta \\
&\geq I(0,(R_{\eta,\omega}-1/\eta,\nu_3(\omega),+1) - C\delta\, .
\end{split}
\end{equation}
So by \eqref{proof_main_eq_highE_rays} and \eqref{proof_main_eq_lowE_rays} we get that for $\omega\in F_{\sigma,\epsilon}$ we have 
$$\int_{1}^\infty \frac{\eta}{2}|\nabla Q_\epsilon|^2 + \frac{1}{\eta} g(\widetilde{Q_\epsilon}) \dx r \geq \min\{ I(0,\infty,\nu_3(\omega),+1), \, I(0,(R_{\eta,\omega}-1/\eta,\nu_3(\omega),+1) - C\delta \}\, .$$ 
Furthermore, by compactness, $\chi_{F_{\sigma,\epsilon}}$ converges point wise a.e. to $\chi_{F_\sigma}$. 
Since $(R_{\eta,\omega}-1)/\eta\rightarrow\infty$ as $\eta\rightarrow\infty$ we can apply Fatou's Lemma to get the energy contribution from $\Omega_0$ related to $F_{\sigma,\epsilon}$ by
\begin{equation*}\label{proof_main_eq_low_far_from_sing}
\begin{split}
 \liminf_{\epsilon,\eta\rightarrow 0}&\int_{F_{\sigma,\epsilon}} \int_1^\infty \frac{\eta}{2}|\nabla Q_\epsilon|^2 + \frac{1}{\eta} g(\widetilde{Q_\epsilon}) \dx r\dx\omega \\
&\geq \int_{\mathbb{S}^2\cap\partial\Omega_0}  \liminf_{\epsilon,\eta\rightarrow 0} \min\{ I(0,\infty,\nu_3(\omega),+1), \, I(0,(R_{\eta,\omega}-1/\eta,\nu_3(\omega),+1) - C\delta \} \chi_{F_{\sigma,\epsilon}}(\omega) \dx\omega \\
&\geq\int_{F_\sigma} I\left(0,\infty,\nu_3(\omega),+1\right) \dx\omega  -  C\delta\, .
\end{split}
\end{equation*}
Now combine this estimate, the analogous result for $\widetilde{F_{\sigma,\epsilon}}$, the formulae for $I(0,\infty,\nu_3(\omega),\pm 1)$ from Lemma \ref{lem:radial_turning} and \eqref{proof_main_eq_low_near_sing} to get 
\begin{equation*}
\begin{split}
 \liminf_{\epsilon,\eta\rightarrow 0} \eta\Eex(\Qex) &{}\geq \int_{F_\sigma} \sqrt[4]{24}s_*(1-\cos(\theta)) \dx\omega + \int_{\widetilde{F_\sigma}} \sqrt[4]{24}s_*(1+\cos(\theta)) \dx\omega \\
&+ \left(1-C\delta\right)^2 \sum_{k\in\M} \frac{\rho_k-\eta}{|x_k|}\pi^2 s_*^2 \beta -C\delta\, ,
\end{split}
\end{equation*}
for the points $x_k=(\rho_k,\theta_k)\in X$.

It remains to show that for all $k\in\M$, the point $x_k/|x_k|$ corresponds to a jump between $F_\sigma$ and $\widetilde{F_\sigma}$. For this it is enough to show that the orientation of $\nn^\epsilon$ relative to the normal on $\partial\Omega$ changes when following $\partial\Omega_k'\cap\Omega'$ for all $k\in\M$. So let $k\in\M$ and consider the curve $\Gamma:\partial\Omega_k'\rightarrow\mathbb{S}^2$ defined by $\nn^\epsilon|_{\partial\Omega_k'}$. By definition of $\M$, the curve is non-trivial in $\pi_1(\N)$, i.e.\ $\Gamma$ jumps an odd number of times from one vector to its antipodal vector on the sphere. Hence, the orientation has to change.
In the limit $\epsilon,\eta\rightarrow 0$, this implies that
$$ 2\pi\sum_{k\in\M} \frac{\rho_k}{|x_k|} = |D\chi_{F_\sigma}|(\mathbb{S}^2\cap\{\rho>\sigma\})\, . $$

This implies our result in the case $X_\epsilon,X\subset(\Omega'\cap B_R(0))\setminus \partial\Omega$. 

We now explain the changes in our construction if there are some $x_i\in X\cap\mathbb{S}^2$. Basically, we use the same construction as before, but we need to take care that the lower bound involving Corollary \ref{cor:near_singular} stays applicable. To see this, we extend the map $Q_\epsilon$ outside of $\Omega$ using the boundary values. 
 We define 
$$ \overline{Q_\epsilon}(x) = \begin{cases} Q_\epsilon(x) & x\in B_\eta(x_i)\cap\Omega\, , \\
s_*\Big(\frac{x}{|x|}\otimes \frac{x}{|x|} - \frac{1}{3}\id \Big) & x\in B_\eta(x_i)\cap B_1(0)\, . \end{cases} $$
Then $f(\overline{Q_\epsilon})=0$ and $|\nabla\overline{Q_\epsilon}|^2,g(\overline{Q_\epsilon})\leq C$ on $B_\eta(x_i)\cap B_1(0)$, i.e.\
$$ \int_{B_\eta(x_i)\cap B_1(0)}\frac{1}{2}|\nabla \overline{Q_\epsilon}|^2 + \frac{1}{\epsilon^2} f(\overline{Q_\epsilon}) + \frac{1}{\eta^2} g(\overline{Q_\epsilon}) \dx x \leq C_1\, . $$
So if $(\R\circ Q_\epsilon)|_{\partial\Omega'_i}$ is non-trivial as element of $\pi_1(\N)$, we can apply Corollary \ref{cor:near_singular} to the extension $\overline{Q_\epsilon}$, i.e.\
\begin{align*}
\eta\int_{B_\eta(x_i)\cap\Omega'} \frac{1}{2}|\nabla {Q_\epsilon}|^2 + \frac{1}{\epsilon^2} f({Q_\epsilon}) \dx x &\geq \eta\int_{B_\eta(x_i)\cap\mathbb{R}^2} |\nabla' \overline{Q_\epsilon}|^2 + \frac{1}{\epsilon^2} f(\overline{Q_\epsilon}) \dx x - \eta\:C_1 \\
&\geq \left(1-\frac{2\sqrt{3}}{s_*}\delta\right)^2 \frac{\pi}{2} s_*^2 \eta|\ln\epsilon| - C\:\eta|\ln\eta| - C\:\eta\, .
\end{align*}
If $(\R\circ Q_\epsilon)|_{\partial\Omega'_i}$ is trivial, then we just estimate as before, using that the energy is non-negative.

It remains one last case. Assume that there is a point $x_k^\epsilon\in X_\epsilon$ such that $|x_k^\epsilon|\rightarrow\infty$ as $\epsilon\rightarrow 0$.
This causes two modifications to our previous results:
This time, we define $\widetilde{\Omega_k^\epsilon}{'}=\mathrm{conv}\{B_\eta(x_k^\epsilon)\cup\{0\}\}\cap\Omega'$.  Doing so, we risk to exclude a region from $\Omega_0$ that is too large for proving the compactness, namely when we define the set $\omega_\eta$ afterwards. But in fact this is not really a difficulty for two reasons:
First, it is possible to extend $\nn^\epsilon$ continuously in $\widetilde{\Omega_k^\epsilon}{'}\setminus \widehat{\Omega_k^\epsilon{'}}$, with $\widehat{\Omega_k^\epsilon{'}}=(B_\eta(x_k^\epsilon)\cup[0,x_k^\epsilon])\cap\Omega'$, where $[0,x_k^\epsilon]$ is the line segment between the points $0$ and $x_k^\epsilon$.
Second, in order to conclude that also the measure of $\widehat{\Omega_k^\epsilon}$ is bounded, we need to show that $\rho_k^\epsilon$ cannot grow to infinity. To see this, note that $x_k^\epsilon\in \Omega_\sigma$ and by applying Proposition \ref{prop:bound_grad_Q_eps_eta_log_dist} one gets from the energy bound that $\rhomin(x_k^\epsilon,\epsilon^\alpha)$ is indeed bounded.
All estimates for the lower bound that we have done before stay valid in this setting.

So far, we have established the inequality
\begin{equation}\label{proof_main_eq_low_liminf_E_sigma}
\begin{split}
\liminf_{\eta,\xi\rightarrow 0} \eta\Eex(\Qex) &\geq (1-C\delta)^2 \frac{\pi}{2} s_*^2 \beta |D\chi_{F_\sigma}|(\mathbb{S}^2\cap \{\rho\geq\sigma\}) \\
&\hspace*{-1cm}+ \int_{F_\sigma} \sqrt[4]{24}s_*(1-\cos(\theta)) \dx\omega + \int_{\widetilde{F_\sigma}} \sqrt[4]{24}s_*(1+\cos(\theta)) \dx\omega - C\delta \, .
\end{split}
\end{equation}

We now define the set $\omega_{\sigma,\epsilon}$ as proxy for the set $\omega_\eta$ from Theorem \ref{thm:main}. Let $\omega_{\sigma,\epsilon}' := \bigcup_{k\geq 1} \widehat{\Omega_k^\epsilon{'}}$, where the sets $\widehat{\Omega_k^\epsilon{'}}=\Omega_k^\epsilon{'}$ for bounded sequences $|x_k^\epsilon|$, and given as in the second construction if $|x_k^\epsilon|$ diverges. This is well defined for $\epsilon$ (and therefore $\eta$) small, depending on $\sigma$ and $\delta$. Recall that since $\eta|\ln\epsilon|\rightarrow \beta\in (0,\infty)$, we have the asymptotic $\eta\sim|\ln\epsilon|^{-1}$. Let $\omega_{\sigma,\epsilon}$ be the corresponding rotational symmetric extended set. Then $|\omega_{\sigma,\epsilon}'|\leq C |\bigcup_{x\in X_\epsilon} B_\eta(x)|\leq  C \eta^2 |X_\epsilon| \leq C \frac{\eta^2}{\delta^4 \sigma^2}$, i.e.\ choosing $\eta$ small we can force the measure of $\omega_{\sigma,\epsilon}'$ to vanish in the limit. Note that this also implies that the measure of $\omega_{\sigma,\epsilon}$ vanishes because we have an upper bound on the $\rho-$component of points in $X_\epsilon$.

We now want to send $\sigma\rightarrow 0$ and choose a diagonal sequence with the properties announced in the Theorem. 
From our previous construction, for a sequence $\sigma_k\searrow 0$ there exist corresponding sequences $\delta_k\searrow 0$, $\eta_k\searrow 0$ and $\epsilon_k\searrow 0$ such that from \eqref{proof_main_eq_low_liminf_E_sigma}
\begin{align*}
\eta\Eex(\Qex) &\geq \frac{\pi}{2} s_*^2 \beta |D\chi_{F_{\sigma_k,\epsilon}}|(\mathbb{S}^2\cap \{\rho\geq\sigma_k\}) \\
&\hspace*{-1cm}+ \int_{F_{\sigma_k,\epsilon}} \sqrt[4]{24}s_*(1-\cos(\theta)) \dx\omega + \int_{\widetilde{F_{\sigma_k,\epsilon}}} \sqrt[4]{24}s_*(1+\cos(\theta)) \dx\omega - \frac{1}{k} \, ,
\end{align*}
and furthermore
$ |\omega_{\sigma_k,\epsilon}|\leq \frac{1}{k}  $,
$ |\mathbb{S}^2\setminus(F_{\sigma_k,\epsilon}\cup\widetilde{F_{\sigma_k,\epsilon}})|\leq \frac{1}{k} $
and
$\Vert \widetilde{Q_\epsilon} - s_*(\nn^\epsilon\otimes\nn^\epsilon - \frac{1}{3}\id)\Vert_{L^2(\Omega_{\sigma_k}\setminus\omega_{\sigma,\epsilon})}\leq \frac{1}{k}$
for $\epsilon\leq \epsilon_k$ and $\eta\leq\eta_k$. The sequences $\epsilon_k$ and $\eta_k$ depend on $\sigma_k$ and $\delta_k$ and are related via $\eta_k|\ln\epsilon_k|\rightarrow \beta$ as $k\rightarrow\infty$.

So we can define the function $\nn^\eta:\Omega\rightarrow\mathbb{S}^2$ announced in the Theorem as $\nn^\eta := \nn^\epsilon$ on $\Omega_{\sigma_k}\setminus\omega_{\eta}$ for $\eta\in (\eta_{k+1},\eta_k)$, $\omega_\eta:= \omega_{\sigma_k,\epsilon}$ and extend it measurably to a map $\Omega\rightarrow\mathbb{S}^2$.  This definition assures that $\nn^\eta\in C^0(\Omega_{\sigma_k}\setminus\omega_\eta,\mathbb{S}^2)$ and the convergence in \eqref{thm:main:cptness} holds. 
Furthermore, we define the set $F_\eta\defi F_{\sigma_k,\epsilon}$ for $\eta\in (\eta_{k+1},\eta_k)$. Then our analysis shows that the sequence $\chi_{F_\eta}$ has the point wise a.e. limit $\chi_F$, for $F=\bigcup_{k>1} F_{\sigma_k}$ since $|\chi_F-\chi_{F_\eta}|\leq |\chi_F - \chi_{F_{\sigma_k}}| + |\chi_{F_{\sigma_k}} - \chi_{F_{\sigma_k,\epsilon}}|$ and the measure of the set on which these two terms are nonzero is smaller than $C\sigma_k^2+\frac{1}{k}$.

This finishes the proof of the first part of Theorem \ref{thm:main} and (\ref{eq:lower_bound}).
\end{proof}

\begin{figure}[H]
\begin{center}
\begin{tikzpicture}[scale=0.85]
\fill[black!10!white] (0.5,-4) -- (11.5,-4) -- (11.5,6) -- (0.5,6) -- cycle;
\fill[black!5!white] (0.5,-4) -- (0,-4) -- (0,6) -- (0.5,6) -- cycle;
\fill[red!20!white] (0,0) -- (4-1.74,4+0.99) -- (4+0.98,4-1.74) -- cycle;
\draw[black] (0,0)--(4-1.74,4+0.99);
\draw[black] (0,0)--(4+0.98,4-1.74);
\draw[fill=red!20!white] (4,4) circle (2cm);
\fill[red] (4,4)  circle[radius=2pt] node[above] {$x_1\in X$};
\fill[red] (4,3)  circle[radius=1pt] node[above] {$y_1\in X_\epsilon$};
\node[red] at (2.4,2.4) {$\Omega_1$};

\fill[blue!30!white] (0,0) -- (6.5,1.94) -- (6.5,-1.94) -- cycle;
\draw[black] (0,0)--(6.5,1.94);
\draw[black] (0,0)--(6.5,-1.94);
\draw[fill=blue!30!white] (7,0) circle (2cm);
\fill[blue] (7,0)  circle[radius=2pt] node[below] {$x_2\in X$};
\fill[blue] (6,0)  circle[radius=1pt] node[above] {$y_2\in X_\epsilon$};
\fill[blue] (8,0)  circle[radius=1pt] node[above] {$y_3\in X_\epsilon$};
\node[blue] at (4.5,0) {$\Omega_2$};

\node[black] at (8,4) {$\Omega_0$};
\draw[black] (0,3)--(0,6);
\draw[black] (0,-3)--(0,-4);
\fill[white] (0,0) -- (90:3cm) arc (90:-90:3cm) -- cycle;
\draw[black] (0,0) -- (90:3cm) arc (90:-90:3cm) -- cycle;

\draw[violet, rotate=80, ->]  (3,0) -- (4,0);
\draw[violet, rotate=75, ->]  (3,0) -- (4,0);
\draw[violet, rotate=70, ->]  (3,0) -- (4,0);
\draw[violet, rotate=67, ->]  (3,0) -- (4,0);

\draw[violet, rotate=-80, ->]  (3,0) -- (2,0);
\draw[violet, rotate=-75, ->]  (3,0) -- (2,0);
\draw[violet, rotate=-70, ->]  (3,0) -- (2,0);
\draw[violet, rotate=-65, ->]  (3,0) -- (2,0);
\draw[violet, rotate=-60, ->]  (3,0) -- (2,0);
\draw[violet, rotate=-55, ->]  (3,0) -- (2,0);
\draw[violet, rotate=-50, ->]  (3,0) -- (2,0);
\draw[violet, rotate=-45, ->]  (3,0) -- (2,0);
\draw[violet, rotate=-40, ->]  (3,0) -- (2,0);
\draw[violet, rotate=-35, ->]  (3,0) -- (2,0);
\draw[violet, rotate=-30, ->]  (3,0) -- (2,0);
\draw[violet, rotate=-25, ->]  (3,0) -- (2,0);
\draw[violet, rotate=-20, ->]  (3,0) -- (2,0);
\draw[violet, rotate=-17, ->]  (3,0) -- (2,0);

\draw[violet, rotate=22, ->]  (3,0) -- (2,0);
\draw[violet, rotate=17, ->]  (3,0) -- (2,0);

\draw[violet, <-]  (11,5.5) -- (11,4.5);
\draw[violet, <-]  (11,3.5) -- (11,2.5);
\draw[violet, <-]  (11,1.5) -- (11,0.5);
\draw[violet, <-]  (11,-0.5) -- (11,-1.5);
\draw[violet, <-]  (11,-2.5) -- (11,-3.5);

\draw[black,dotted] (0.5,-4)--(0.5,6);
\draw[black, <->]  (0,5) -- (0.5,5);
\node at (0.25,5.25) {$\sigma$};
\node at (0.25,4.25) {$Z_\sigma$};

\node[] at (1,2.25) {$F_{\sigma,\epsilon}$};
\draw[<->] (80:2.8cm) arc (80:65:2.8cm);
\node[] at (3,-2.5) {$\widetilde{F_{\sigma,\epsilon}}$};
\draw[<->] (-80:3.2cm) arc (-80:-17:3.2cm);
\draw[<->] (23.5:3.2cm) arc (23.5:17.5:3.2cm);
\end{tikzpicture}
\end{center}
\caption{Construction made in the proof of Theorem \ref{thm:main}. The arrows show a lifting $\nn^\epsilon$. In the region $\Omega_1$ the director field $\nn^\epsilon$ has non-trivial homotopy class, around the region $\Omega_2$, $\nn^\epsilon$ has a trivial one.}
\label{fig:away_from_sing}
\end{figure}
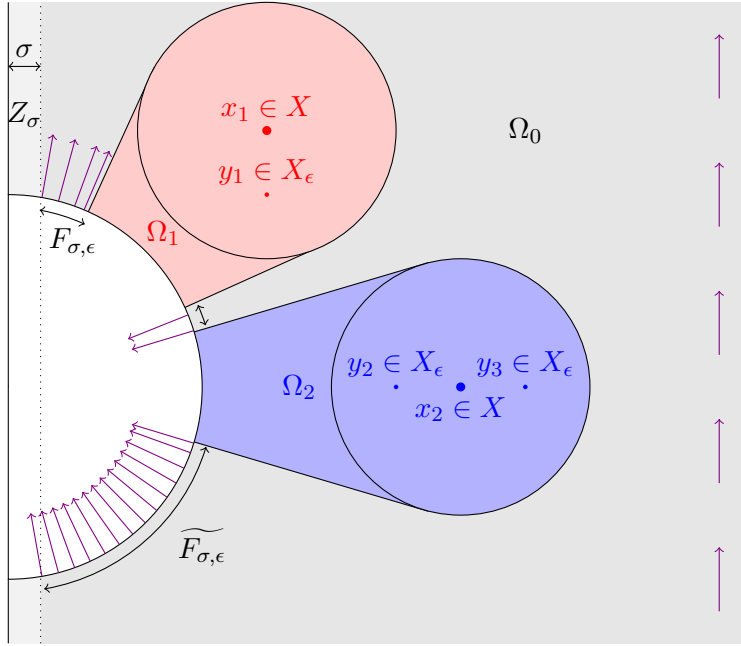

\section{Upper bound}

In this section we are going to prove the upper bound from Theorem \ref{thm:main}, namely (\ref{eq:upper_bound}). Since all functions are rotationally equivariant, it is useful to introduce the two dimensional energy for sets $\omega'\subset\Omega'$
$$ \E_\epsilon^{2D}(Q,\omega') = \int_{\omega'} \rho\left( \frac{1}{2}|\nabla' Q|^2 + \frac{1}{\rho^2}Q_{2\times 2}:Q + \frac{1}{\epsilon^2}f(Q) + \frac{1}{\eta^2} g(Q) \right) \dx\rho \dx\theta\, . $$
First, we show the following Lemma, which gives the upper bound in the case where there are no singularities near the axis $\rho=0$.

\begin{lemma}\label{lem:upper_bound_construction}
Let $\sigma>0$ and $F\subset\mathbb{S}^2$ be be a rotationally symmetric set of finite perimeter such that $\mathbb{S}^2\cap\{\rho\leq \sigma,z>0\},\mathbb{S}^2\cap\{\rho\leq \sigma,z<0\}$ are contained in one of the sets $F,F^c$. Then there exists a rotationally equivariant  sequence of functions $Q_\epsilon\in H^1(\Omega,\Sym)$ such that the compactness claim (\ref{thm:main:cptness}) holds, $\Vert Q_\epsilon\Vert_{L^\infty}\leq \sqrt{\frac{2}{3}}s_*$ and
$$ \limsup_{\epsilon\rightarrow 0} \eta\:\Eex(Q_\epsilon) \leq \E_0(F)\, . $$
\end{lemma}

\begin{proof}
The proof consists in providing an explicit definition for $Q_\epsilon$, generalizing the construction made in \cite{Alama2017}. The idea is the following: 
Let $F\subset \mathbb{S}^2\cap\{\rho\geq \sigma\}$ be rotationally symmetric. Since we assume $F$ to be of finite perimeter, $|D\chi_F|(\mathbb{S}^2\cap\{\rho\geq \sigma\})<\infty$. Let $\overline{F}\cap \overline{F^c}\cap\Omega_\sigma' = \{\theta_0,...,\theta_M\}$ for some $M\in\mathbb{N}$ and $\theta_{i}<\theta_{i+1}$ for all $i=0,...,M-1$.
We now define the map $Q_\epsilon$ on the two dimensional domain $\Omega'$. We divide $\Omega'$ into several regions and define $Q_\epsilon$ on each region separately (see Figure \ref{fig:upper_bound_construction}). After that, we derive the estimates that are needed to ensure that the rotated map $R_\varphi^\top Q_\epsilon R_\varphi$ satisfies the energy estimate. 

Let $\Omega'$ be parametrized by polar coordinates $(r,\theta)$. As usual, we denote by $F'=F\cap\Omega'$ and $F^c{'}=F^c\cap\Omega'$. Note that $\rho=r\sin\theta$.

\textit{Step 1 (Construction on $F_\eta'$ and $(F^c)_\eta'$):}  We define $F_\eta' = F'\setminus \bigcup_{i=0}^M B_{2\eta}(\theta_i)\subset\mathbb{S}^1\subset\Omega'$ and $(F^c)_\eta' = F^c{'}\setminus \bigcup_{i=0}^M B_{2\eta}(\theta_i)\subset\mathbb{S}^1\subset\Omega'$.
For $(r,\theta)\in [1,\infty)\times F_\eta'$ we define
\begin{equation} \label{upper_bound_F}
Q_\epsilon(r,\theta) \defi s_*\left( \nn\otimes\nn -  \frac{1}{3}\id\right) \quad\text{with}\quad \nn(r,\theta) = \begin{pmatrix}
\sqrt{1-\nn_3^2((r-1)/\eta,\theta)} \\ 0 \\ \nn_3((r-1)/\eta,\theta)
\end{pmatrix}\, ,
\end{equation} where $\nn_3$ is given by Lemma \ref{lem:radial_turning}.
Analogously, for $(r,\theta)\in [1,\infty)\times (F^c)_\eta$ we define
\begin{equation} \label{upper_bound_F^c}
Q_\epsilon(r,\theta) \defi s_*\left( \nn\otimes\nn -  \frac{1}{3}\id\right) \quad\text{with}\quad \nn(r,\theta) = \begin{pmatrix}
-\sqrt{1-\nn_3^2((r-1)/\eta,\pi-\theta)} \\ 0 \\ \nn_3((r-1)/\eta,\pi-\theta)
\end{pmatrix}\, .
\end{equation}
Since the defined $Q_\epsilon$ is uniaxial, we have $f(Q_\epsilon)=0$ and by Proposition \ref{prop:prop_g} we can estimate the energy on $\Omega_{F_\eta'}=\{(r,\theta)\sd \theta\in F_\eta'\}$
\begin{align*}
\eta\E^{2D}_\epsilon(Q_\epsilon,\Omega_{F_\eta'}) &= \eta \int_{F_\eta'} \int_1^\infty \rho\:\left( s_*^2|\partial_r \nn|^2 + \frac{s_*^2}{r^2}|\partial_\theta \nn|^2 + \frac{1}{\rho^2}Q_{2\times 2,\epsilon}:Q_\epsilon + \frac{1}{\eta^2} \sqrt{\frac{3}{2}}(1-\nn_3^2) \right) r \dx r\dx \theta \\
&= \int_{F_\eta'} \int_0^\infty \left( s_*^2|\partial_{t} \nn|^2  + \sqrt{\frac{3}{2}}(1-\nn_3^2) \right) (1+\eta t)^2 \sin\theta \dx t\dx \theta\\
&\:\:\:\: + \int_{F_\eta'} \int_0^\infty \frac{\eta^2s_*^2}{(1+\eta t)^2}\left[|\partial_\theta \nn|^2 + \frac{2}{\sin^2\theta} (1-\nn_3^2)\right](1+\eta t)^2 \sin\theta \dx t\dx \theta\, ,
\end{align*} where we set $r = 1+\eta t$ and used that $Q_{2\times 2,\epsilon}:Q = |Q_\epsilon|^2-6s_*(1-\nn_3^2)s_*\nn_3^2=2s_*^2(1-\nn_3^2)$.  Lemma \ref{lem:radial_turning} implies
\begin{equation} \label{upper_bound:F_Energy}
\eta\:\E_\epsilon^{2D}(Q_\epsilon,\Omega_{F_\eta'}) \leq \sqrt[4]{24}s_*\int_{F'} (1-\cos\theta)\sin\theta \dx \theta + C \:\eta\, .
\end{equation} Applying the same steps to $(F^c)_\eta'$, we get
\begin{equation} \label{upper_bound:F^c_Energy}
\eta\:\E_\epsilon^{2D}(Q_\epsilon,\Omega_{(F^c)_\eta'}) \leq \sqrt[4]{24}s_*\int_{F^c{'}} (1+\cos\theta)\sin\theta\dx \theta + C \:\eta\, .
\end{equation}

\textit{Step 2 (Construction on $(\Omega_{\theta_i,\eta}^{+}){'}$ and $(\Omega_{\theta_i,\eta}^{-}){'}$):} Next, we construct $Q_\epsilon$ for $(r,\theta)\in [1+4\eta)\times\bigcup_{i=0}^M B_{2\eta}(\theta_i)$. Without loss of generality, we assume $\theta\in B_{2\eta}(\theta_0)$ and that smaller angles belong to $F'$, while larger values lie in $F^c{'}$. We define $(\Omega_{\theta_0,\eta}^{+}){'}=\{(r,\theta)\sd \theta_0-2\eta\leq\theta\leq \theta_0\}$ and $(\Omega_{\theta_0,\eta}^{-}){'}=\{(r,\theta)\sd \theta_0\leq\theta\leq \theta_0+2\eta\}$.

Since we want $Q_\epsilon$ to have $H^1$-regularity, we need to respect the values of $Q_\epsilon$ that we already constructed at $\theta=\theta_0-2\eta$ and $\theta=\theta_0+2\eta$.  We do this by interpolating between these given values and $s_*(\ee_3\otimes\ee_3-\frac{1}{3}\id)$ at $\theta=\theta_0$. More precisely, for $(r,\theta)\in (\Omega_{\theta_0,\eta}^{+}){'}$ we define
$$ Q_\epsilon(r,\theta) = s_*\left( \nn\otimes\nn -  \frac{1}{3}\id\right) \quad\text{with}\quad \nn(r,\theta) = \begin{pmatrix}
\sin(\phi(r,\theta)) \\ 0 \\ \cos(\phi(r,\theta))
\end{pmatrix}\, , $$
where the phase $\phi$ is given by 
\begin{equation} \label{upper_bound_Ome_2_p}
\phi(r,\theta) = \frac{\theta_0-\theta}{2\eta} \arccos\left( \nn_3\left(r,\theta_0-2\eta\right) \right)\, .
\end{equation}
Similarly, the phase for $(r,\theta)\in (\Omega_{\theta_0,\eta}^{-}){'}$ is given by
\begin{equation}\label{upper_bound_Ome_2_m}
\phi(r,\theta) = -\frac{\theta-\theta_0}{2\eta} \arccos\left( \nn_3\left(r,\pi-(\theta_0+2\eta)\right) \right)\, .
\end{equation}
Note that $Q_\epsilon$ is indeed continuous for $\theta=\theta_0$ and that $Q_\epsilon$ coincides with our previous definition at $\theta=\theta_0-2\eta$ and $\theta=\theta_0+2\eta$. 

Now we calculate the energy coming from the two regions. We assume that $(r,\theta)\in (\Omega_{\theta_0,\eta}^{+}){'}$, the estimates for $(\Omega_{\theta_0,\eta}^{-}){'}$ are similar. Since $Q_\epsilon$ is uniaxial by construction, $f(Q_\epsilon)=0$ and furthermore by Proposition \ref{prop:prop_g}
\begin{align*}
g(Q_\epsilon) &= \sqrt{\frac{3}{2}} (1-\cos^2(\phi(r,\theta))) = \sqrt{\frac{3}{2}} \sin^2(\phi(r,\theta)) \leq \sqrt{\frac{3}{2}} \sin^2(\phi(r,\theta_0-2\eta))\, .
\end{align*}
For the gradient, we note that
\begin{align*}
\frac{1}{2}|\nabla' Q_\epsilon(r,\theta)|^2 &= s_*^2|\partial_r \nn(r,\theta)|^2 + \frac{s_*^2}{r^2} |\partial_\theta \nn(r,\theta)|^2 = s_*^2|\partial_r \phi(r,\theta)|^2 + \frac{s_*^2}{r^2} |\partial_\theta \phi(r,\theta)|^2 \\
&= \left(\frac{\theta-\theta_0}{2\eta}\right)^2 s_*^2|\partial_r \phi(r,\theta_0-2\eta)|^2 + \frac{s_*^2}{4r^2\eta^2}|\phi(r,\theta_0-2\eta)|^2  \\
&\leq s_*^2|\partial_r \nn(r,\theta_0-2\eta)|^2 + \frac{s_*^2}{4r^2\eta^2}|\phi(r,\theta_0-2\eta)|^2\, .
\end{align*}
Note, that for $\eta\rightarrow 0$ the phase $\phi$ stays bounded. Furthermore, all terms decrease exponentially in $r$ by Lemma \ref{lem:radial_turning} and are thus integrable.  Since $\frac{1}{2}|\partial_\varphi Q_\epsilon|^2 = Q_{2\times 2}:Q = 2 s_*^2\sin^2(\phi(r,\theta))$, this term converges to zero exponentially for $r\rightarrow\infty$ and is bounded for $\eta\rightarrow 0$. So finally we use the estimates and the usual change of variables $t=1+\eta t$ to get
\begin{equation} \label{upper_bound:Ome_2_Energy_upper}
\eta \:\E_\epsilon^{2D}(Q_\epsilon,(\Omega_{\theta_i,\eta}^{+}){'}) \leq C\:\eta\, .
\end{equation} Analogously,
\begin{equation} \label{upper_bound:Ome_2_Energy_lower}
\eta \:\E_\epsilon^{2D}(Q_\epsilon,(\Omega_{\theta_i,\eta}^{-}){'}) \leq C\:\eta\, ,
\end{equation} 
since $\phi(r,\theta_0+2\eta)\rightarrow 0$ as $r\rightarrow \infty$ exponentially.

\textit{Step 3 (Construction on $B'$ and $D'$):} 
Throughout this construction, we assume that we are in the same situation as in Step 2, namely that we are switching from $F'$ to $F^c{'}$ as the angle $\theta$ increases. In this situation, we are going to construct a defect of degree $-1/2$. Otherwise, one would need to define a defect of degree $1/2$, i.e.\ one needs to switch the sign of the angle in the definition of $Q(\alpha)$.
\begin{itemize}
\item We first define a map $Q_B$ on the two dimensional ball $B_1(0)$ using polar coordinates as follows
\begin{equation} \label{upper_bound:def_Q_B1}
Q_B(r,\alpha) = \begin{cases}
0  & r \in [0,\epsilon) \\
\left(\frac{r}{\epsilon}-1\right) Q(\alpha) & r \in [\epsilon, 2\epsilon) \\
Q(\alpha)  & r \in [2\epsilon, 1)\, ,  \end{cases} 
\end{equation}
where
$$ Q(\alpha) = s_*\left( \nn(\alpha)\otimes\nn(\alpha) -  \frac{1}{3}\id\right) \quad\text{with}\quad \nn(\alpha) = \begin{pmatrix}
\sin(\alpha/2) \\ 0 \\ \cos(\alpha/2)
\end{pmatrix}\, . $$
\item On $B_1\setminus B_{2\epsilon}$ we calculate
\begin{align*}
\int_{B_1\setminus B_{2\epsilon}} \frac{1}{2}|\nabla' Q_B|^2 \dx x &= \frac{1}{2}\int_0^{2\pi} \int_{2\epsilon}^1 \left(|\partial_r Q_B|^2 + \frac{1}{r^2}|\partial_\alpha Q_B|^2\right) r \dx\alpha\dx r \\
&= \frac{1}{2} \int_{2\epsilon}^1 \frac{1}{r} \dx r \int_0^{2\pi}|\partial_\alpha Q_B|^2 \dx\alpha \\
&= - \ln(2\epsilon) \int_0^{2\pi} s_*^2 \frac{1}{4}(\cos^2(\alpha/2)+\sin^2(\alpha/2)) \dx\alpha \\
&= \frac{\pi}{2} s_*^2 |\ln(\epsilon)| - \frac{\ln(2)\pi}{2}s_*^2\, .
\end{align*}
Furthermore, $f(Q_B)=0$ on $B_1\setminus B_{2\epsilon}$ and $ \int_{B_1\setminus B_{2\epsilon}} g(Q_B) \dx x \leq C |B_1\setminus B_{2\epsilon}| $. This implies
\begin{equation}\label{upper_bound:energy_B1_-_B2eps}
\int_{B_1\setminus B_{2\epsilon}} \frac{1}{2}|\nabla' Q_B|^2 + \frac{1}{\epsilon^2}f(Q_B) + \frac{1}{\eta^2}g(Q_B) \dx x \leq \frac{\pi}{2} s_*^2 |\ln(\epsilon)| + \frac{C_1}{\eta^2} |B_1\setminus B_{2\epsilon}|\, .
\end{equation}
\item On $B_{2\epsilon}\setminus B_\epsilon$ we find
\begin{align*}
\int_{B_{2\epsilon}\setminus B_{\epsilon}} \frac{1}{2}|\nabla' Q_B|^2 \dx x &= \frac{1}{2}\int_0^{2\pi} \int_{\epsilon}^{2\epsilon} \left(|\partial_r Q_B|^2 + \frac{1}{r^2}|\partial_\alpha Q_B|^2\right) r \dx\alpha\dx r \\
&= \frac{1}{2}\int_0^{2\pi} \int_{\epsilon}^{2\epsilon} \left(\frac{1}{\epsilon}-1\right)^2|Q(\alpha)|^2 r + \frac{1}{r}\left(\frac{r}{\epsilon}-1\right)^2|\partial_\alpha Q(\alpha)|^2 \dx r \dx\alpha \\
&= \frac{2}{3}\pi s_*^2 \left(\frac{1}{\epsilon}-1\right)^2 \int_{\epsilon}^{2\epsilon} r \dx r + \frac12\pi s_*^2 \int_{\epsilon}^{2\epsilon} \frac{1}{r}\left(\frac{r}{\epsilon}-1\right)^2 \dx r \\
&= \pi s_*^2 \left(\frac{1}{\epsilon}-1\right)^2 \epsilon^2 + \frac{\pi}{2}s_*^2\left(\ln(2) - \frac{1}{2}\right)  \\
&\leq C\, .
\end{align*}
In addition, $f(Q_B) =0 $ and $\int_{B_{2\epsilon}\setminus B_{\epsilon}} g(Q_B) \dx x \leq C |B_{2\epsilon}\setminus B_{\epsilon}| $. Together, we get
\begin{equation}\label{upper_bound:energy_B2eps_-_Beps}
\int_{B_{2\epsilon}\setminus B_{\epsilon}} \frac{1}{2}|\nabla' Q_B|^2 + \frac{1}{\epsilon^2}f(Q_B) + \frac{1}{\eta^2}g(Q_B) \dx x \leq C_2 \left(1+\frac{1}{\eta^2}\right)|B_{2\epsilon}\setminus B_{\epsilon}|\, .
\end{equation}
Finally,  the gradient of $Q_B$ on $B_\epsilon(0)$ is zero. The contributions from $f$ and $g$ are easily seen to be bounded by $C |B_\epsilon|$, so that
\begin{equation}\label{upper_bound:energy_Beps}
\int_{B_{\epsilon}} \frac{1}{2}|\nabla' Q_B|^2 + \frac{1}{\epsilon^2}f(Q_B) + \frac{1}{\eta^2}g(Q_B) \dx x \leq C_3\left(\frac{1}{\epsilon^2}+\frac{1}{\eta^2}\right)|B_\epsilon|\, .
\end{equation}
\end{itemize}
Combining (\ref{upper_bound:energy_B1_-_B2eps}), (\ref{upper_bound:energy_B2eps_-_Beps}) and (\ref{upper_bound:energy_Beps}) we get
\begin{equation} \label{upper_bound:energy_B}
\int_{B_1(0)} \frac{1}{2}|\nabla' Q_B|^2 + \frac{1}{\epsilon^2}f(Q_B) + \frac{1}{\eta^2}g(Q_B) \dx x \leq \frac{\pi}{2} s_*^2 |\ln(\epsilon)| + C\left(1+\frac{1}{\eta^2}\right)|B_1(0)| + C\, .
\end{equation}
Note that we have the same bound for $Q_{B_{\tilde{r}}}(r,\alpha) = Q_B(r/{\tilde{r}},\alpha)$ on $B_{\tilde{r}}(0)$, where $\tilde{r}\leq 1$. In addition, this bound is invariant under rotations and translations of the domain.
Again we assume that $\theta\in B_\eta(\theta_0)$.
We use the construction of $Q_B$ to define $Q_\epsilon$ on the set $B\defi B_\eta(1+2\eta,\theta_0)\subset [1,1+4\eta]\times [\theta_0-2\eta,\theta_0+2\eta]$ via
\begin{equation} \label{upper_bound:def_Q_B_eta}
Q_\epsilon(r,\theta) = R_{\theta_0} Q_B(\overline{r}/\eta,\alpha)\, ,
\end{equation} where $R_{\theta_0}$ is the rotation matrix around the $\rho-$axis with angle $\theta_0$, $\overline{r}^2 = (r-1-2\eta)^2+(\theta-\theta_0)^2$ and $\alpha$ being the angle between the vectors $(0,1)^\top$ and $(\theta_0-\theta,r-1-2\eta)^\top$. Note, that the term $|B_1(0)|$ in \eqref{upper_bound:energy_B} transforms to $|B|$, which can be estimated by $C\eta^2$. For the remaining term of $\E_\epsilon^{2D}$ we notice that $Q_{2\times 2,\epsilon}:Q_\epsilon$ is bounded on $B$ and that $\rho\geq \sigma-\eta$, thus $\int_B \rho^{-1}Q_{2\times 2,\epsilon}:Q_\epsilon \leq C (\sigma-\eta)^{-1}$. Then, using $\rho\leq (1+2\eta)\sin(\theta_0)+\eta$ we get from (\ref{upper_bound:energy_B}) that
\begin{equation} \label{upper_bound:energy_B_eta}
\eta\:\E_\epsilon^{2D}(Q_\epsilon,B) \leq ((1+2\eta)\sin(\theta_0)+\eta)\frac{\pi}{2} s_*^2 \eta|\ln(\epsilon)| + C\eta + \frac{C}{\sigma-\eta}\eta\, .
\end{equation}

We now want to construct the map $Q_\epsilon$ on the set $D=\{ (r,\theta)\in [1,1+4\eta]\times [\theta_0-2\eta,\theta_0+2\eta] \}\setminus B$ by interpolating between the values given by Steps 1 and 2 on the one hand, and the values on $\partial B$ on the other hand. We use the same polar coordinates  $(\overline{r},\alpha)$ as for the definition of $Q_\epsilon$ on $B$ to parametrize $D$. Let $\Phi_{\alpha/2}(\alpha)$ be the phase associated to the director of $Q_\epsilon(\eta,\alpha)$ and $\Phi(\alpha)$ the phase of the boundary values on $\partial (D\cup B)$. We set
$$ \phi(\overline{r},\alpha) = \frac{R(\alpha)-\overline{r}}{R(\alpha)-\eta}\Phi_{\alpha/2} + \frac{\overline{r}-\eta}{R(\alpha)-\eta}\Phi(\alpha)\, , $$
where 
$$ R(\alpha) = \begin{cases} \frac{2\eta}{|\cos(\alpha)|} & \text{if }\alpha\in [-\pi/4,\pi/4]\cup[3\pi/4,5\pi/4]\, , \\
\frac{2\eta}{|\sin(\alpha)|} & \text{otherwise}\, . \end{cases} $$
In particular, $|R(\alpha)|\leq 2\sqrt{2}\eta$ and $|\partial_\alpha R(\alpha)|\leq 2\sqrt{2}\eta$.
Then we define 
$$ Q_{D}(\overline{r},\alpha) = s_*\left( \nn(\overline{r},\alpha)\otimes\nn(\overline{r},\alpha) -  \frac{1}{3}\id\right) \quad\text{with}\quad \nn(\overline{r},\alpha) = \begin{pmatrix}
\sin(\phi(\overline{r},\alpha)) \\ 0 \\ \cos(\phi(\overline{r},\alpha))
\end{pmatrix}\, . $$

Then $f(Q_\epsilon|_D)=0$ since $Q_\epsilon|_D$ is uniaxial and $g(Q_\epsilon|_D)$ is bounded. We can estimate the gradient
\begin{align} \label{upper_bound:energy_D_first}
\begin{split}
\int_{D} \frac{1}{2}|\nabla' Q_\epsilon|^2 \dx x &= \int_{D} \frac{1}{2}\left(|\partial_r Q_\epsilon|^2 + \frac{1}{r^2}|\partial_\theta Q_\epsilon|^2\right) r \dx r \dx\theta \\
&\leq (1+4\eta) \int_{0}^{2\pi} \int_\eta^{R(\alpha)} \frac{1}{2}\left(|\partial_{\overline{r}} Q_\epsilon|^2 + \frac{1}{\overline{r}^2}|\partial_\alpha Q_\epsilon|^2\right) \overline{r} \dx \overline{r} \dx\alpha \\
&\leq (1+4\eta)s_*^2 \int_{0}^{2\pi} \int_\eta^{R(\alpha)} \left( |\partial_{\overline{r}} \phi|^2 + \frac{1}{\overline{r}^2}|\partial_\alpha\phi|^2 \right) \overline{r} \dx \overline{r} \dx\alpha\, .
\end{split}
\end{align} 

Since $\Phi_{\alpha/2}$ and $\Phi(\alpha)$ are bounded and $\partial_{\overline{r}}\phi = \frac{-1}{R(\alpha)-\eta}\Phi_{\alpha/2} + \frac{1}{R(\alpha)-\eta}\Phi(\alpha)$, we can easily infer that $|\partial_{\overline{r}}\phi|^2\leq\frac{C}{\eta^2}$. 
Furthermore it is clear by definition that $|\partial_\alpha\Phi_{\alpha/2}|^2\leq C$. So it remains to derive bounds on $\partial_\alpha \Phi(\alpha)$.
For $\alpha\in [0,\pi/4]$ we have $\Phi(\alpha) = \arccos(\nn_3(1+4\eta,\theta_0-2\eta)) \frac{\sqrt{R(\alpha)^2-4\eta^2}}{2\eta}$, i.e.\ $|\partial_\alpha \Phi(\alpha)|^2\leq C$. Similarly, $\partial_\alpha \Phi$ is bounded for $\alpha\in [-\pi/4,0]$.
For $\alpha\in [\pi/4,3\pi/4]$ and $r(\alpha) = 1+\sqrt{R^2(\alpha) + 8\eta^2 -4\sqrt{2}R(\alpha)\eta\cos(3\pi/4-\alpha)}$ one can show that $\Phi(\alpha) = \arccos(\nn_3(r(\alpha),\theta_0-2\eta))$. An explicit calculation yields $|\partial_\alpha \Phi(\alpha)|^2\leq C$. By the same argument, $\partial_\alpha \Phi$ is also bounded for $\alpha\in [-3\pi/4,-\pi/4]$
For $\alpha\in [3\pi/4,\pi]$ we have $\Phi(\alpha) = -2\eta\tan(\pi-\alpha)+\theta_0-\frac{\pi}{2}$, so that $|\partial_\alpha \Phi(\alpha)|^2$ is also bounded by a constant.
We plug this result into (\ref{upper_bound:energy_D_first}) and use the fact that $Q_{2\times 2,\epsilon}:Q_\epsilon$ is also bounded and $\sigma\leq 1+4\eta$ to get
\begin{equation} \label{upper_bound:energy_D_last}
\E_\epsilon^{2D}(Q_\epsilon,D) \leq 2(1+4\eta)s_*^2 \int_{0}^{2\pi} \int_\eta^{R(\alpha)} \left( C + \frac{C}{\sigma^2} \right) \sigma \dx \sigma \dx\alpha + \frac{C}{\sigma-c\eta} \leq C + \frac{C}{\sigma-c\eta}\, .
\end{equation}
Hence by (\ref{upper_bound:energy_B_eta}) and (\ref{upper_bound:energy_D_last})
\begin{equation} \label{upper_bound:energy_B_D}
\eta\:\E_\epsilon^{2D}(Q_\epsilon,B\cup D) \leq ((1+2\eta)\sin(\theta_0)+2\eta)\frac{\pi}{2}s_*^2\eta|\ln\epsilon| +  C \eta + \frac{C}{\sigma-C\eta}\eta\, .
\end{equation}
This finishes our construction of $Q_\epsilon(\rho,\theta)$. If we now extend $Q_\epsilon$ to $\Omega$ by using the rotated function $Q_\epsilon(\rho,\varphi,\theta) = R_\varphi^\top Q_\epsilon(\rho,\theta) R_\varphi$ and integrate $\E_\epsilon^{2D}$ in $\varphi$-direction, we get from (\ref{upper_bound:F_Energy}), (\ref{upper_bound:F^c_Energy}), (\ref{upper_bound:Ome_2_Energy_upper}), (\ref{upper_bound:Ome_2_Energy_lower}) and (\ref{upper_bound:energy_B_D})
\begin{equation}\label{upper_bound:energy_total_3d}
\begin{split}
\eta\E_\epsilon(Q_\epsilon,\Omega) &\leq \sqrt[4]{24}s_*\int_0^{2\pi}\int_{F'} (1-\cos(\theta))\sin(\theta) \dx\theta\dx\varphi + \sqrt[4]{24}s_*\int_0^{2\pi}\int_{F^c{'}} (1+\cos(\theta))\sin(\theta) \dx\theta\dx\varphi\\
&+ \frac{\pi}{2}s_*^2\eta|\ln\epsilon|\sum_{i=0}^{M-1}\int_0^{2\pi}((1+2\eta)\sin(\theta_i)+2\eta)\dx\varphi + C\eta + \frac{C\eta}{\sigma-c\eta}\, .
\end{split}
\end{equation}
Taking the limsup $\eta,\epsilon\rightarrow 0$ in (\ref{upper_bound:energy_total_3d}) yields the inequality 
\begin{align*}
\limsup_{\eta,\epsilon\rightarrow 0} \Eex(Q_\epsilon) &\leq \sqrt[4]{24}s_*\int_{F} (1-\cos(\theta)) \dx\omega + \sqrt[4]{24}s_*\int_{F^c} (1+\cos(\theta)) \dx\omega + \frac{\pi}{2}s_*^2\beta |D\chi_F|(\mathbb{S}^2) \\
&= \E_0(F)\, .
\end{align*}

It remains to show the claimed convergence. It is clear by definition of $Q_\epsilon$ that $\bigcup_{\eta>0} F_\eta = F$ and $\bigcup_{\eta>0} (F^c)_\eta = F^c$ which implies the convergence for $\chi_F$. The continuity of $\nn^\epsilon$ as a function with values in $\mathbb{S}^2$ outside a set $\omega_\eta$ is clear by construction if we choose $\omega_\eta$ to contain all balls $B$, we used in step 3. Taking $\omega_\eta$ as the union of all sets $B$ and $D$ from step 3. we can also achieve that $\Omega\setminus\omega_\eta$ is simply connected. Extending $\nn^\epsilon$ inside $B$ measurably, yields the compactness claim.
\end{proof}

\begin{proof}[Proof of the upper bound (\ref{eq:upper_bound}) of Theorem \ref{thm:main}] We choose a sequence $\sigma_k>0$ which converges to zero as $k\rightarrow\infty$. We approximate the set $F$ by sets $F_k$ such that the domains $\mathbb{S}^2\cap\{\rho\leq \sigma_k,z> 0\}$ and $\mathbb{S}^2\cap\{\rho\leq \sigma_k,z< 0\}$ are fully contained in $F_k$ or $F_k^c$. By Lemma \ref{lem:upper_bound_construction} there exist sequences $Q_{\epsilon,k}$ such that $\limsup_{\eta,\epsilon\rightarrow 0} \Eex(Q_{\epsilon,k})\leq \E_0(F_k)$ and (\ref{thm:main:cptness}) holds. We observe that
\begin{align*}
|D\chi_{F_k}|(\mathbb{S}^2) = |D\chi_{F_k}|(\mathbb{S}^2\cap\{\rho\geq \sigma_k\}) = |D\chi_{F}|(\mathbb{S}^2\cap\{\rho\geq \sigma_k\})
\end{align*}
and
\begin{align*}
\left| \int_{F} \big(1-\cos(\theta)\big) \dx\omega - \int_{F_k} \big(1-\cos(\theta)\big) \dx\omega \right|,\left| \int_{F^c} \big(1+\cos(\theta)\big) \dx\omega - \int_{F_k^c} \big(1+\cos(\theta)\big) \dx\omega \right| \leq C \sigma_k^2\, .
\end{align*}
Hence $\limsup_{\eta,\epsilon\rightarrow 0} \Eex(Q_{\epsilon,k})\leq \E_0(F_k) \leq \E_0(F) + C \sigma_k^2$ and taking a diagonal sequence $Q_\epsilon=Q_{\epsilon,k(\epsilon)}$ we get
$$ \limsup_{\eta,\epsilon\rightarrow 0} \Eex(Q_\epsilon)\leq \E_0(F)\, . $$
The compactness (\ref{thm:main:cptness}) follows by triangle inequality.
\end{proof}

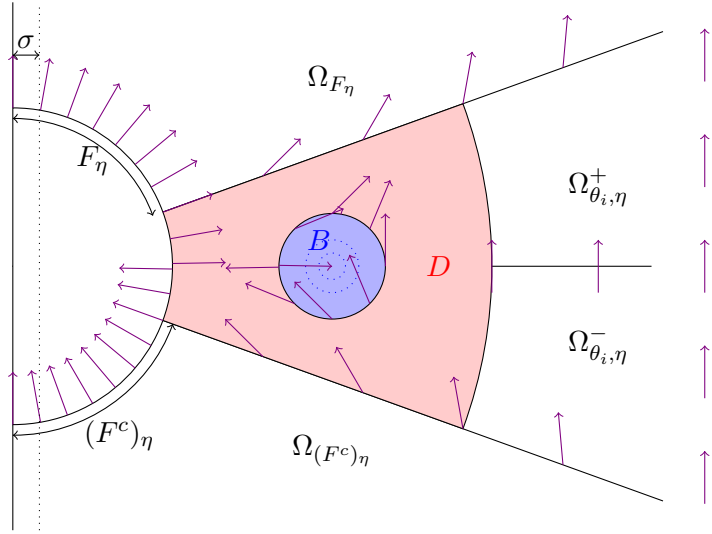
\begin{figure}[H]
\begin{center}
\begin{tikzpicture}[scale=0.7]
\draw[] (0,0) -- (0:12cm);
\draw[] (0,0) -- (-20:13cm);
\draw[] (0,0) -- (20:13cm);
\fill[red!20!white] (0,0) -- (-20:9cm) arc (-20:20:9cm) -- (0,0);
\draw[] (0,0) -- (-20:9cm) arc (-20:20:9cm) -- (0,0);

\draw[fill=blue!30!white] (6,0) circle (1cm);
\draw[blue,dotted] (6,0) circle (0.25cm);
\draw[blue,dotted] (6,0) circle (0.5cm);

\draw[black] (0,3)--(0,5);
\draw[black] (0,-3)--(0,-5);
\fill[white] (0,0) -- (90:3cm) arc (90:-90:3cm) -- cycle;
\draw[black] (0,0) -- (90:3cm) arc (90:-90:3cm) -- cycle;

\draw[violet, ->]  (13,3.5) -- (13,4.5);
\draw[violet, ->]  (13,1.5) -- (13,2.5);
\draw[violet, ->]  (13,-0.5) -- (13,0.5);
\draw[violet, ->]  (13,-2.5) -- (13,-1.5);
\draw[violet, ->]  (13,-4.5) -- (13,-3.5);
\draw[violet, ->]  (11,-0.5) -- (11,0.5);
\draw[violet, ->]  (9,-0.5) -- (9,0.5);
\draw[violet,->] (20:5cm) -- +(45:1);
\draw[violet,->] (20:7cm) -- +(60:1);
\draw[violet,->] (20:9cm) -- +(80:1);
\draw[violet,->] (20:11cm) -- +(85:1);
\draw[violet,->] (-20:5cm) -- +(135:1);
\draw[violet,->] (-20:7cm) -- +(120:1);
\draw[violet,->] (-20:9cm) -- +(100:1);
\draw[violet,->] (-20:11cm) -- +(95:1);
\draw[violet, rotate=-90, ->]  (3,0) -- (2,0);
\draw[violet, rotate=-80, ->]  (3,0) -- (2,0);
\draw[violet, rotate=-70, ->]  (3,0) -- (2,0);
\draw[violet, rotate=-60, ->]  (3,0) -- (2,0);
\draw[violet, rotate=-50, ->]  (3,0) -- (2,0);
\draw[violet, rotate=-40, ->]  (3,0) -- (2,0);
\draw[violet, rotate=-30, ->]  (3,0) -- (2,0);
\draw[violet, rotate=-20, ->]  (3,0) -- (2,0);
\draw[violet, rotate=-10, ->]  (3,0) -- (2,0);
\draw[violet, rotate=-1, ->]  (3,0) -- (2,0);
\draw[violet, rotate=90, ->]  (3,0) -- (4,0);
\draw[violet, rotate=80, ->]  (3,0) -- (4,0);
\draw[violet, rotate=70, ->]  (3,0) -- (4,0);
\draw[violet, rotate=60, ->]  (3,0) -- (4,0);
\draw[violet, rotate=50, ->]  (3,0) -- (4,0);
\draw[violet, rotate=40, ->]  (3,0) -- (4,0);
\draw[violet, rotate=30, ->]  (3,0) -- (4,0);
\draw[violet, rotate=20, ->]  (3,0) -- (4,0);
\draw[violet, rotate=10, ->]  (3,0) -- (4,0);
\draw[violet, rotate=1, ->]  (3,0) -- (4,0);
\draw[violet, rotate around={179:(6,0)}, ->]  (7,0) -- +(-180:1);
\draw[violet, rotate around={135:(6,0)}, ->]  (7,0) -- +(-112.5:1);
\draw[violet, rotate around={90:(6,0)}, ->]  (7,0) -- +(-45:1);
\draw[violet, rotate around={45:(6,0)}, ->]  (7,0) -- +(22.5:1);
\draw[violet, rotate around={0:(6,0)}, ->]  (7,0) -- +(90:1);
\draw[violet, rotate around={-45:(6,0)}, ->]  (7,0) -- +(157.5:1);
\draw[violet, rotate around={-90:(6,0)}, ->]  (7,0) -- +(225:1);
\draw[violet, rotate around={-135:(6,0)}, ->]  (7,0) -- +(292.5:1);
\draw[violet, rotate around={-179:(6,0)}, ->]  (7,0) -- +(360:1);

\node[blue] at (5.75,0.5) {$B$};
\node[red] at (8,0) {$D$};

\node[] at (6,3.5) {$\Omega_{F_\eta}$};
\node[] at (6,-3.5) {$\Omega_{(F^c)_\eta}$};
\draw[<->] (0,2.8) arc (90:20:2.8cm);
\node[] at (1.5,2) {$F_\eta$};
\draw[<->] (0,-3.2) arc (-90:-20:3.2cm);
\node[] at (2,-3.2) {$(F^c)_\eta$};

\node[] at (11,1.5) {$\Omega_{\theta_i,\eta}^{+}$};
\node[] at (11,-1.5) {$\Omega_{\theta_i,\eta}^{-}$};

\draw[black,dotted] (0.5,-5)--(0.5,5);
\draw[black, <->]  (0,4) -- (0.5,4);
\node at (0.25,4.25) {$\sigma$};
\end{tikzpicture} 
\end{center}
\caption{Partition of $\Omega'$ into regions for the construction of $Q_\epsilon$ (arrows show $\nn^\epsilon$)}
\label{fig:upper_bound_construction}
\end{figure}

\section{Limit problem, transition and hysteresis}

This last section is devoted to the study of the limit functional. In particular we are going to study the minimizing configurations for different values of $\beta$. 

In a first step, we claim that if $F$ is a minimizer of $\E_0$, then $F$ and $F^c$ are connected. Indeed, assume that one of the two sets, say $F$, is not connected. Then there are two possibilities: If $F^c$ is connected, then $F$ also contains the point $\theta=\pi$ and we can decrease the energy $\E_0$ by handing over this set to $F^c$. If $F^c$ is also not connected, then we can similarly exchange points between $F$ and $F^c$ while decreasing the energy until both sets are connected.

Now that we know that $F$ and $F^c$ are connected, we deduce that there can only be one angle under which the defect occurs. Let us name this angle $\theta_d\in [0,\pi]$ and let $F\subset \mathbb{S}^2$ be the set corresponding to $0\leq \theta\leq \theta_d$. Then we can express the limit energy as
\begin{align*}
\E_0(F) &= \sqrt[4]{24}s_*\int_F (1 - \cos(\theta)) \dx\omega + \sqrt[4]{24}s_*\int_{F^c} (1 + \cos(\theta)) \dx\omega + \frac{\pi}{2}s_*^2\beta |D\chi_F|(\mathbb{S}^2) \\
&= \sqrt[4]{24}s_*\int_0^{2\pi} \int_0^{\theta_d} (1 - \cos(\theta))\sin(\theta) \dx\theta \dx\varphi + \sqrt[4]{24}s_*\int_0^{2\pi} \int_{\theta_d}^\pi (1 + \cos(\theta))\sin(\theta) \dx\theta \dx\varphi \\
&\:\:\:\:\:+ \frac{\pi}{2}s_*^2\beta (2\pi\sin(\theta_d)) \\
&= 4\sqrt[4]{24}\pi s_* \Big( \sin^4(\theta_d/2) + \cos^4(\theta_d/2)\Big) + \pi^2 \beta s_*^2 \sin(\theta_d)\, .
\end{align*}
Setting the derivative of this expression to zero  gives the equation
\begin{align*}
\pi s_* \cos(\theta_d)\Big( \pi \beta s_* - 4\sqrt[4]{24} \sin(\theta_d)\Big) = 0\, ,
\end{align*} which yields the two families of solutions $\theta_1 = \pi/2 + \pi\mathbb{Z}$ and $\theta_2 = \arcsin(\frac{\pi\beta s_*}{4\sqrt[4]{24}}) + 2\pi\mathbb{Z}$. We note:
\begin{enumerate}
\item For $\beta s_* = \frac{4\sqrt[4]{24}}{\pi}\approx 2.818$, the two families are equal. We conclude that for $\beta s_* \geq \frac{4\sqrt[4]{24}}{\pi}$ the only stable configuration is a dipole at $\theta_d=0,\pi$ (see Figure \ref{fig:energy_landscape}).
\item The energy of the saturn ring $\theta_d=\pi/2$ and the dipole $\theta_d=0$ are equal for $\beta s_* = \frac{2\sqrt[4]{24}}{\pi} \approx 1.409$, which means for greater values of $\beta s_*$ the dipole is the globally energy minimizing configuration, while for smaller values the saturn ring is optimal.
\item The case where $\theta_d=\pi/2$ is the only (local) minimizer corresponds to $\beta s_*=0$, i.e.\ $\theta_2=0$.
\end{enumerate}

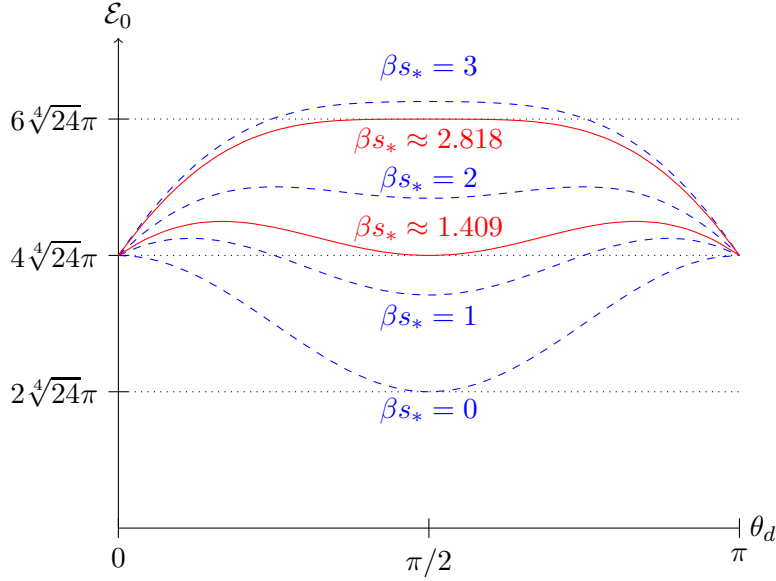
\begin{figure}[H]
\begin{center}
\begin{tikzpicture}[scale=1.3]
\pgfmathsetmacro{\a}{0.1} 
\draw[->] (0,0) -- (0,5) node[above] {$\E_0$};
\draw[] (0.1,3*\a*13.907) -- (-0.1,3*\a*13.907) node[left] {$6\sqrt[4]{24}\pi$};
\draw[dotted] (0.1,3*\a*13.907) -- (2*3.1415,3*\a*13.907);
\draw[] (0.1,2*\a*13.907) -- (-0.1,2*\a*13.907) node[left] {$4\sqrt[4]{24}\pi$};
\draw[dotted] (0.1,2*\a*13.907) -- (2*3.1415,2*\a*13.907);
\draw[] (0.1,\a*13.907) -- (-0.1,\a*13.907) node[left] {$2\sqrt[4]{24}\pi$};
\draw[dotted] (0.1,\a*13.907) -- (2*3.1415,\a*13.907);
\draw[-] (0,0) -- (2*3.1415,0) node[right] {$\theta_d$};
\draw[] (0,0.1) -- (0,-0.1) node[below] {$0$};
\draw[] (3.1415,0.1) -- (3.1415,-0.1) node[below] {$\pi/2$};
\draw[] (2*3.1415,0.1) -- (2*3.1415,-0.1) node[below] {$\pi$};
\foreach \b in {0,1,2,3}{
\draw[domain=0:2*3.1415,smooth,variable=\x,blue,dashed] plot ({\x},{  \a*4*pi*pow(24,1./4)*pow(sin(deg(\x/4)),4) + \a*4*pi*pow(24,1./4)*pow(cos(deg(\x/4)),4) + \a*pi*pi*\b*sin(deg(\x/2))  });
} 
\pgfmathsetmacro{\b}{2.818} 
\draw[domain=0:2*3.1415,smooth,variable=\x,red] plot ({\x},{  \a*4*pi*pow(24,1./4)*pow(sin(deg(\x/4)),4) + \a*4*pi*pow(24,1./4)*pow(cos(deg(\x/4)),4) + \a*pi*pi*\b*sin(deg(\x/2))  });
\pgfmathsetmacro{\b}{1.409} 
\draw[domain=0:2*3.1415,smooth,variable=\x,red] plot ({\x},{  \a*4*pi*pow(24,1./4)*pow(sin(deg(\x/4)),4) + \a*4*pi*pow(24,1./4)*pow(cos(deg(\x/4)),4) + \a*pi*pi*\b*sin(deg(\x/2))  });
\node[blue] at (3.14,4.7) {$\beta s_*=3$};
\node[red] at (3.14,3.95) {$\beta s_*\approx 2.818$};
\node[blue] at (3.14,3.55) {$\beta s_*=2$};
\node[red] at (3.14,3.05) {$\beta s_*\approx 1.409$};
\node[blue] at (3.14,2.15) {$\beta s_*=1$};
\node[blue] at (3.14,1.2) {$\beta s_*=0$};
\end{tikzpicture}
\end{center}
\caption{Plot of the energy $\E_0$ for different values of $\beta s_*$ as a function of the angle $\theta_d$}
\label{fig:energy_landscape}
\end{figure}

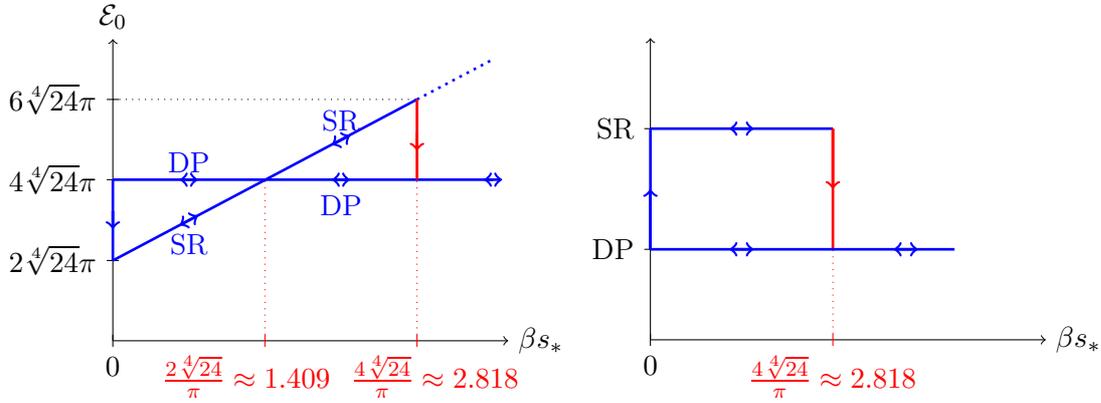
\begin{figure}[H]
\begin{tikzpicture}[scale=0.8]
\draw[->] (0,0) -- (0,5) node[above] {$\E_0$};
\draw[] (0.1,1.33) -- (-0.1,1.33) node[left] {$2\sqrt[4]{24}\pi$};
\draw[] (0.1,2.67) -- (-0.1,2.67) node[left] {$4\sqrt[4]{24}\pi$};
\draw[] (0.1,4) -- (-0.1,4) node[left] {$6\sqrt[4]{24}\pi$};
\draw[dotted] (0,4) -- (5,4);
\draw[->] (0,0) -- (6.5,0) node[right] {$\beta s_*$};
\draw[] (0,0.1) -- (0,-0.1) node[below] {$0$};
\draw[red] (2.5,0.1) -- (2.5,-0.1);
\node[red] at (2.2,-0.6) {$\frac{2\sqrt[4]{24}}{\pi}\approx 1.409$};
\draw[red,dotted] (2.5,0.1) -- (2.5,2.67);
\draw[red] (5,0.1) -- (5,-0.1);
\node[red] at (5.3,-0.6) {$\frac{4\sqrt[4]{24}}{\pi}\approx 2.818$};
\draw[red,dotted] (5,0.1) -- (5,2.67);

\draw[-,color=blue,line width=1] (0,2.67) -- (6.25,2.67) node[above] {};
\draw[<->,color=blue,line width=1] (1.1,2.67) -- (1.4,2.67) node[above] {};
\draw[<->,color=blue,line width=1] (3.6,2.67) -- (3.9,2.67) node[above] {};
\draw[<->,color=blue,line width=1] (6.1,2.67) -- (6.4,2.67) node[above] {};
\node[blue] at (1.25,2.95) {DP};
\node[blue] at (3.75,2.25) {DP};

\draw[-,color=blue,line width=1] (0,1.33) -- (5,4) node[above] {};
\draw[-,color=blue,line width=1,dotted] (5,4) -- (6.25,4.665) node[above] {};
\draw[<->,color=blue,line width=1] (1.1,1.9152) -- (1.4,2.0748) node[above] {};
\draw[<->,color=blue,line width=1] (3.6,3.2452) -- (3.9,3.4048) node[above] {};
\node[blue] at (1.25,1.6) {SR};
\node[blue] at (3.75,3.65) {SR};

\draw[-,color=red,line width=1] (5,2.67) -- (5,4) node[above] {};
\draw[<-,color=red,line width=1] (5,3.15) -- (5,3.45) node[above] {};

\draw[-,color=blue,line width=1] (0,2.67) -- (0,1.33) node[above] {};
\draw[<-,color=blue,line width=1] (0,1.85) -- (0,2.15) node[above] {};
\end{tikzpicture}
\begin{tikzpicture}[scale=0.8]
\draw[->] (0,0) -- (0,5) node[above] {};
\draw[] (0.1,1.5) -- (-0.1,1.5) node[left] {DP};
\draw[] (0.1,3.5) -- (-0.1,3.5) node[left] {SR};
\draw[->] (0,0) -- (6.5,0) node[right] {$\beta s_*$};
\draw[] (0,0.1) -- (0,-0.1) node[below] {$0$};
\draw[red] (3,0.1) -- (3,-0.1) node[below] {$\frac{4\sqrt[4]{24}}{\pi}\approx 2.818$};
\draw[red,dotted] (3,0.1) -- (3,1.5);

\draw[-,color=blue,line width=1] (0,3.5) -- (3,3.5) node[above] {};
\draw[<->,color=blue,line width=1] (1.3,3.5) -- (1.7,3.5) node[above] {};

\draw[-,color=red,line width=1] (3,3.5) -- (3,1.5) node[above] {};
\draw[->,color=red,line width=1] (3,3.5) -- (3,2.5) node[above] {};

\draw[-,color=blue,line width=1] (5,1.5) -- (0,1.5) node[above] {};
\draw[<->,color=blue,line width=1] (4,1.5) -- (4.4,1.5) node[above] {};
\draw[<->,color=blue,line width=1] (1.7,1.5) -- (1.3,1.5) node[above] {};

\draw[-,color=blue,line width=1] (0,1.5) -- (0,3.5) node[above] {};
\draw[->,color=blue,line width=1] (0,1.5) -- (0,2.5) node[above] {};
\end{tikzpicture}
\caption{Left: Plot of the energy of the dipole and saturn ring as a function of $\beta s_*$. Right: Hysteresis induced by changing $\beta s_*$}
\label{fig:hysteresis}
\end{figure}

Our analysis confirms the numerical simulations by H. Stark \cite{Stark1999} (see also \cite{Lavrentovich2001} for similar problems) as well as the physical observation, e.g.\ \cite[p.190ff]{Antonietti2003}. The reduced magnetic coherence length $\xi_H$ introduced in \cite{Stark1999} corresponds to our parameter $\eta$ in the one constant approximation. The assumption of high magnetic fields $\xi_H\ll 1$ translates to our limit $\eta\rightarrow 0$. Although the calculations in \cite{Stark1999} are based on the Oseen-Frank model rather than the Landau-de Gennes that we are using, we are able to reproduce the behaviour of the energy $\E_0$ as a function of $\theta_d$, compare Figure \ref{fig:energy_landscape} and  \cite[Fig. 11]{Stark1999}. From our calculation, we also find the hysteresis for changing values of $\beta s_*$. For $\beta\gg 1$, i.e.\ small external fields, the dipole is the only stable configuration. Increasing the field, the system will maintain the dipole, until we reach $\beta=0$, where a transition to the saturn ring takes place. Decreasing the field while starting from a saturn ring, we will retain the structure until we reach $\beta s_*\approx 2.818$ and the saturn ring closes to a dipole.

\paragraph{Acknowledgment.} DS thanks Xavier Lamy for the useful discussions at several occasions.

\addcontentsline{toc}{section}{References}
\bibliographystyle{abbrv}
\bibliography{LC_SR_rotsym}

\end{document}